\documentclass[12pt,a4paper]{article}

\usepackage{amsfonts,amsmath,amsthm,amssymb,amsopn}
\usepackage{graphicx}
\usepackage{epstopdf}
\usepackage{algorithmic}
\usepackage{enumitem}

\usepackage{bm,color,xcolor}
\usepackage{stmaryrd}
\usepackage[T1]{fontenc}
\usepackage{float, caption, subcaption}
\usepackage{booktabs}
\usepackage{multirow}
\usepackage{siunitx}
\usepackage{pifont}
\usepackage{hyperref}
\usepackage{cleveref}
\newcommand{\cmark}{\ding{51}}%
\newcommand{\xmark}{\ding{55}}%

\newcommand\bbR{\mathbb{R}}

\newcommand\bF{\bm{F}}

\newcommand\bU{\bm{U}}

\newcommand\bR{\bm{R}}

\newcommand\bJ{\bm{J}}
\newcommand\bD{\bm{D}}

\newcommand\dd{\mathrm{d}}
\newcommand\pd[2]{\frac{\partial {#1}}{\partial {#2}}}

\newcommand\abs[1]{\lvert #1 \rvert}

\DeclareMathOperator{\diag}{diag}

\newcommand\xr{i+\frac12}
\newcommand\xl{i-\frac12}

\theoremstyle{plain}
\newtheorem{lemma}{Lemma}[section]

\theoremstyle{definition}
\newtheorem{definition}{Definition}[section]

\newtheorem{theorem}{Theorem}[section]
\newtheorem{example}{Example}[section]
\newtheorem{proposition}{Proposition}[section]

\theoremstyle{remark}
\newtheorem{remark}{Remark}[section]

\crefname{equation}{}{}
\crefname{figure}{Figure}{Figures}
\crefname{table}{Table}{Tables}
\crefname{example}{Example}{Examples}
\crefname{section}{Section}{Sections}

\renewcommand{\title}{Active flux methods for hyperbolic conservation laws -- flux vector splitting and bound-preservation: One-dimensional case}

\newcommand{\authorOne}{Junming Duan\footnote{Corresponding author. Institute of Mathematics, University of W\"urzburg, Emil-Fischer-Stra\ss e 40, 97074 W\"urzburg, Germany, junming.duan@uni-wuerzburg.de}}
\newcommand{\authorTwo}{Wasilij Barsukow\footnote{Institut de Math\'ematiques de Bordeaux (IMB), CNRS UMR 5251, University of Bordeaux, 33405 Talence, France, wasilij.barsukow@math.u-bordeaux.fr}}
\newcommand{\authorThree}{Christian Klingenberg\footnote{Institute of Mathematics, University of W\"urzburg, Emil-Fischer-Stra\ss e 40, 97074 W\"urzburg, Germany, christian.klingenberg@uni-wuerzburg.de}}

\begin{document}

\begin{center} \Large
\title

\vspace{1cm}

\date{}
\normalsize

\authorOne, \authorTwo, \authorThree
\end{center}

\begin{abstract}

{\color{blue}A more elaborate version, based on this preprint and arXiv:2407.13380 can be found as arXiv:2411.00065.}

The active flux (AF) method is a compact high-order finite volume method that evolves cell averages and point values at cell interfaces simultaneously.
Within the method of lines framework, the point value can be updated based on Jacobian splitting (JS), incorporating the upwind idea.
However, such JS-based AF methods encounter transonic issues for nonlinear problems due to inaccurate upwind direction estimation.
This paper proposes to use flux vector splitting for the point value update,
offering a natural and uniform remedy to the transonic issue.
To improve robustness, this paper also develops bound-preserving (BP) AF methods for one-dimensional hyperbolic conservation laws.
Two cases are considered:
preservation of the maximum principle for the scalar case, and preservation of positive density and pressure for the compressible Euler equations.
The update of the cell average in high-order AF methods is rewritten as a convex combination of using the original high-order fluxes and robust low-order (local Lax-Friedrichs or Rusanov) fluxes, and the desired bounds are enforced by choosing the right amount of low-order fluxes.
A similar blending strategy is used for the point value update.
Several challenging benchmark tests are conducted to verify the accuracy, BP properties, and shock-capturing ability of the methods.



\end{abstract}

\section{Introduction}\label{sec:introduction}
This paper is concerned with solving systems of hyperbolic conservation laws
\begin{equation}\label{eq:1d_hcl}
		\pd{\bU(x, t)}{t} + \pd{\bF(\bU)}{x} = 0,\quad
		\bU(x,0) = \bU_0(x),\quad (x, t)\in\bbR\times \bbR^+,
\end{equation}
where $\bU\in \bbR^{m}$ is the vector of $m$ conservative variables,
$\bF\in \bbR^{m}$ is the physical flux,
and $\bU_0(x)$ is assumed to be initial data of bounded variation.
In this paper, we would like to consider two cases.
The first is a scalar conservation law ($m=1$)
\begin{equation}\label{eq:1d_scalar}
	\pd{u}{t} + \pd{f(u)}{x} = 0,\quad u(x,0) = u_0(x).
\end{equation}
The second case is that of compressible Euler equations of gas dynamics
with $\bU=(\rho, \rho v, E)^\top$ and $\bF=(\rho v, \rho v^2 + p, (E+p)v)^\top$, i.e.,
\begin{equation}\label{eq:1d_euler}
	\begin{aligned}
		\dfrac{\partial}{\partial t}\begin{pmatrix}
			\rho\\ \rho v\\ E \\
		\end{pmatrix}
		+ \dfrac{\partial}{\partial x}\begin{pmatrix}
			\rho v\\ \rho v^2 + p\\ (E + p)v \\
		\end{pmatrix}
		 = \bm{0},
	\end{aligned}\quad
	(\rho, v, p)(x, 0) = (\rho_0, v_0, p_0).
\end{equation}
Here $\rho$ denotes the density, $v$ the velocity,
$p$ the pressure, and $E=\frac12\rho v^2 + \rho e$ the total energy with $e$ the specific internal energy.
The system \cref{eq:1d_euler} should be closed by an equation of state (EOS).
This paper considers the perfect gas EOS, $p = (\gamma-1)\rho e$,
with the adiabatic index $\gamma > 1$.
Note that this paper uses bold symbols to refer to vectors and matrices, such that they are easier to distinguish from scalars.

The active flux (AF) method is a new finite volume method \cite{Eymann_2011_Active_InCollection, Eymann_2011_Active_InProceedings,Eymann_2013_Multidimensional_InCollection,Roe_2017_Is_JoSC},
that Roe took inspiration by \cite{VanLeer_1977_Towards_JoCP}.
Apart from cell averages, it incorporates additional degrees of freedom as point values
located at the cell interfaces, evolved independently from the cell average.
The original AF method gives a global continuous representation of the numerical solution using a piecewise quadratic reconstruction,
leading naturally to a third-order accurate method with a compact stencil.
The introduction of point values at the cell interfaces avoids the usage of Riemann solvers as in usual Godunov methods, because the numerical solution is continuous across the cell interface and the numerical flux for the cell average update is available directly.

The independence of the point value update adds flexibility to the AF methods.
Based on the evolution of the point value, there are generally two kinds of AF methods.
The original one uses exact or approximate evolution operators and Simpson's rule for flux quadrature in time, i.e. it does not require time integration methods like Runge-Kutta methods.
Exact evolution operators have been studied for linear equations in \cite{Barsukow_2019_Active_JoSC,Fan_2015_Investigations_InCollection,Eymann_2013_Multidimensional_InCollection, VanLeer_1977_Towards_JoCP}.
Approximate evolution operators have been explored for Burgers' equation \cite{Eymann_2011_Active_InCollection,Eymann_2011_Active_InProceedings,Roe_2017_Is_JoSC,Barsukow_2021_active_JoSC},
the compressible Euler equations in one spatial dimension \cite{Eymann_2011_Active_InCollection,Helzel_2019_New_JoSC,Barsukow_2021_active_JoSC},
and hyperbolic balance laws \cite{Barsukow_2021_Active_SJoSC, Barsukow_2023_Well_CoAMaC}, etc.
One of the advantages of the AF method over standard finite volume methods is its structure-preserving property.
For instance, it preserves the vorticity and stationary states for multi-dimensional acoustic equations \cite{Barsukow_2019_Active_JoSC}, and it is naturally well-balanced for acoustics with gravity \cite{Barsukow_2021_Active_SJoSC}.

Since it may not be convenient to derive exact or approximate evolution operators for nonlinear systems, especially in multi-dimensions,
another kind of generalized AF method was presented in \cite{Abgrall_2023_Combination_CoAMaC, Abgrall_2023_Extensions_EMMaNA, Abgrall_2023_active}. A method of lines was used,
where the cell average and point value updates are written in semi-discrete form and advanced in time with time integration methods. 
In the point values update, the Jacobian matrix is split based on the sign of the eigenvalues (Jacobian splitting (JS)), and upwind-biased stencils are used to compute the approximation of derivatives.
There are some deficiencies of the JS when used for the AF methods, e.g., the transonic issue \cite{Helzel_2019_New_JoSC} for nonlinear problems,
leading to spikes in the cell average.
Some remedies are suggested in the literature,
e.g., using discontinuous reconstruction \cite{Helzel_2019_New_JoSC} or
evaluating the upwind direction using more information from the neighbors \cite{Barsukow_2021_active_JoSC}. 

Solutions to hyperbolic systems \cref{eq:1d_hcl} often stay in an \emph{admissible state set} $\mathcal{G}$,
also called the invariant domain.
For instance, the solutions to initial value problems of scalar conservation laws \cref{eq:1d_scalar} satisfy a strict maximum principle (MP) \cite{Dafermos_2000_Hyperbolic_book}, i.e.,
\begin{equation}\label{eq:1d_scalar_g}
	\mathcal{G} = \left\{ u ~|~ m_0 \leqslant u \leqslant M_0 \right\},
	\quad m_0 = \min_{x} u_0(x), ~M_0 = \max_{x} u_0(x).
\end{equation}
Physically, both the density and pressure in the solutions to the compressible Euler equations \cref{eq:1d_euler} should stay positive, i.e.,
\begin{equation}\label{eq:1d_euler_g}
	\mathcal{G} = \left\{\bU = \left(\rho, \rho v, E\right) ~\Big|~ \rho > 0,~ p = (\gamma-1)\left(E - \frac{(\rho v)^2}{2\rho}\right) > 0 \right\}.
\end{equation}
Throughout this paper, it is assumed that $\mathcal{G}$ is a \emph{convex} set,
which is obvious for the scalar case \cref{eq:1d_scalar_g}
and can be verified for the Euler equations \cref{eq:1d_euler_g}, see e.g. \cite{Zhang_2011_Positivity_JoCP}.
It is desirable to conceive so-called bound-preserving (BP) methods,
i.e., those guaranteeing that the numerical solutions at a later time will stay in $\mathcal{G}$, if the initial numerical solutions belong to $\mathcal{G}$.
The BP property of numerical methods is very important for both theoretical analysis and numerical stability.
Many BP methods have been developed in the past few decades,
e.g., a series of works by Shu and collaborators \cite{Zhang_2011_Maximum_PotRSAMPaES, Hu_2013_Positivity_JoCP, Xu_2014_Parametrized_MoC},
a recent general framework on BP methods \cite{Wu_2023_Geometric_SR},
and the convex limiting approach \cite{Guermond_2018_Second_SJoSC, Hajduk_2021_Monolithic_C&MwA, Kuzmin_2020_Monolithic_CMiAMaE},
which can be traced back to the flux-corrected transport (FCT) schemes for scalar conservation laws \cite{Cotter_2016_Embedded_JoCP, Guermond_2017_Invariant_SJoNA, Lohmann_2017_Flux_JoCP, Kuzmin_2012_Flux_book}.
The previous studies on the AF methods pay limited attention to high-speed flows, or problems containing strong discontinuities,
with some efforts on the limiting for the point value update, see e.g. \cite{Barsukow_2021_active_JoSC, Chudzik_2021_Cartesian_AMaC}.
However, those limitings are not enough to guarantee the BP property, as shown in our numerical tests.
In a very recent paper, the MOOD \cite{Clain_2011_high_JoCP} based stabilization was adopted to achieve the BP property \cite{Abgrall_2023_Activea} in an a posteriori fashion.

This paper presents a new way for the point value update to cure the transonic issue and develops suitable BP limiting strategies for the AF methods.
The main contributions and findings in this work can be summarized as follows.
\begin{enumerate}[label=\roman{enumi})., wide=0pt, nosep]
	\item We propose to employ the flux vector splitting (FVS) methods for the point value update to cure the transonic issue,
	since it borrows information from the neighbors naturally and uniformly.
	The FVS was originally used to identify the upwind directions,
	which is simpler and somewhat more efficient than Godunov-type methods for solving hyperbolic systems \cite{Toro_2009_Riemann}.
	In our numerical tests, the FVS is also shown to give better results than the JS, especially the local Lax-Friedrichs (LLF) or Rusanov FVS, in terms of the CFL number and shock-capturing ability.
	The FVS can also cure some defects in two dimensions observed in the JS,
	which will be shown in our future companion paper.
	\item We design BP limitings for both the update of the cell average and the point value by blending the high-order AF methods with the first-order LLF method in a convex combination.
	The convex limiting \cite{Guermond_2018_Second_SJoSC, Hajduk_2021_Monolithic_C&MwA, Kuzmin_2020_Monolithic_CMiAMaE} and the scaling limiter \cite{Liu_1996_Nonoscillatory_SJoNA} are applied to the cell average and point value updates, respectively.
	The main idea is to retain as much as possible of the high-order method while guaranteeing the numerical solutions to be BP,
	and the blending coefficients are computed by enforcing the bounds.
	We show that using a suitable time step size and BP limitings, the numerical solutions of the BP AF methods satisfy the MP for scalar conservation laws, and give positive density and pressure for the compressible Euler equations.
	\item Several challenging test cases such as the LeBlanc and double rarefaction Riemann problems, the Sedov point blast wave, and blast wave interaction problems are conducted to demonstrate the BP properties and the shock-capturing ability,
	which are rare in the literature for the AF methods.
\end{enumerate}

The remainder of this paper is structured as follows.
\Cref{sec:1d_af_schemes} introduces the AF methods based on the JS or FVS for the point value update,
and the power law reconstruction for limiting the derivatives in the point value update.
To design BP methods,  \Cref{sec:1d_limiting_average} describes our convex limiting approach for the cell average,
while \Cref{sec:1d_limiting_point} deals with the limiting for the point value.
Some numerical tests are conducted in  \Cref{sec:results} to experimentally demonstrate the accuracy, BP properties, and shock-capturing ability of the methods.
\Cref{sec:conclusion} concludes the paper with final remarks and future directions.

\section{1D active flux methods for hyperbolic conservation laws}\label{sec:1d_af_schemes}
This section presents the 1D semi-discrete AF methods for the hyperbolic conservation laws \cref{eq:1d_hcl},
based on the JS \cite{Abgrall_2023_Extensions_EMMaNA} or FVS for the point value update. The fully-discrete methods are obtained using Runge-Kutta methods.

Assume that a 1D computational domain is divided into $N$ cells
$I_i = [x_{\xl}, x_{\xr}]$ with cell centers $x_i = (x_{\xl} + x_{\xr})/2$ and cell sizes $\Delta x_i = x_{\xr}-x_{\xl}$, $i=1,\cdots,N$.
The degrees of freedom of the AF methods are the approximations to cell averages of the conservative variable as well as point values at the cell interfaces,
allowing some freedom in the choice of the point values,
e.g. conservative variables, primitive variables, entropy variables, etc.
This paper only considers using the conservative variables, and the degrees of freedom are denoted by
\begin{equation}
	\overline{\bU}_i(t) = \dfrac{1}{\Delta x_i}\int_{I_i} \bU(x,t) ~\dd x,\quad
	\bU_{\xr}(t) = \bU(x_{\xr}, t).
\end{equation}
The cell average is updated by integrating \cref{eq:1d_hcl} over $I_i$ in the following semi-discrete finite volume manner
\begin{equation}\label{eq:semi_av_1d}
	\dfrac{\dd \overline{\bU}_i}{\dd t} = -\dfrac{1}{\Delta x_i}\left[\bF(\bU_{\xr}) - \bF(\bU_{\xl})\right].
\end{equation}
Thus, the accuracy of \cref{eq:semi_av_1d} is determined by the approximation accuracy of the point values.
It was so far (e.g. in \cite{Abgrall_2023_Extensions_EMMaNA}) considered sufficient to update the point values with any finite-difference-like formula
\begin{equation}\label{eq:semi_pnt_1d}
	\dfrac{\dd \bU_{\xr}}{\dd t} = - \bm{\mathcal{R}}\left(\bU_{i+\frac12-l_1}(t), \overline{\bU}_{i+1-l_1}(t), \cdots, \overline{\bU}_{i+l_2}(t), \bU_{i+\frac12+l_2}(t)\right), ~l_1,l_2\geqslant 0,
\end{equation}
with $\bm{\mathcal{R}}$ a consistent approximation of $\partial\bF/\partial x$ at $x_{\xr}$,
as long as it gave rise to a stable method.
This paper explores further conditions on $\bm{\mathcal{R}}$ for nonlinear problems.

\subsection{Point value update using Jacobian splitting}
For smooth solutions, we have an equivalent formulation in the form
\begin{equation}\label{eq:1d_hcl_js}
	\pd{\bU}{t} + \bJ(\bU)\pd{\bU}{x} = 0,\quad \bJ(\bU) = \pd{\bF(\bU)}{\bU}.
\end{equation}
Inspired by the upwind scheme, \cref{eq:1d_hcl_js} can be discretized by the JS \cite{Abgrall_2023_Combination_CoAMaC, Abgrall_2023_Extensions_EMMaNA} as follows
\begin{equation}\label{eq:1d_semi_js}
	\dfrac{\dd \bU_{\xr}}{\dd t} = - \left[\bJ^+(\bU_{\xr})\bD^+_{\xr}(\bU) + \bJ^-(\bU_{\xr})\bD^-_{\xr}(\bU) \right],
\end{equation}
where the splitting of the Jacobian matrix $\bJ = \bJ^{+} + \bJ^{-}$ is defined as
\begin{align*}
	&\bJ^+ = \bR\bm{\Lambda}^+\bR^{-1},\quad \bJ^- = \bR\bm{\Lambda}^-\bR^{-1},\\
	&\bm{\Lambda}^+ = \diag\{\max(\lambda_1, 0), \dots, \max(\lambda_m, 0)\}, \\
	&\bm{\Lambda}^- = \diag\{\min(\lambda_1, 0), \dots, \min(\lambda_m, 0)\},
\end{align*}
based on the eigendecomposition
${\partial \bF}/{\partial\bU} = \bR\bm{\Lambda}\bR^{-1},~\bm{\Lambda} = \diag\{\lambda_1, \dots, \lambda_m\}$,
where $\lambda_1,\cdots,\lambda_m$ are the eigenvalues, with the columns of $\bR$ the corresponding eigenvectors.

To derive the approximation of the derivatives in \cref{eq:1d_semi_js}, one can first obtain a high-order reconstruction for $\bU$ in the upwind cell, and then differentiate the reconstructed polynomial.
As an example, a parabolic reconstruction in cell $i$ is
\begin{align}
	\bU_{\texttt{para}, 1}(x) =& -3(2\overline{\bU}_i - \bU_{\xl} - \bU_{\xr}) \frac{x^2}{\Delta x_i^2}
	+ (\bU_{\xr} - \bU_{\xl}) \frac{x}{\Delta x_i} \nonumber\\
 &+ \frac14(6\overline{\bU}_i - \bU_{\xl} - \bU_{\xr})\label{eq:parabolic_reconstruction}
\end{align}
satisfying
$\bU_{\texttt{para}, 1}(\pm\Delta x_i/2) = \bU_{i\pm\frac12},~
	\frac1{\Delta x_i}\int_{-\Delta x_i/2}^{\Delta x_i/2} \bU_{\texttt{para}, 1}(x) ~\dd x = \overline{\bU}_{i}$.
Then the derivatives are
\begin{subequations}\label{eq:parabolic_av}
	\begin{align}
		\bD^{+}_{\xr}(\bU) =\bU_{\texttt{para}, 1}'(\Delta x_i/2) &= \dfrac{1}{\Delta x_i}\left(2 \bU_{\xl}- 6 \overline{\bU}_i  + 4 \bU_{\xr} \right), \\
		\bD^{-}_{\xr}(\bU) &= \dfrac{1}{\Delta x_{i+1}}\left(- 4 \bU_{\xr} + 6 \overline{\bU}_{i+1} - 2 \bU_{i+\frac32} \right).
	\end{align}
\end{subequations}
They are third-order accurate.
Higher-order extensions can be obtained by higher-order finite difference formulae using a larger spatial stencil, see \cite{Abgrall_2023_Extensions_EMMaNA} for examples.

\subsection{Point value update using flux vector splitting}\label{sec:1d_fvs}
One of the deficiencies of using the JS is the transonic issue that appears for nonlinear problems, as observed in \cite{Helzel_2019_New_JoSC, Barsukow_2021_active_JoSC} and described in more detail next.
Consider \cref{ex:1d_burgers}, where we solve Burgers' equation with a square wave as the initial data.
\Cref{fig:1d_burgers_shock_js} shows the cell averages and point values based on the JS with $200$ cells,
as well as the reference solution.
The numerical solution based on the JS without limiting gives a spike at the initial discontinuity $x=0.2$, which grows linearly in time.
The reason for this behaviour is the inaccurate estimation of the upwind direction at the cell interface.
In this example, there are two successive point values with different initial data near the initial discontinuity, denoted by $u_{\xl}=2$, $u_{\xr}=-1$, respectively.
At the cell interface $x_{\xl}$ or $x_{\xr}$, the upwind discretization in \cref{eq:parabolic_av} only uses the data from the left or right,
leading to zero derivatives, thus the point values $u_{\xl}$ and $u_{\xr}$ stay unchanged.
However, according to the update of the cell average \cref{eq:semi_av_1d},
$\bar{u}_i$ increases gradually (which is the observed spike). This deficiency cannot be eliminated by limitings, as one observes from \cref{fig:1d_burgers_shock_js}.
Some remedies have been proposed, such as using discontinuous reconstruction \cite{Helzel_2019_New_JoSC} and an
``entropy fix'' that evaluates the upwind direction not only at the corresponding cell interface but also with values from its neighbors \cite{Barsukow_2021_active_JoSC}.

In this paper, we propose to use the FVS for the point value update,
which borrows the information from the neighbors naturally, still based on the continuous reconstruction, 
and can eliminate the generation of the spike effectively, as shown in \cref{fig:1d_burgers_shock_fvs}.
The FVS for the point value update reads
\begin{equation}\label{eq:1d_semi_fvs}
	\dfrac{\dd \bU_{\xr}}{\dd t} 
	= - \left[\widetilde{\bD}^+\bF^{+}(\bU)+ \widetilde{\bD}^-\bF^{-}(\bU) \right]_{\xr},
\end{equation}
where the flux $\bF$ is split into the positive and negative parts $\bF= \bF^{+} + \bF^{-}$
satisfying
\begin{equation}\label{eq:FVS_condition}
	\lambda\left(\pd{\bF^{+}}{\bU}\right) \geqslant 0, \quad \lambda\left(\pd{\bF^{-}}{\bU}\right) \leqslant 0,
\end{equation}
i.e., all the eigenvalues of $\pd{\bF^{+}}{\bU}$ and $\pd{\bF^{-}}{\bU}$ are non-negative and non-positive, respectively.
Different FVS can be adopted as long as they satisfy the constraint \cref{eq:FVS_condition}, to be discussed later.
Finite difference formulae to approximate the flux derivatives are obtained similarly to the computation of the derivatives in the JS.
A parabolic reconstruction of the flux can be obtained based on the three flux values as follows
\begin{equation*}
	\bF_{\texttt{para}, 2}(x) = 
 2(\bF_{\xl} - 2\bF_i + \bF_{\xr}) \frac{x^2}{\Delta x_i^2}
	+ (\bF_{\xr} - \bF_{\xl}) \frac{x}{\Delta x_i} + \bF_{i},
\end{equation*}
satisfying
$\bF_{\texttt{para}, 2}(\pm\Delta x_i/2) = \bF_{i\pm\frac12},~
	\bF_{\texttt{para}, 2}(0) = \bF_{i}$,
with $\bF_{i\pm\frac12} = \bF(\bU_{i\pm\frac12})$, and the cell-centered point value $\bF_i = \bF(\bU_i)$ is obtained by evaluating the reconstruction of $\bU$, i.e. according to Simpson's rule
$\bU_i = (- \bU_{\xl} + 6\overline{\bU}_i - \bU_{\xr})/4$.
Then the derivatives are
\begin{subequations}\label{eq:parabolic_pnt}
	\begin{align}
		\left(\widetilde{\bD}^{+}\bF^+\right)_{\xr} = \bF_{\texttt{para}, 2}'(\Delta x_i/2) &= \dfrac{1}{\Delta x_i}\left(\bF_{\xl} - 4 \bF_i  + 3 \bF_{\xr} \right), \\
		\left(\widetilde{\bD}^{-}\bF^-\right)_{\xr} &= \dfrac{1}{\Delta x_{i+1}}\left(- 3 \bF_{\xr} + 4 \bF_{i+1} - \bF_{i+\frac32} \right).
	\end{align}
\end{subequations}
These finite differences are third-order accurate.
While the reconstructions of both $\bU$ and $\bF$ are parabolic, the coefficients in the formula \cref{eq:parabolic_pnt} differ from \cref{eq:parabolic_av} because \cref{eq:parabolic_pnt} uses the cell-centered value rather than the cell average.
Our numerical tests in \Cref{sec:results} show that the AF methods based on the FVS generally give better results than the JS.

\subsubsection{Local Lax-Friedrichs flux vector splitting}
The first FVS we consider is the LLF FVS, defined as
\begin{equation*}
	\bF^\pm = \frac12(\bF(\bU) \pm \alpha \bU),
\end{equation*}
where the choice of $\alpha$ should fulfill \cref{eq:FVS_condition} across the spatial stencil.
In our implementation, it is determined by
\begin{equation}\label{eq:1d_Rusanov_alpha}
	\alpha_{\xr} = \max_{r,\ell} \left\{\abs{\lambda_{\ell}(\bU_r)}\right\},
	~r \in \left\{ i-\frac12, i, i+\frac12, i+1, i+\frac32\right\},~\ell=1,\cdots,m.
\end{equation}
One can also choose $\alpha$ to be the maximal absolute value of the eigenvalues in the whole domain, corresponding to the (global) LF splitting.
Note, however, that a larger $\alpha$ generally leads to a smaller time step size and more dissipation.

\subsubsection{Upwind flux vector splitting}
One can also split the Jacobian matrix based on each characteristic field,
\begin{equation}\label{eq:upwind_fvs}
	\bF^\pm = \frac12(\bF(\bU)\pm \abs{\bJ} \bU),\quad
	\abs{\bJ} = \bR(\bm{\Lambda}^{+} - \bm{\Lambda}^{-})\bR^{-1}.
\end{equation}
For linear systems, one has $\bF = \bJ \bU$, so \cref{eq:upwind_fvs} reduces to the JS. To be specific,
\begin{equation*}
	\bF^\pm = \frac12(\bJ\pm \abs{\bJ})\bU
	= \bR\bm{\Lambda}^{\pm}\bR^{-1} \bU= \bJ^{\pm}\bU,
\end{equation*}
with $\bm{J}^{\pm}$ a constant matrix so that
$\widetilde{\bD}^{\pm}\bF^\pm(\bU) = \bJ^{\pm}\widetilde{\bD}^{\pm}\bU$,
which is the same as $\bJ^{\pm}\bD^{\pm}\bU$ if $\bD^{+}$ and $\widetilde\bD^{+}$ are derived from the same reconstructed polynomial.
In other words, the AF methods using this FVS enjoy the same properties as the original JS-based AF methods for linear systems.

Such an FVS is also known as the Steger-Warming (SW) FVS \cite{Steger_1981_Flux_JoCP} for the Euler equations \cref{eq:1d_euler},
since the ``homogeneity property'' holds \cite{Toro_2009_Riemann}, i.e., $\bF = \bJ \bU$.
One can write down the SW FVS explicitly
\begin{align*}
	\bF^{\pm} &= \begin{bmatrix}
		\frac{\rho}{2\gamma}\alpha^\pm \\
		\frac{\rho}{2\gamma}\left(\alpha^\pm v + a (\lambda_2^\pm - \lambda_3^\pm)\right) \\
		\frac{\rho}{2\gamma}\left(\frac12\alpha^\pm v^2 + a v (\lambda_2^\pm - \lambda_3^\pm) + \frac{a^2}{\gamma-1}(\lambda_2^\pm + \lambda_3^\pm)\right) \\
	\end{bmatrix},
\end{align*}
where
$\lambda_1 = v, ~\lambda_2 = v + a, ~\lambda_3 = v - a,
~
	\alpha^\pm = 2(\gamma-1)\lambda_1^\pm + \lambda_2^\pm + \lambda_3^\pm$,
and $a = \sqrt{\gamma p / \rho}$ is the sound speed.

It should be noted that $\bF^\pm$ in this FVS may not be differentiable with respect to $\bU$ for nonlinear systems (e.g. Euler), as the splitting is based on the absolute value.
In \cite{Leer_1982_Flux_InProceedings}, the mass flux of $\bF^\pm$ is shown to be not differentiable, which explains the accuracy degradation in \cref{ex:1d_accuracy} and the kinks appeared in the density profile.

\subsubsection{Van Leer-H\"anel flux vector splitting for the Euler equations}
Another popular FVS for the Euler equations was proposed by van Leer \cite{Leer_1982_Flux_InProceedings},
and improved by \cite{Haenel_1987_accuracy_InCollection}.
The flux can be split based on the Mach number $M=v/a$ as
\begin{equation*}
	\bF = \begin{bmatrix}
		\rho a M \\
		\rho a^2(M^2 + \frac{1}{\gamma}) \\
		\rho a^3M(\frac{1}{2}M^2 + \frac{1}{\gamma-1}) \\
	\end{bmatrix} = \bF^+ + \bF^-,\quad
	\bF^{\pm} = \begin{bmatrix}
			\pm\frac{1}{4}\rho a (M\pm1)^2 \\
			\pm\frac{1}{4}\rho a (M\pm1)^2 v + p^{\pm} \\
			\pm\frac{1}{4}\rho a (M\pm1)^2 H \\
		\end{bmatrix},
\end{equation*}
with the enthalpy $H=(E+p)/\rho$,
and the pressure splitting
$p^{\pm} = \frac{1}{2}(1\pm\gamma M)p$.
This FVS gives a quadratic differentiable splitting with respect to the Mach number.

\begin{remark}
	Different FVS may lead to different stability conditions. However, it is difficult to study them theoretically. We provide CFL numbers for each test case using different FVS.
	Based on the tests, the LLF FVS is shown to be the best among all the three FVS in terms of CFL number and non-oscillatory property.
\end{remark}

\subsection{Time discretization}
The fully-discrete scheme is obtained by using the SSP-RK3 method \cite{Gottlieb_2001_Strong_SRa}
\begin{equation}\label{eq:ssp_rk3}
	\begin{aligned}
		\bU^{*} &= \bU^{n}+\Delta t^n \bm{L}\left(\bU^{n}\right),\\
		\bU^{**} &= \frac{3}{4}\bU^{n}+\frac{1}{4}\left(\bU^{*}+\Delta t^n \bm{L}\left(\bU^{*}\right)\right),\\
		\bU^{n+1} &= \frac{1}{3}\bU^{n}+\frac{2}{3}\left(\bU^{**}+\Delta t^n \bm{L}\left(\bU^{**}\right)\right),
	\end{aligned}
\end{equation}
where $\bm{L}$ is the right-hand side of the semi-discrete schemes \cref{eq:semi_av_1d} or  \cref{eq:semi_pnt_1d}.
The time step size is determined by the usual CFL condition
\begin{equation}\label{eq:cfl_condition}
	\Delta t^n = \dfrac{C_{\texttt{CFL}}}{\max\limits_{i,\ell}\{\lambda_\ell(\overline{\bU}_i)/\Delta x_i\}}.
\end{equation}

\section{Convex limiting for the cell average}\label{sec:1d_limiting_average}
Although the power law reconstruction \cite{Barsukow_2021_active_JoSC} has been shown to effectively reduce spurious oscillations,
the numerical solutions may still violate certain bounds,
e.g., the appearance of negative density or pressure,
leading to unphysical solutions or even causing the simulations to blow up.
Since the degrees of freedom in the AF methods include both cell averages and point values,
it is necessary to design suitable BP limitings for both of them to achieve the BP property.
The limiting for the cell average has not been addressed much in the literature, except for a very recent work \cite{Abgrall_2023_Activea}.

\begin{definition}
An AF method is called \emph{bound-preserving} (BP) if starting from cell averages and point values in the admissible state set $\mathcal G$, the cell averages and point values remain in $\mathcal G$ at the next time step.
\end{definition}

This section presents a convex limiting approach to achieve the BP property of the cell average update.
The basic idea of the convex limiting approaches \cite{Guermond_2018_Second_SJoSC, Hajduk_2021_Monolithic_C&MwA, Kuzmin_2020_Monolithic_CMiAMaE} is to enforce the preservation of local and global bounds by constraining individual numerical fluxes.
The BP or invariant domain-preserving (IDP) properties of flux-limited approximations are shown using representations in terms of intermediate states that stay in convex admissible state sets \cite{Guermond_2018_Second_SJoSC, Guermond_2019_Invariant_CMiAMaE}.
The low-order scheme is chosen as the first-order LLF scheme
\begin{align*}
	&\overline{\bU}^{\texttt{L}}_i = \overline{\bU}^{n}_i - \mu_{i}\left(\widehat{\bF}^{\texttt{L}}_{\xr} - \widehat{\bF}^{\texttt{L}}_{\xl}\right),\quad
	\mu_{i} = \Delta t^n / \Delta x_i, \\
	&\widehat{\bF}^{\texttt{L}}_{\xr} = \frac12\left(\bF(\overline{\bU}^{n}_i) + \bF(\overline{\bU}^{n}_{i+1})\right) - \frac{\alpha_{\xr}}{2}\left(\overline{\bU}^{n}_{i+1} - \overline{\bU}^{n}_i\right),
\end{align*}
where $\alpha_{\xr}$ is an \emph{upper bound} for the maximum wave speed of the Riemann problem with the initial data $\bU_{i}, \bU_{i+1}$,
whose estimation for scalar conservation laws and the Euler equations can be found in \cite{Guermond_2016_Invariant_SJoNA} and \cite{Guermond_2016_Fast_JoCP}, respectively.
Note that here $\alpha_{\xr}$ is not the same as the one in the LLF FVS \cref{eq:1d_Rusanov_alpha}.
Following \cite{Guermond_2016_Invariant_SJoNA}, the first-order LLF scheme can be rewritten as
\begin{equation}\label{eq:1d_lo_decomp}
	\overline{\bU}^{\texttt{L}}_i = \left[1-\mu_i\left(\alpha_{\xl}+\alpha_{\xr}\right)\right] \overline{\bU}^{n}_i
	+ \mu_i\alpha_{\xl}\widetilde{\bU}_{\xl} + \mu_i\alpha_{\xr}\widetilde{\bU}_{\xr},
\end{equation}
with the intermediate states defined as
\begin{equation}\label{eq:1d_llf_inter_states}
	\begin{aligned}
		&\widetilde{\bU}_{\xl} := \frac12\left(\overline{\bU}^{n}_{i-1} + \overline{\bU}^{n}_{i}\right)
		+ \frac{1}{2\alpha_{\xl}}\left[\bF(\overline{\bU}^{n}_{i-1}) - \bF(\overline{\bU}^{n}_{i})\right], \\
		&\widetilde{\bU}_{\xr} := \frac12\left(\overline{\bU}^{n}_{i} + \overline{\bU}^{n}_{i+1}\right)
		+ \frac{1}{2\alpha_{\xr}}\left[\bF(\overline{\bU}^{n}_{i}) - \bF(\overline{\bU}^{n}_{i+1})\right].
	\end{aligned}
\end{equation}

\begin{remark}\rm
	As $\alpha_{\xr}$ is chosen to be larger than the leftmost and rightmost wave speed,
	the intermediate state defined in \cref{eq:1d_llf_inter_states} is indeed an average of the exact Riemann solution \cite{Guermond_2016_Invariant_SJoNA},
	thus it belongs to $\mathcal{G}$.
	For systems, it is also the intermediate state of the HLL solver \cite{Harten_1983_Upstream_SR}.
	Moreover, the intermediate state \cref{eq:1d_llf_inter_states} preserves all \emph{convex invariants} (e.g., density and pressure positivity, and minimum entropy principle for the Euler equations) of initial value problems for hyperbolic systems \cite{Guermond_2016_Invariant_SJoNA}.
\end{remark}

\begin{lemma}[Guermond and Popov \cite{Guermond_2016_Invariant_SJoNA}]\label{lem:llf_g}\rm
	If the time step size $\Delta t^n$ satisfies
	\begin{equation}\label{eq:1d_convex_combination_dt}
		\Delta t^n \leqslant \dfrac{\Delta x_i}{\alpha_{\xl}+\alpha_{\xr}},
	\end{equation}
	then \cref{eq:1d_lo_decomp} is a convex combination,
	and the first-order LLF scheme is BP.
\end{lemma}
The proof relies on the fact that the intermediate state \cref{eq:1d_llf_inter_states} stays in the admissible state set $\mathcal{G}$ and the convexity of $\mathcal{G}$.
The arguments based on the convex combination follow in spirit to those used in \cite{Perthame_1996_positivity_NM}.

Upon defining the anti-diffusive flux $\Delta \widehat{\bF}_{i\pm\frac12} := \widehat{\bF}^{\texttt{H}}_{i\pm\frac12} - \widehat{\bF}^{\texttt{L}}_{i\pm\frac12}$ with $\widehat{\bF}^{\texttt{H}}_{i\pm\frac12} := \bF(\bU_{i\pm\frac12})$, a forward-Euler step applied to the semi-discrete high-order scheme for the cell average \cref{eq:semi_av_1d} can be written as
\begin{align}
	&\overline{\bU}^{\texttt{H}}_i = \overline{\bU}^{n}_i - \mu_i (\widehat{\bF}^{\texttt{H}}_{\xr} - \widehat{\bF}^{\texttt{H}}_{\xl}) = \overline{\bU}^{n}_i - \mu_i (\widehat{\bF}^{\texttt{L}}_{\xr} - \widehat{\bF}^{\texttt{L}}_{\xl}) - \mu_i (\Delta \widehat{\bF}_{\xr} - \Delta \widehat{\bF}_{\xl}) \nonumber\\
 &~~~~ =: \left[1-\mu_i\left(\alpha_{\xl}+\alpha_{\xr}\right)\right] \overline{\bU}^{n}_i
	+ \mu_i\alpha_{\xl} \widetilde{\bU}_{\xl}^{\texttt{H},+}
	+ \mu_i\alpha_{\xr} \widetilde{\bU}_{\xr}^{\texttt{H},-},\label{eq:1d_ho_decomp}\\
 &\widetilde{\bU}_{\xl}^{\texttt{H},+} := \left(\widetilde{\bU}_{\xl} + \frac{\Delta\widehat{\bF}_{\xl}}{\alpha_{\xl}}\right),\quad
	\widetilde{\bU}_{\xr}^{\texttt{H},-} := \left(\widetilde{\bU}_{\xr} - \frac{\Delta\widehat{\bF}_{\xr}}{\alpha_{\xr}}\right).\nonumber
\end{align}

With the low-order scheme \cref{eq:1d_lo_decomp} and high-order scheme \cref{eq:1d_ho_decomp} having the same form one can now define the limited scheme for the cell average as
\begin{equation}\label{eq:1d_limited_decomp}
	\overline{\bU}^{\texttt{Lim}}_i = \left[1-\mu_i\left(\alpha_{\xl}+\alpha_{\xr}\right)\right] \overline{\bU}^{n}_i
	+ \mu_i\alpha_{\xl}\widetilde{\bU}_{\xl}^{\texttt{Lim},+}
	+ \mu_i\alpha_{\xr}\widetilde{\bU}_{\xr}^{\texttt{Lim},-},
\end{equation}
with the limited intermediate states
\begin{align*}
	&\widetilde{\bU}_{\xl}^{\texttt{Lim},+} = \widetilde{\bU}_{\xl} + \frac{\Delta\widehat{\bF}^{\texttt{Lim}}_{\xl}}{\alpha_{\xl}}
	:= \widetilde{\bU}_{\xl} + \frac{\theta_{\xl}\Delta\widehat{\bF}_{\xl}}{\alpha_{\xl}}, \\
	&\widetilde{\bU}_{\xr}^{\texttt{Lim},-} = \widetilde{\bU}_{\xr} - \frac{\Delta\widehat{\bF}^{\texttt{Lim}}_{\xr}}{\alpha_{\xr}}
	:= \widetilde{\bU}_{\xr} - \frac{\theta_{\xr}\Delta\widehat{\bF}_{\xr}}{\alpha_{\xr}},
\end{align*}
where the coefficients $\theta_{i\pm \frac12}\in[0,1]$.

\begin{proposition}\rm
	If the cell average at the last time step $\overline{\bU}_i^n$ and the limited intermediate states $\widetilde{\bU}_{i\pm\frac12}^{\texttt{Lim},\mp}$ belong to the admissible state set $\mathcal{G}$,
	then the limited average update \cref{eq:1d_limited_decomp} is BP, i.e., $\overline{\bU}^{\texttt{Lim}}_i \in \mathcal{G}$, under the CFL condition \cref{eq:1d_convex_combination_dt}.
	If the SSP-RK3 \cref{eq:ssp_rk3} is used for the time integration,
	the high-order scheme is also BP.	
\end{proposition}

\begin{proof}
    Under the constraint \cref{eq:1d_convex_combination_dt}, the limited cell average update $\overline{\bU}^{\texttt{Lim}}_i$ is a convex combination of $\overline{\bU}_i^n$ and $\widetilde{\bU}_{i\pm\frac12}^{\texttt{Lim},\mp}$,
	thus it belongs to $\mathcal{G}$ due to the convexity of $\mathcal{G}$.
	Because the SSP-RK3 is a convex combination of forward-Euler stages,
	the high-order scheme equipped with the SSP-RK3 is also BP according to the convexity.
\end{proof}

\begin{remark}\rm
    The scheme \cref{eq:1d_limited_decomp} is conservative as it amounts to using the numerical flux
    $    \widehat{\bF}^{\texttt{L}}_{\xr} + \theta_{\xr}\Delta \widehat{\bF}_{\xr} = \theta_{\xr} \widehat{\bF}^{\texttt{H}}_{\xr} + (1 - \theta_{\xr}) \widehat{\bF}^{\texttt{L}}_{\xr}$,
    which is a convex combination of the high-order and low-order fluxes.
\end{remark}

\begin{remark}\rm
	It should be noted that the time step size \cref{eq:1d_convex_combination_dt} is determined based on the solutions at $t^n$.
	If the constraint is not satisfied at the later stage of the SSP-RK3,
	the BP property may not be achieved because \cref{eq:1d_limited_decomp} is no longer a convex combination.
	In our implementation, we start from the usual CFL condition \cref{eq:cfl_condition}.
    Then, if the high-order AF states need BP limitings and \cref{eq:1d_llf_inter_states} is not BP or \cref{eq:1d_convex_combination_dt} is not satisfied,
	the numerical solutions are set back to the last time step,
	and we rerun with a halved time step size until \cref{eq:1d_llf_inter_states} is BP and the constraint \cref{eq:1d_convex_combination_dt} is satisfied.
	This is also a typical implementation to save computational costs in other BP methods.
\end{remark}

The remaining task is to determine the coefficients at each interface $\theta_{i\pm\frac12}$ such that $\widetilde{\bU}_{i\pm\frac12}^{\texttt{Lim},\mp}\in\mathcal{G}$ and stay as close as possible to the high-order states $\widetilde{\bU}_{i\pm\frac12}^{\texttt{H}}$,
i.e., the goal is to find the largest $\theta_{i\pm\frac12}\in[0,1]$ such that $\widetilde{\bU}_{i\pm\frac12}^{\texttt{Lim},\mp}\in\mathcal{G}$.

\subsection{Application to scalar conservation laws}
This section is devoted to applying the convex limiting approach to scalar conservation laws \cref{eq:1d_scalar},
such that the numerical solutions satisfy the global or local MP.
For the global MP, the blending coefficient $\theta_{i+\frac12}\in[0,1]$ should be chosen such that
$m_0 \leqslant \tilde{u}_{i+\frac12}^{\texttt{Lim},\pm} \leqslant M_0$,
with $m_0, M_0$ defined in \cref{eq:1d_scalar_g},
which gives
\begin{equation*}
	\theta_{\xr} = \begin{cases}
		\min\left\{1, \frac{\alpha_{\xr}(\tilde{u}_{\xr}-m_0)}{\Delta \hat{f}_{\xr}},
		\frac{\alpha_{\xr}(M_0-\tilde{u}_{\xr})}{\Delta \hat{f}_{\xr}} \right\}, &\text{if}~~ \Delta \hat{f}_{\xr} > 0, \\
		\min\left\{1, \frac{\alpha_{\xr}(m_0-\tilde{u}_{\xr})}{\Delta \hat{f}_{\xr}},
		\frac{\alpha_{\xr}(\tilde{u}_{\xr}-M_0)}{\Delta \hat{f}_{\xr}} \right\}, &\text{if}~~ \Delta \hat{f}_{\xr} < 0. \\
	\end{cases}
\end{equation*}
To avoid a small denominator, the limited anti-diffusive flux can be obtained directly,
\begin{equation*}
	\Delta\hat{f}^{\texttt{Lim}}_{\xr} = \begin{cases}
		\min\left\{\Delta\hat{f}_{\xr}, ~\alpha_{\xr}(\tilde{u}_{\xr}-m_0),
		~\alpha_{\xr}(M_0-\tilde{u}_{\xr}) \right\}, &\text{if}~~ \Delta \hat{f}_{\xr} \geqslant 0, \\
		\max\left\{\Delta\hat{f}_{\xr}, ~\alpha_{\xr}(m_0-\tilde{u}_{\xr}),
		~\alpha_{\xr}(\tilde{u}_{\xr}-M_0) \right\}, &\text{otherwise}. \\
	\end{cases}
\end{equation*}
On the other hand, one can also enforce the local MP
$u^{\min}_i \leqslant \tilde{u}_{i+\frac12}^{\texttt{Lim},-} \leqslant u^{\max}_i$,
$u^{\min}_{i+1} \leqslant \tilde{u}_{i+\frac12}^{\texttt{Lim},+} \leqslant u^{\max}_{i+1}$,
which helps to suppress spurious oscillations and improve shock-capturing ability.
The corresponding limited anti-diffusive flux is
\begin{equation*}
	\Delta\hat{f}^{\texttt{Lim}}_{\xr} = \begin{cases}
		\min\left\{\Delta\hat{f}_{\xr}, ~\alpha_{\xr}(\tilde{u}_{\xr}-u^{\min}_i),
		~\alpha_{\xr}(u^{\max}_{i+1}-\tilde{u}_{\xr}) \right\}, &\text{if}~ \Delta \hat{f}_{\xr} \geqslant 0, \\
		\max\left\{\Delta\hat{f}_{\xr}, ~\alpha_{\xr}(u^{\min}_{i+1}-\tilde{u}_{\xr}),
		~\alpha_{\xr}(\tilde{u}_{\xr}-u^{\max}_{i}) \right\}, &\text{otherwise}. \\
	\end{cases}
\end{equation*}
The choice of the local bounds can be based on the intermediate states
\begin{equation*}
	u^{\min}_i = \min\left\{\bar{u}_i^n, ~\tilde{u}_{\xl}, ~\tilde{u}_{\xr}\right\},~
	u^{\max}_i = \max\left\{\bar{u}_i^n, ~\tilde{u}_{\xl}, ~\tilde{u}_{\xr}\right\}.
\end{equation*}
Finally, the numerical flux is
\begin{equation}\label{eq:1d_flux_limited_scalar}
	\hat{f}^{\texttt{Lim}}_{\xr} = \hat{f}^{\texttt{L}}_{\xr} + \Delta\hat{f}^{\texttt{Lim}}_{\xr}.
\end{equation}

\subsection{Application to the compressible Euler equations}

This section aims at enforcing the positivity of density and pressure.

\subsubsection{Positivity of density}
The first step is to impose the density positivity
$\widetilde{\bU}_{\xr}^{\texttt{Lim},\pm,\rho} > \bar{\varepsilon}^\rho_{\xr}$,
where $\bU^{*,\rho}$ denotes the density component of $\bU^{*}$,
and $\bar{\varepsilon}^\rho_{\xr}$ is a small positive number defined by $\bar{\varepsilon}^\rho_{\xr} := \min\{10^{-13},  \widetilde{\bU}_{\xr}^{\rho}\}$.
The corresponding density component of the limited anti-diffusive flux is
\begin{equation*}
	\Delta\widehat{\bF}^{\texttt{Lim},*,\rho}_{\xr} = \begin{cases}
		\min\left\{\Delta\widehat{\bF}^{\rho}_{\xr}, ~\alpha_{\xr}\left(\widetilde{\bU}^{\rho}_{\xr}-\bar{\varepsilon}^\rho_{\xr}\right) \right\}, &\text{if}~ \Delta \widehat{\bF}^{\rho}_{\xr} \geqslant 0, \\
		\max\left\{\Delta\widehat{\bF}^{\rho}_{\xr}, ~\alpha_{\xr}\left(\bar{\varepsilon}^\rho_{\xr}-\widetilde{\bU}^{\rho}_{\xr}\right) \right\}, &\text{otherwise}. \\
	\end{cases}
\end{equation*}
Then the density component of the limited numerical flux is $\widehat{\bF}_{\xr}^{\texttt{Lim}, *, \rho} = \widehat{\bF}_{\xr}^{\texttt{L}, \rho} + \Delta\widehat{\bF}_{\xr}^{\texttt{Lim}, *, \rho}$,
with the other components remaining the same as $\widehat{\bF}_{\xr}^{\texttt{H}}$.

\subsubsection{Positivity of pressure}
The second step is to enforce pressure positivity
$p(\widetilde{\bU}_{\xr}^{\texttt{Lim},\pm}) > \bar{\varepsilon}^p_{\xr} := \min\{10^{-13},  p(\widetilde{\bU}_{\xr})\}$, $p(\bU^{*})$ denotes the pressure recovered from $\bU^{*}$, with
\begin{equation*}
	\widetilde{\bU}_{\xr}^{\texttt{Lim},\pm} = \widetilde{\bU}_{\xr} \pm \frac{\theta_{\xr}\Delta \widehat{\bF}_{\xr}^{\texttt{Lim}, *}}{\alpha_{\xr}},\quad
	\Delta \widehat{\bF}_{\xr}^{\texttt{Lim}, *} = \widehat{\bF}_{\xr}^{\texttt{Lim}, *} - \widehat{\bF}_{\xr}^{\texttt{L}}.
\end{equation*}
Such constraints lead to two inequalities
\begin{equation}\label{eq:pressure_inequality}
	A_{\xr}\theta^2_{\xr} \pm B_{\xr}\theta_{\xr} < C_{\xr},
\end{equation}
with the coefficients
\begin{align*}
	&A_{\xr} = \dfrac{1}{2} \left(\Delta \widehat\bF^{\texttt{Lim}, *, \rho v}_{\xr}\right)^2
	- \Delta \widehat\bF^{\texttt{Lim}, *, \rho}_{\xr} \Delta \widehat\bF^{\texttt{Lim}, *, E}_{\xr}, \\
	&B_{\xr} = \alpha_{\xr}\left(\Delta \widehat\bF^{\texttt{Lim}, *, \rho}_{\xr} \widetilde\bU^{E}_{\xr}
	+ \widetilde\bU^{\rho}_{\xr} \Delta \widehat\bF^{\texttt{Lim}, *, E}_{\xr}
	- \Delta \widehat\bF^{\texttt{Lim}, *, \rho v}_{\xr} \widetilde\bU^{\rho v}_{\xr}
	- \tilde{\varepsilon} \Delta \widehat\bF^{\texttt{Lim}, *, \rho}_{\xr}\right), \\
	&C_{\xr} = \alpha_{\xr}^2\left(\widetilde\bU^{\rho}_{\xr} \widetilde\bU^{E}_{\xr}
	- \dfrac{1}{2} \left(\widetilde\bU^{\rho v}_{\xr}\right)^2
	- \tilde{\varepsilon}\widetilde\bU^{\rho}_{\xr}\right), \\
    & \tilde{\varepsilon} = \bar{\varepsilon}^p_{\xr}/(\gamma-1).
\end{align*}
Following \cite{Kuzmin_2020_Monolithic_CMiAMaE}, the inequalities \cref{eq:pressure_inequality} hold under the linear sufficient condition
\begin{equation*}
	\left(\max\left\{0, A_{\xr}\right\} + \left|B_{\xr}\right|\right)\theta_{\xr}< C_{\xr},
\end{equation*}
if making use of $\theta_{\xr}^2 \leqslant \theta_{\xr},~ \theta_{\xr}\in[0,1]$.
Thus the coefficient can be chosen as
\begin{equation*}
	\theta_{\xr} = \min\left\{1,~ \frac{C_{\xr}}{\max\{0, A_{\xr}\} + \abs{B_{\xr}}}\right\},
\end{equation*}
and the final limited numerical flux is
\begin{equation}\label{eq:1d_flux_limited_Euler}
	\widehat{\bF}_{\xr}^{\texttt{Lim}} = \widehat{\bF}_{\xr}^{\texttt{L}} + \theta_{\xr}\Delta\widehat{\bF}_{\xr}^{\texttt{Lim}, *}.
\end{equation}

\section{Scaling limiter for point value}\label{sec:1d_limiting_point}
To achieve the BP property, it is also necessary to introduce BP limiting for the point value.
As one will see in the numerical tests in \Cref{sec:results},
using power law reconstruction or BP limiting for cell average, individually or in combination, cannot guarantee the bounds.
As there is no conservation requirement on the point value update,
a simple scaling limiter \cite{Liu_1996_Nonoscillatory_SJoNA} is directly performed on the high-order point values rather than on the flux for the cell average.

A first-order LLF scheme for the point value update can be
\begin{equation}\label{eq:1d_llf_pnt}
	 \bU_{\xr}^{\texttt{L}} =  \bU_{\xr}^{n} - \dfrac{2\Delta t^n}{\Delta x_i + \Delta x_{i+1}}\left(\widehat{\bF}^{\texttt{L}}_{i+1}( \bU_{i+\frac12}^n,  \bU_{i+\frac32}^n)
	- \widehat{\bF}^{\texttt{L}}_{i}( \bU_{i-\frac12}^n,  \bU_{i+\frac12}^n)\right),
\end{equation}
with the numerical flux
\begin{align*}
	&\widehat{\bF}^{\texttt{L}}_{i}( \bU_{i-\frac12}^n,  \bU_{i+\frac12}^n) =
	\frac12\left( \bF(\bU_{i-\frac12}^n) +  \bF(\bU_{i+\frac12}^n)\right) - \frac{\alpha_{i}}{2}\left(\bU_{i+\frac12}^n - \bU_{i-\frac12}^n\right),\\
	&\alpha_{i} = \max\{\lambda(\bU_{i-\frac12}^n), ~\lambda(\bU_{i+\frac12}^n)\}.
\end{align*}
Such an LLF scheme can be interpreted as a scheme on a staggered mesh if the point value is viewed as the cell average on the staggered mesh.
Based on the proof in \cite{Perthame_1996_positivity_NM},
it is straightforward to obtain the following Lemma.

\begin{lemma}\rm
	The LLF scheme for the point value \cref{eq:1d_llf_pnt} is BP under the CFL condition
	\begin{equation}\label{eq:1d_pnt_llf_dt}
		\Delta t^n \leqslant \dfrac{\Delta x_i + \Delta x_{i+1}}{4 \max\{\alpha_i, \alpha_{i+1}\} }.
	\end{equation}
\end{lemma}

The limited state is obtained by blending the high-order AF scheme \cref{eq:semi_pnt_1d} with the forward Euler scheme and the LLF scheme \cref{eq:1d_llf_pnt} as $\bU_{\xr}^{\texttt{Lim}} = \theta_{\xr} \bU_{\xr}^{\texttt{H}} + (1-\theta_{\xr}) \bU_{\xr}^{\texttt{L}}$,
such that $\bU_{\xr}^{\texttt{Lim}}\in\mathcal{G}$.

\begin{remark}\rm
	In the FVS for the point value update, the cell-centered value $\bU_i$ is used.
	It is possible that $\bU_i\notin\mathcal{G}$, then a scaling limiter \cite{Zhang_2010_positivity_JoCP} will be used to blend $\bU_i$ with $\overline{\bU}_i$, such that $\bU_i\in\mathcal{G}$.
\end{remark}

\subsection{Application to scalar conservation laws}
This section enforces the global MP $m_0 \leqslant u_{\xr}^{\texttt{Lim}} \leqslant M_0$ by choosing the coefficient as
\begin{equation*}
	\theta_{\xr} = \begin{cases}
		\dfrac{u_{\xr}^{\texttt{L}}-m_0}{u_{\xr}^{\texttt{L}}-u_{\xr}^{\texttt{H}}},&\text{if}~~ u_{\xr}^{\texttt{H}} < m_0, \\
		\dfrac{M_0-u_{\xr}^{\texttt{L}}}{u_{\xr}^{\texttt{H}}-u_{\xr}^{\texttt{L}}},&\text{if}~~ u_{\xr}^{\texttt{H}} > M_0. \\
	\end{cases}
\end{equation*}
The final limited state is
\begin{equation}\label{eq:1d_pnt_limited_state_scalar}
	u_{\xr}^{\texttt{Lim}} = \theta_{\xr} u_{\xr}^{\texttt{H}} + \left(1-\theta_{\xr}\right) u_{\xr}^{\texttt{L}}.
\end{equation}
	One can also enforce the local bounds $u_{\xr}^{\min} \leqslant u_{\xr}^{\texttt{Lim}} \leqslant u_{\xr}^{\max}$, with
	\begin{equation*}
		u^{\min}_{\xr} = \min\left\{\bar{u}_i^n, ~\bar{u}_{i+1}^n, ~u_{\xr}^n\right\},~
		u^{\max}_{\xr} = \max\left\{\bar{u}_i^n, ~\bar{u}_{i+1}^n, ~u_{\xr}^n\right\}.
	\end{equation*}

\subsection{Application to the compressible Euler equations}
The limiting consists of two steps.
First, the high-order state $\bU_{\xr}^{\texttt{H}}$ is modified as $\bU_{\xr}^{\texttt{Lim}, *}$,
such that its density component satisfies $\bU_{\xr}^{\texttt{Lim}, *, \rho} > \varepsilon^\rho_{\xr}:=\min\{10^{-13}, \bU_{\xr}^{\texttt{L}, \rho}\}$.
Solving this inequality gives the coefficient
\begin{equation*}
	\theta^{*}_{\xr} = \begin{cases}
		\dfrac{\bU_{\xr}^{\texttt{L}, \rho}-\varepsilon^\rho_{\xr}}{\bU_{\xr}^{\texttt{L}, \rho}-\bU_{\xr}^{\texttt{H}, \rho}},&\text{if}~~ \bU_{\xr}^{\texttt{H}, \rho} < \varepsilon^\rho_{\xr}, \\
		1, &\text{otherwise}. \\
	\end{cases}
\end{equation*}
Then the density component of the limited state is $\bU_{\xr}^{\texttt{Lim}, *, \rho} = \theta^{*}_{\xr} \bU_{\xr}^{\texttt{H}, \rho} + (1-\theta^{*}_{\xr}) \bU_{\xr}^{\texttt{L}, \rho}$,
with the other components remaining the same as $\bU_{\xr}^{\texttt{H}}$.

Then the limited state $\bU_{\xr}^{\texttt{Lim}, *}$ is modified as $\bU_{\xr}^{\texttt{Lim}}$,
such that it gives positive pressure, i.e., $p\left(\bU_{\xr}^{\texttt{Lim}}\right) > \varepsilon^p_{\xr}$.
Let $\bU_{\xr}^{\texttt{Lim}} = \theta^{**}_{\xr} \bU_{\xr}^{\texttt{Lim}, *} + (1-\theta^{**}_{\xr}) \bU_{\xr}^{\texttt{L}}$.
Note that the pressure is a concave function (see e.g. \cite{Zhang_2011_Maximum_PotRSAMPaES}) of the conservative variables,
such that
\begin{equation*}
	p\left(\bU_{\xr}^{\texttt{Lim}}\right) \geqslant \theta^{**}_{\xr}p\left(\bU_{\xr}^{\texttt{Lim}, *}\right) + \left(1-\theta^{**}_{\xr}\right)p\left(\bU_{\xr}^{\texttt{L}}\right)
\end{equation*}
based on Jensen's inequality and $\bU_{\xr}^{\texttt{Lim}, *, \rho} > 0$, $\bU_{\xr}^{\texttt{L}, \rho} > 0$, $\theta_{\xr}^{**} \in [0,1]$.
Thus a sufficient condition is
\begin{equation*}
	\theta^{**}_{\xr} = \begin{cases}
		\dfrac{p\left(\bU_{\xr}^{\texttt{L}}\right) - \varepsilon^p_{\xr}}{p\left(\bU_{\xr}^{\texttt{L}}\right) - p\left(\bU_{\xr}^{\texttt{Lim}, *}\right)},&\text{if}~~ p\left(\bU_{\xr}^{\texttt{Lim}, *}\right) < \varepsilon^p_{\xr}, \\
		1, &\text{otherwise}. \\
	\end{cases}
\end{equation*}
The final limited state is
\begin{equation}\label{eq:1d_pnt_limited_state_Euler}
	\bU_{\xr}^{\texttt{Lim}} = \theta^{**}_{\xr} \bU_{\xr}^{\texttt{Lim}, *} + \left(1-\theta^{**}_{\xr}\right) \bU_{\xr}^{\texttt{L}}.
\end{equation}

Let us summarize the main results of the BP AF methods in this paper.
\begin{theorem}\rm
	If the initial numerical solution $\overline{\bU}_i^0, \bU_{\xr}^0\in\mathcal{G}$ for all $i$,
	and the time step size satisfies \cref{eq:1d_convex_combination_dt} and \cref{eq:1d_pnt_llf_dt},
	then the AF methods \cref{eq:semi_av_1d}-\cref{eq:semi_pnt_1d} equipped with the SSP-RK3 \cref{eq:ssp_rk3} and the BP limitings
	\begin{itemize}[leftmargin=*]
		\item \cref{eq:1d_flux_limited_scalar} and \cref{eq:1d_pnt_limited_state_scalar} preserve the maximum principle for scalar case;
		\item \cref{eq:1d_flux_limited_Euler} and \cref{eq:1d_pnt_limited_state_Euler} preserve the density and pressure positivity for the Euler equations.
	\end{itemize}
\end{theorem}

\section{Numerical results}\label{sec:results}
This section conducts some numerical tests to verify the accuracy of using the FVS for point value updates,
the BP property, and the shock-capturing ability of our AF methods.

\subsection{Scalar conservation laws}
This section shows the results for the linear advection equation and the Burgers' equation, which demonstrate that the proposed limiting can preserve the MP and suppress oscillations well.

\begin{example}[Advection equation]\label{ex:1d_advection_discontinuity}\rm
	Consider the 1D advection equation $u_t + u_x = 0$,
	on the periodic domain $[-1, 1]$ with the initial data \cite{Jiang_1996_Efficient_JoCP}
	\begin{equation*}
		\begin{cases}
			~\frac{1}{6}\left(G_1(x, \beta, z-\delta) + G_1(x, \beta, z+\delta) + 4G_1(x, \beta, z)\right), &\text{if}~ -0.8\leqslant x \leqslant -0.6, \\
			~1, &\text{if}~ -0.4 \leqslant x \leqslant -0.2, \\
			~1 - \abs{10(x-0.1)}, &\text{if}~~ 0 \leqslant x \leqslant 0.2, \\
			~\frac{1}{6}\left(G_2(x, \alpha, a-\delta) + G_2(x, \alpha, a+\delta) + 4G_2(x, \alpha, a)\right), &\text{if}~~ 0.4 \leqslant x \leqslant 0.6, \\
			~0, &\text{otherwise}, \\
		\end{cases}
	\end{equation*}
	where $G_1(x, \beta, z) = \exp(-\beta(x-z)^2)$, $G_2(x, \alpha, a) = \sqrt{\max(1-\alpha^2(x-a)^2, 0)}$,
	and the constants are $a = -0.5, z = -0.7, \delta = 0.005, \alpha = 10, \beta = \ln 2/(36\delta ^2)$.
	The problem is solved for one period, i.e., until $T= 2$.
\end{example}

For the advection equation, the JS and LLF FVS are equivalent.
The maximal CFL number leading to a stable simulation is $0.41$ without any limiting,
and it reduces to $0.13$ when only the power law reconstruction is activated,
and it increases a little bit to $0.42$ when only the BP limitings are used.
When the power law reconstruction and the BP limitings are employed together,
the maximal CFL number can be $0.4$.
The reduction of the CFL number with the power law reconstruction for semi-discrete AF has, in fact, not been noticed previously.
Thus, in the following simulations we try not to use the power law reconstruction unless otherwise stated.

The results obtained with different limitings are shown in \cref{fig:1d_advection_limiting},
which are computed with $400$ cells and the CFL number is $0.1$.
The ranges of the numerical solutions are listed in \cref{tab:1d_advection_bounds},
considering both the cell averages and point values.
One can observe that there are some oscillations near the discontinuities without any limiting,
and that the power law reconstruction can eliminate the oscillations effectively but is still not BP.
The activation of either the BP limiting for the cell average alone or the BP limiting for the point value alone also fails to preserve the bounds $[0,1]$, as one can see from \cref{tab:1d_advection_bounds},
as is the case when using both the BP limiting for the cell average and the power law reconstruction in the point value update.
Only when a BP limiting is performed on both the cell average and the point value, the BP property is achieved,
showing that using the two BP limitings simultaneously is necessary for the preservation of the MP.
\Cref{fig:1d_advection_limiting} also shows the results obtained by imposing the global or local MP for the cell average and point value (without power law reconstruction),
indicating that the use of local MP tends to dissipate the numerical solutions near the discontinuities and clip maxima more than the global MP.

\begin{figure}[htbp]
	\centering
	\begin{subfigure}[b]{0.48\textwidth}
		\centering
		\includegraphics[width=1.0\linewidth]{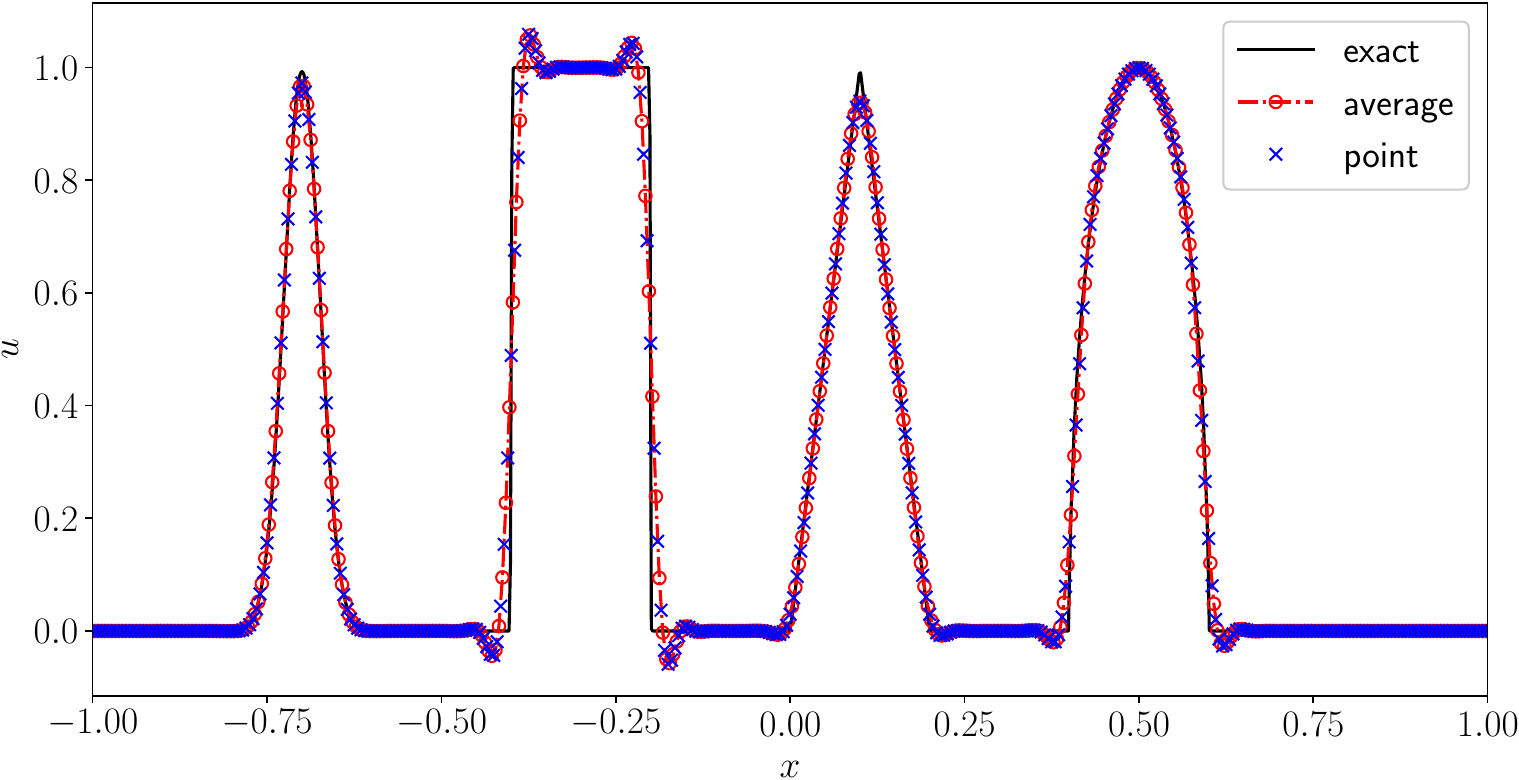}
	\end{subfigure}
	\begin{subfigure}[b]{0.48\textwidth}
		\centering
		\includegraphics[width=1.0\linewidth]{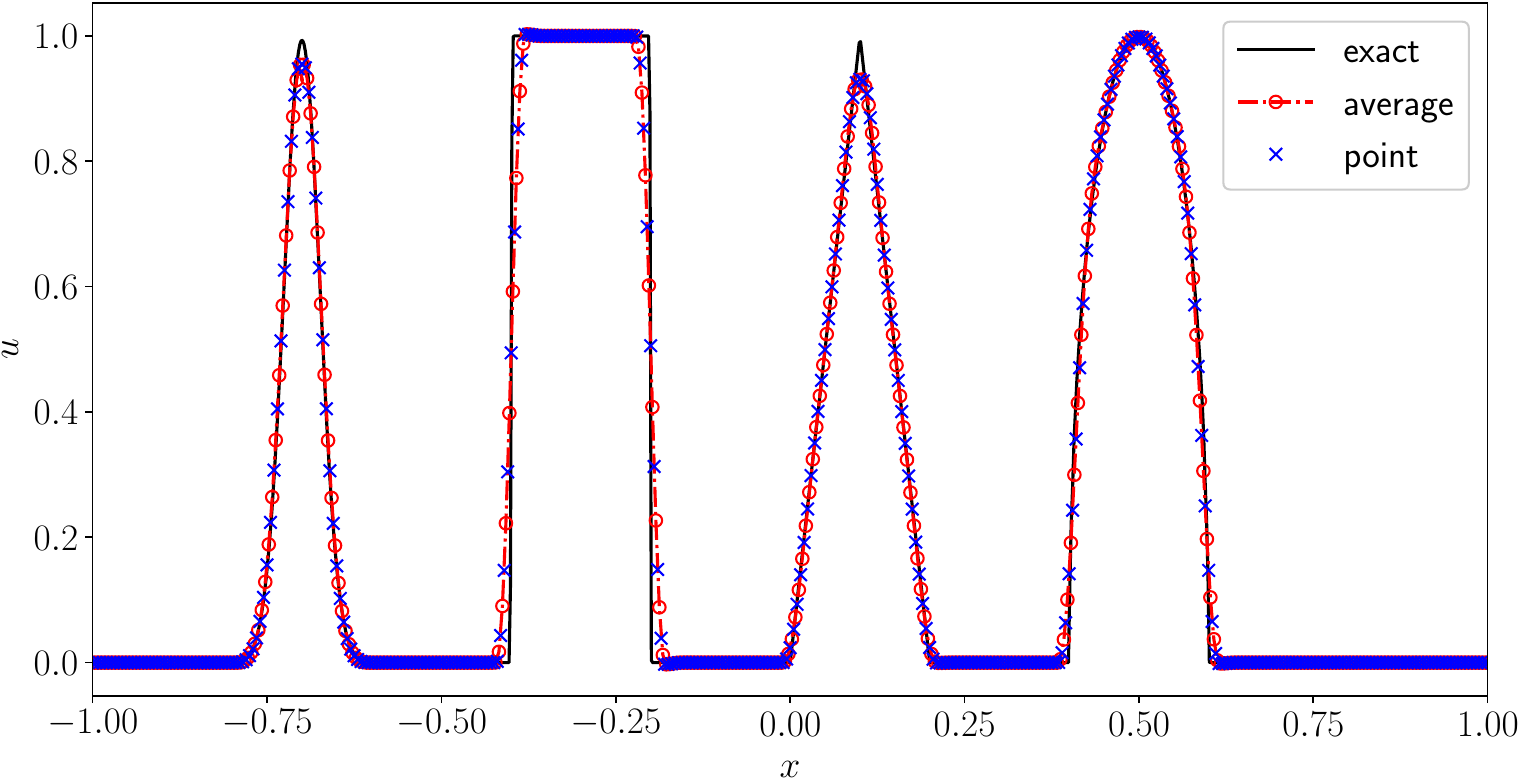}
	\end{subfigure}
	
	\begin{subfigure}[b]{0.48\textwidth}
		\centering
		\includegraphics[width=1.0\linewidth]{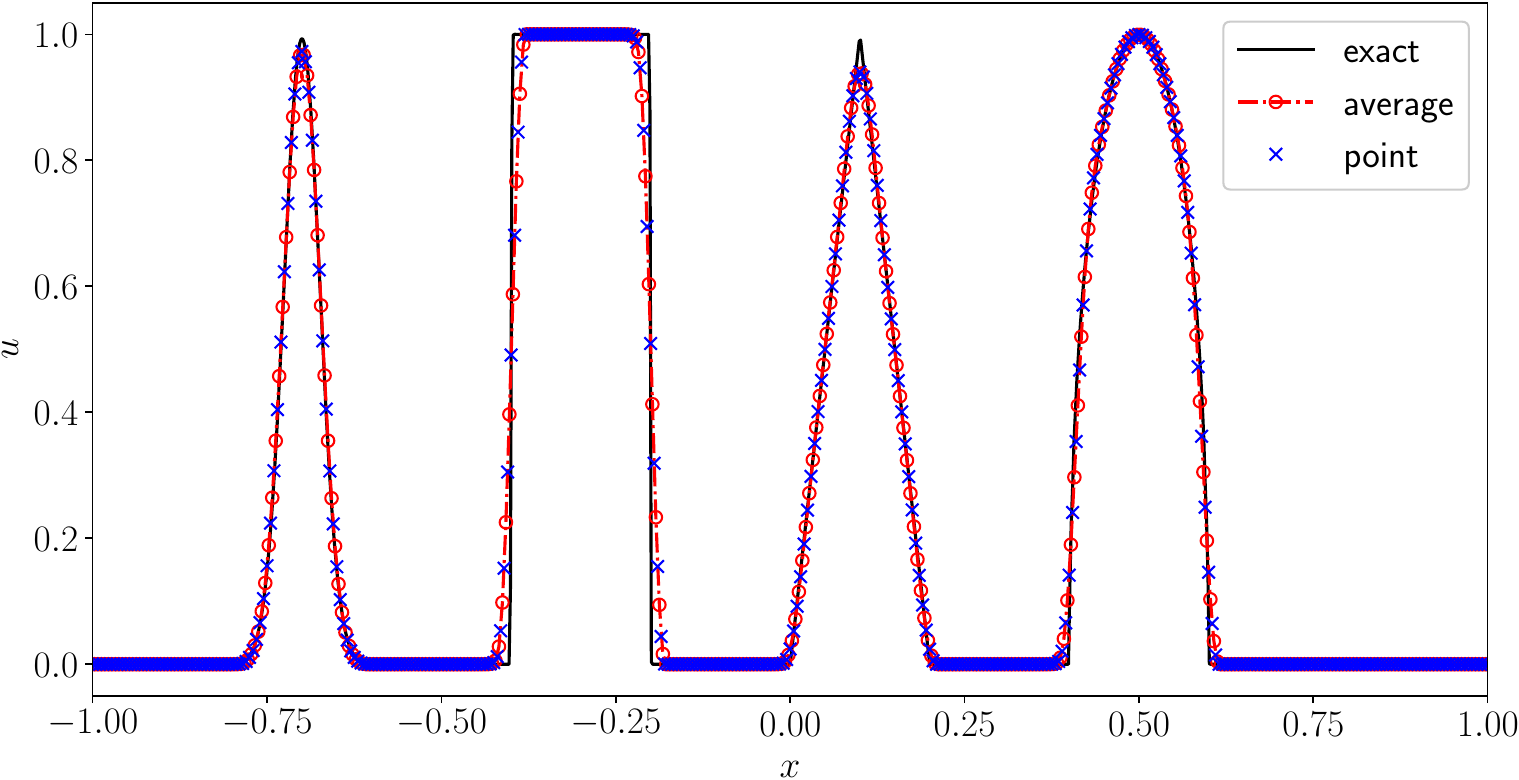}
	\end{subfigure}
	\begin{subfigure}[b]{0.48\textwidth}
		\centering
		\includegraphics[width=1.0\linewidth]{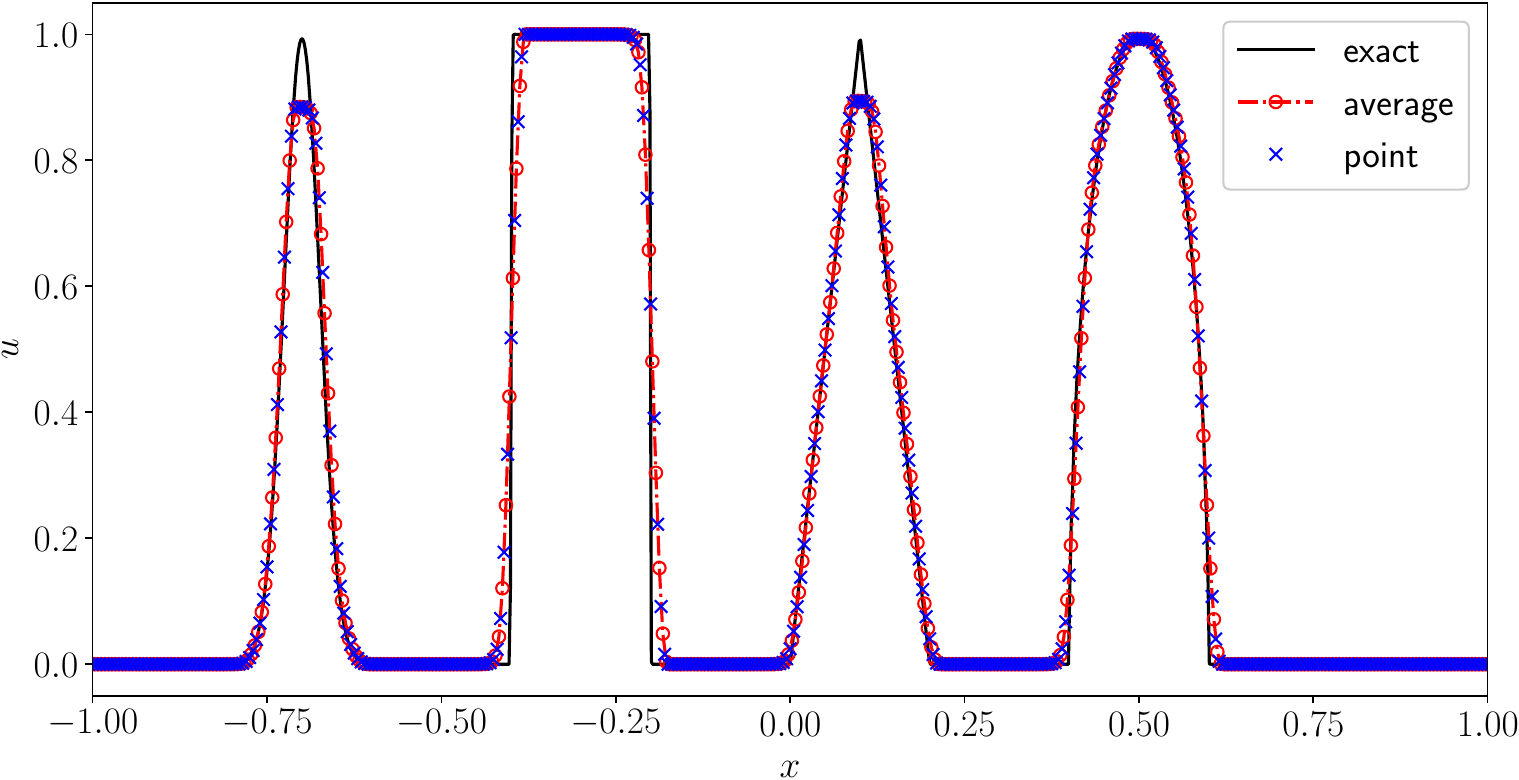}
	\end{subfigure}
	\caption{\Cref{ex:1d_advection_discontinuity}, advection.
		The results are obtained without any limiting (upper left),
		with power law reconstruction (upper right),
		with BP limitings imposing global MP for the cell average and point value (lower left),
		with BP limitings imposing local MP for the cell average and point value (lower right).}
	\label{fig:1d_advection_limiting}
\end{figure}

\begin{table}[htbp]
	\footnotesize\centering
	\begin{tabular}{cr}
		\hline\hline
		none & $[-\num{5.9e-02}, 1+\num{5.9e-02}]$\quad{\color{red}\xmark} \\ \hline
		PLR & $[-\num{2.7e-03}, 1+\num{2.6e-03}]$\quad{\color{red}\xmark} \\
		global MP for average & $[-\num{1.7e-03}, 1+\num{1.7e-03}]$\quad{\color{red}\xmark} \\
		local MP for average & $[-\num{1.3e-03}, 1+\num{1.3e-03}]$\quad{\color{red}\xmark} \\
		global MP for point & $[-\num{3.0e-04}, 1+\num{2.6e-04}]$\quad{\color{red}\xmark} \\
		local MP for point & $[-\num{4.6e-02}, 1+\num{4.6e-02}]$\quad{\color{red}\xmark} \\
  PLR + global MP for average & $[\num{-9.8e-06}, 1+\num{2.7e-06}]$\quad{\color{red}\xmark}  \\
  PLR + local MP for average & $[\num{-1.4e-05}, 1+\num{1.9e-05}]$\quad{\color{red}\xmark}  \\
  \hline
		global MP for average + global MP for point & $[\num{0.0}, 1.0]$\quad{\color{blue}\cmark}\\
		local MP for average + global MP for point & $[\num{0.0}, 1-\num{9.4e-13}]$\quad{\color{blue}\cmark} \\
		PLR + global MP for average + global MP for point & $[\num{0.0}, 1-\num{1.1e-16}]$\quad{\color{blue}\cmark} \\
		PLR + local MP for average + global MP for point & $[\num{0.0}, 1-\num{7.3e-14}]$\quad{\color{blue}\cmark} \\
		
		global MP for average + local MP for point & $[\num{0.0}, 1-\num{9.9e-15}]$\quad{\color{blue}\cmark}\\
		local MP for average + local MP for point & $[\num{0.0}, 1-\num{1.8e-12}]$\quad{\color{blue}\cmark} \\
		PLR + global MP for average + local MP for point & $[\num{0.0}, 1-\num{3.9e-15}]$\quad{\color{blue}\cmark} \\
		PLR + local MP for average + local MP for points & $[\num{0.0}, 1-\num{2.4e-13}]$\quad{\color{blue}\cmark} \\
		\hline\hline
	\end{tabular}
	\caption{\Cref{ex:1d_advection_discontinuity}, advection.
	The ranges of the numerical solutions (including both the cell averages and the point values) obtained with different limitings after one period.
	``PLR'' denotes the power law reconstruction.}
	\label{tab:1d_advection_bounds}
\end{table}

\begin{example}[Self-steepening shock]\label{ex:1d_burgers}\rm
	Consider the 1D Burgers' equation
$u_t + \left(\frac12 u^2\right)_x = 0$
	on the domain $[-1,1]$ with periodic boundary conditions.
	This test is solved until $T=0.5$ with the initial condition as a square wave
	\begin{equation*}
		u_0(x) = \begin{cases}
			2, &\text{if}~~ \abs{x} < 0.2, \\
			-1, &\text{otherwise}. \\
		\end{cases}
	\end{equation*}
\end{example}

\Cref{fig:1d_burgers_shock_js,fig:1d_burgers_shock_fvs} plot the cell averages and point values based on different point value updates with $200$ cells,
as well as the reference solution.
The spike generation has been observed in \cite{Helzel_2019_New_JoSC},
and the reason is also discussed in \Cref{sec:1d_fvs}.
Such spike generation cannot be eliminated by using the power law reconstruction,
nor do both BP limitings help to eliminate artefacts,
as can be seen from \cref{fig:1d_burgers_shock_js}.
The numerical solutions based on the LLF or SW FVS are shown in \cref{fig:1d_burgers_shock_fvs}, in which no spike appears.
There are some oscillations near the discontinuity without limitings,
and the numerical solutions agree well with the reference solution when the limitings are activated.

\begin{figure}[htbp]
	\centering
	\begin{subfigure}[b]{0.28\textwidth}
		\centering
		\includegraphics[width=1.0\linewidth]{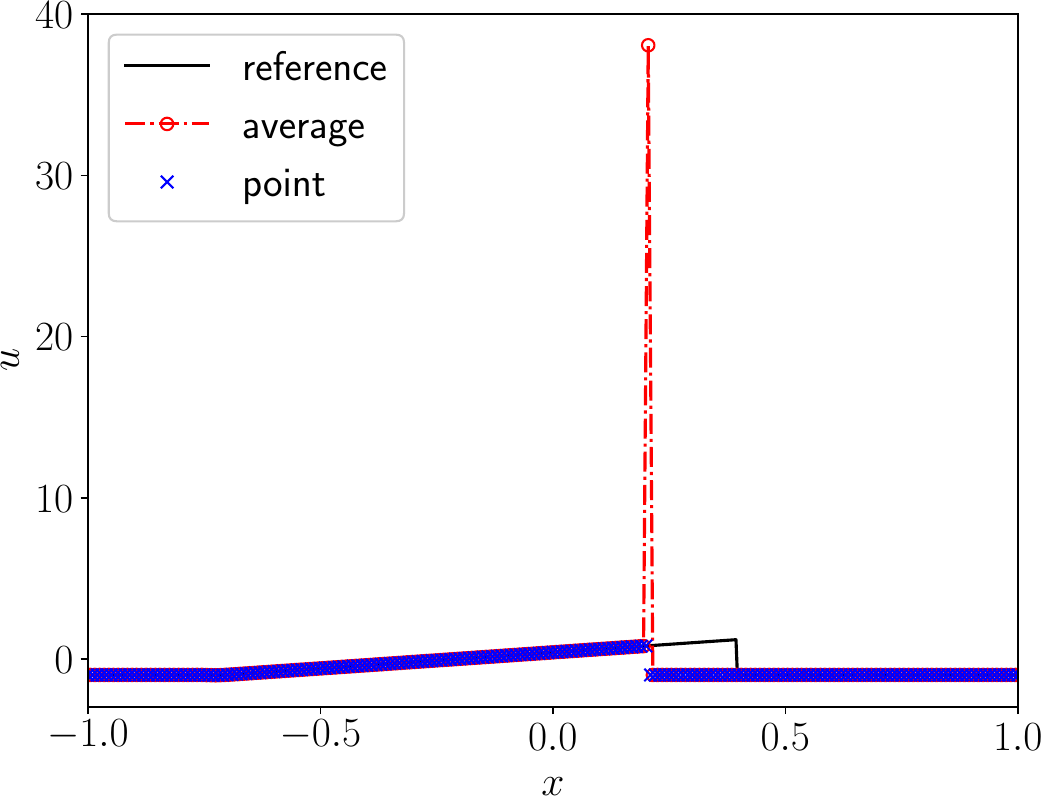}
	\end{subfigure}
	\begin{subfigure}[b]{0.28\textwidth}
		\centering
		\includegraphics[width=1.0\linewidth]{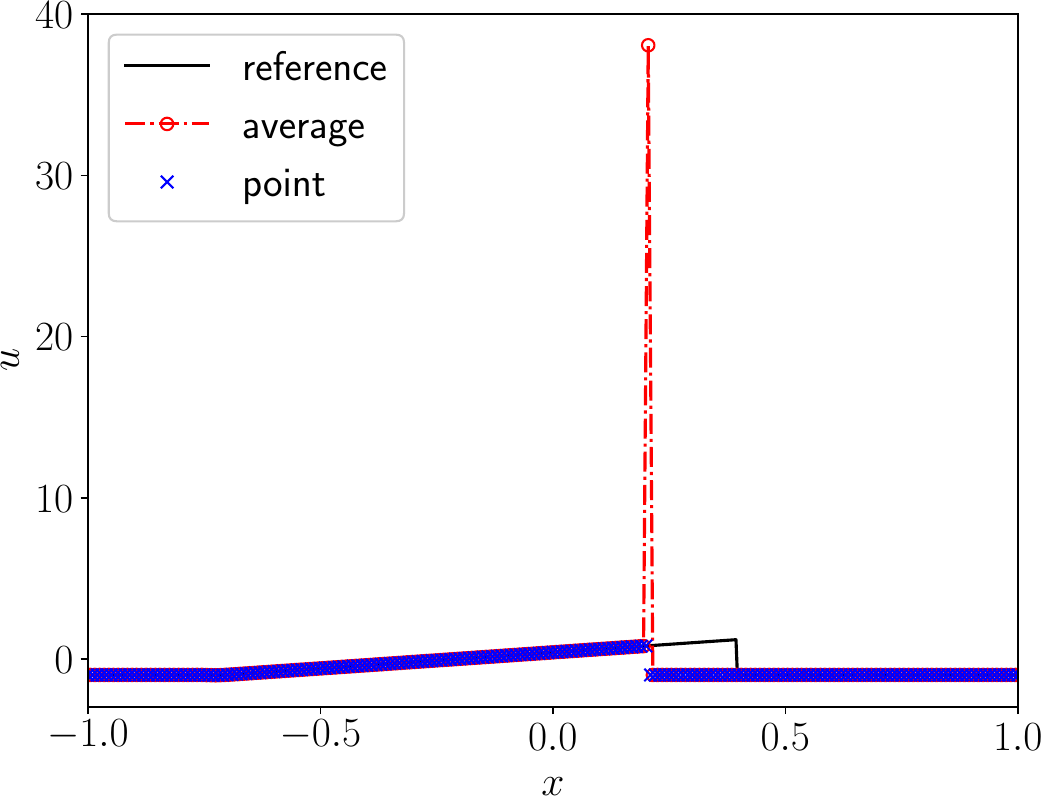}
	\end{subfigure}
	\begin{subfigure}[b]{0.29\textwidth}
		\centering
		\includegraphics[width=1.0\linewidth]{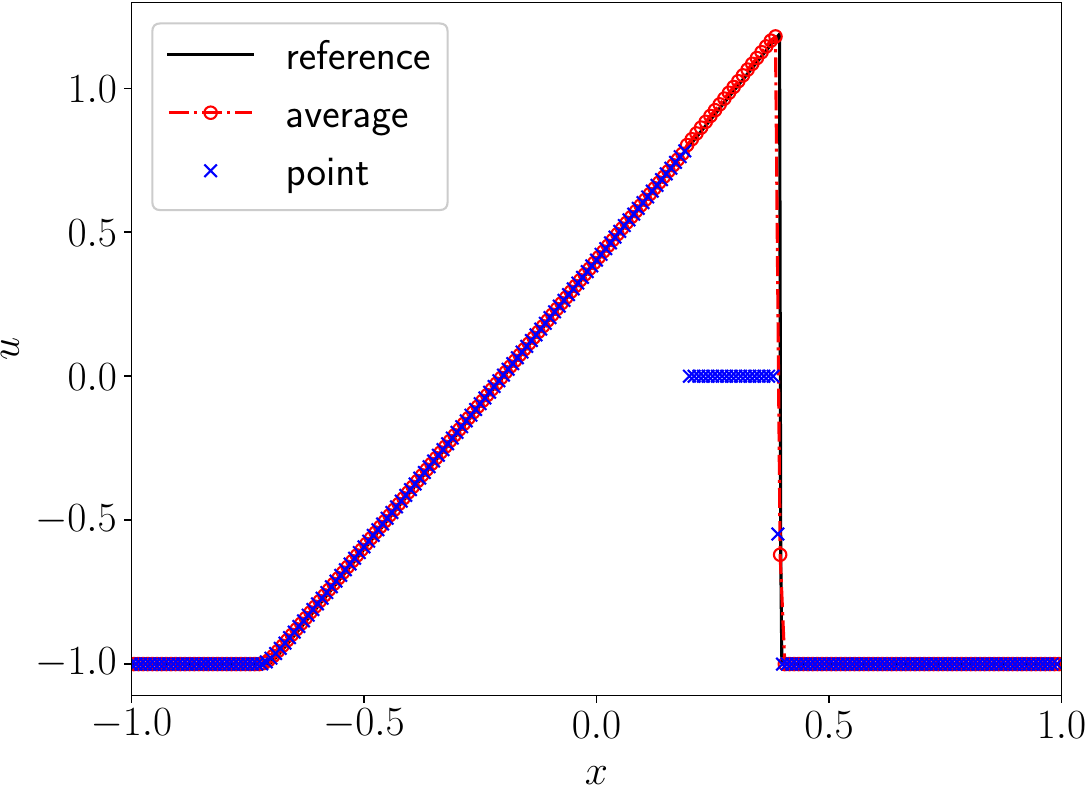}
	\end{subfigure}
	\caption{\Cref{ex:1d_burgers}, self-steepening shock for the Burgers' equation.
	The numerical solutions are based on the JS.
	From left to right: without limiting, with the power law reconstruction,
	with the BP limitings imposing local MP for the cell average and point value update, respectively.}
	\label{fig:1d_burgers_shock_js}
\end{figure}

\begin{figure}[htbp]
	\centering
	\begin{subfigure}[b]{0.24\textwidth}
		\centering
		\includegraphics[width=0.98\linewidth]{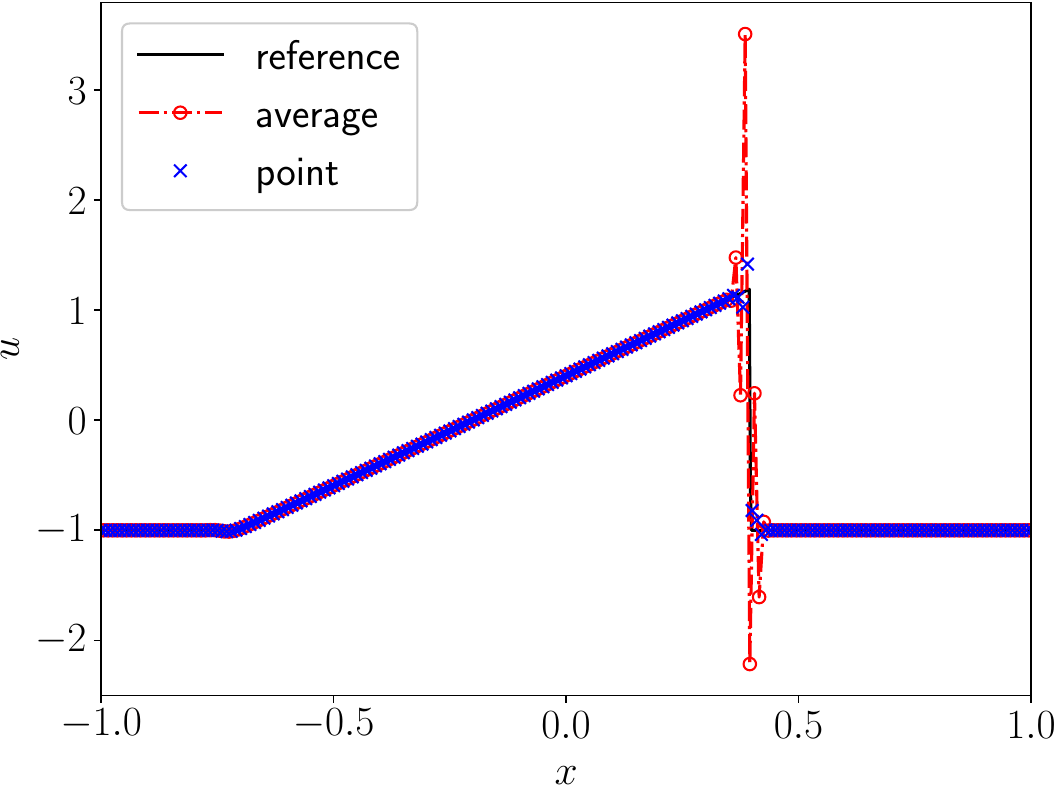}
	\end{subfigure}
	\begin{subfigure}[b]{0.24\textwidth}
		\centering
		\includegraphics[width=1.0\linewidth]{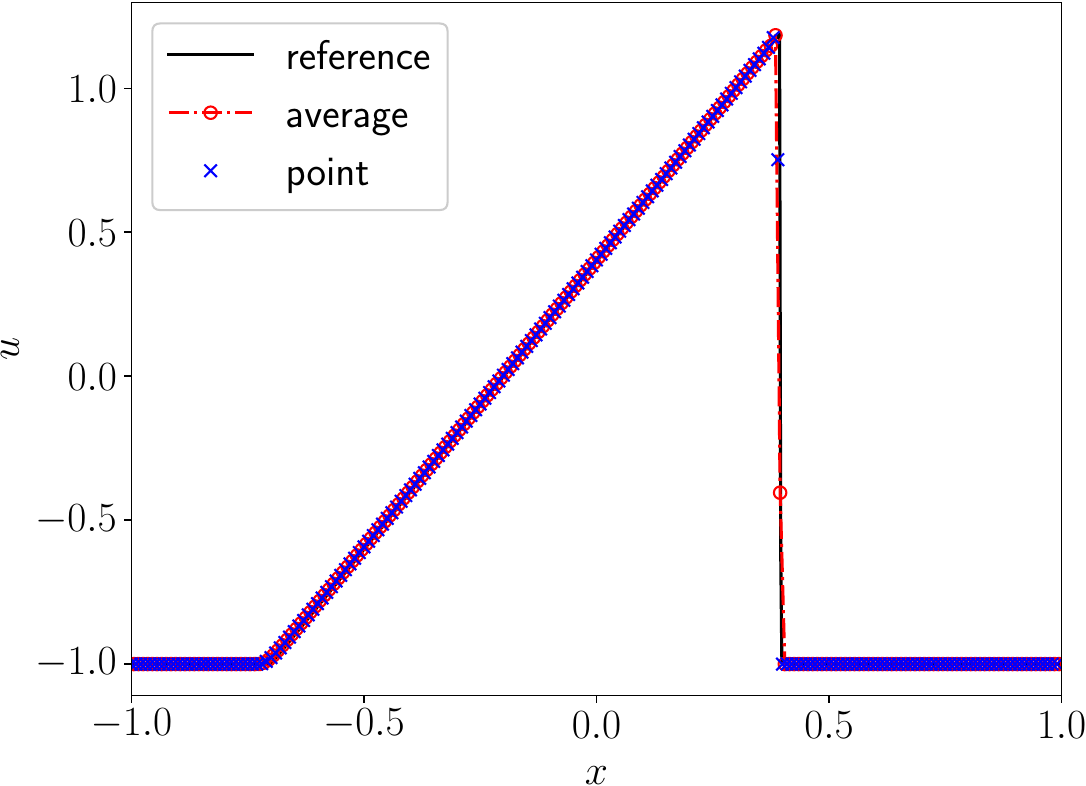}
	\end{subfigure}
	\begin{subfigure}[b]{0.24\textwidth}
		\centering
		\includegraphics[width=0.98\linewidth]{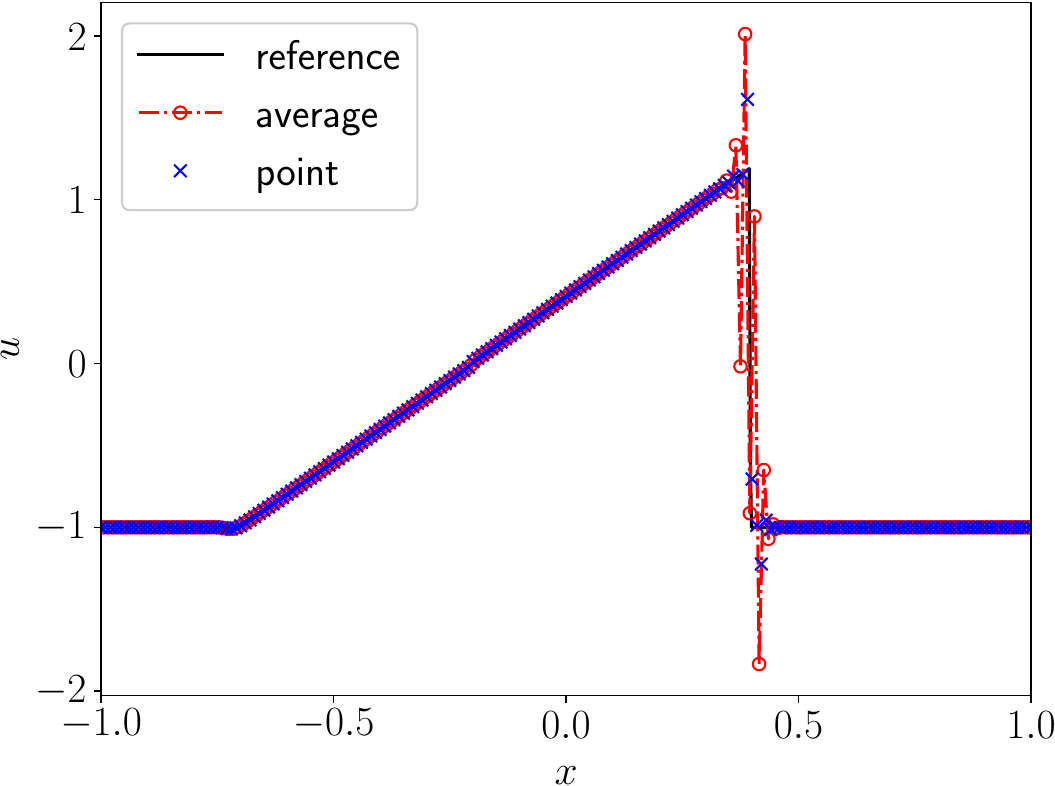}
	\end{subfigure}
	\begin{subfigure}[b]{0.24\textwidth}
		\centering
		\includegraphics[width=1.0\linewidth]{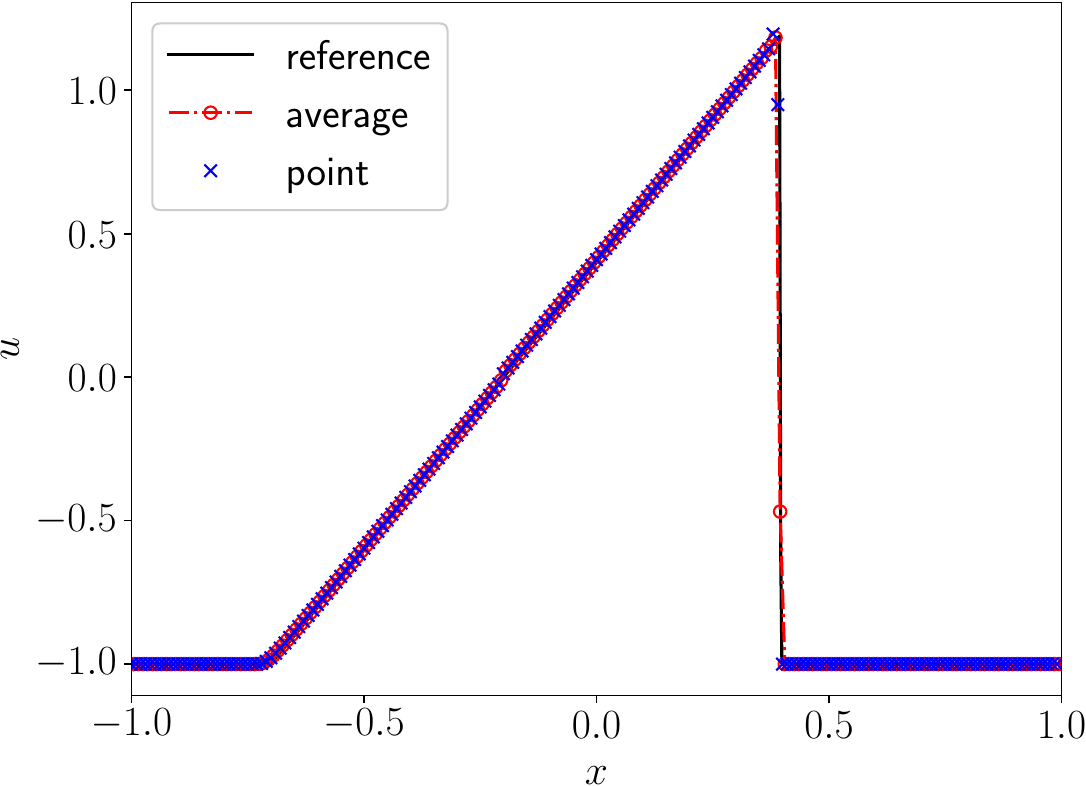}
	\end{subfigure}
	\caption{\Cref{ex:1d_burgers}, self-steepening shock for the Burgers' equation.
	From left to right:
	the LLF FVS without limiting, the LLF FVS with limitings, 
	the SW FVS without limiting, the SW FVS with limitings.
	The limitings consider the local MP for the cell average and point value updates, respectively.}
	\label{fig:1d_burgers_shock_fvs}
\end{figure}

\subsection{The compressible Euler equations}
This section shows some challenging tests, which require the BP property of the numerical methods in order to prevent simulations from crashing at some time.
The adiabatic index is $\gamma = 1.4$ unless otherwise stated.
Note that the BP limiting naturally reduces some oscillations.


\begin{example}[1D accuracy test for the Euler equations]\label{ex:1d_accuracy}\rm
	This test is used to examine the accuracy of using different point value updates, following the setup in \cite{Abgrall_2023_Combination_CoAMaC}.
	The domain is $[-1,1]$ with periodic boundary conditions.
	The adiabatic index is chosen as $\gamma=3$ so that the characteristic equations of two Riemann invariants $w=u\pm a$ are $w_t+ww_x=0$.
	The initial condition is $\rho_0(x) = 1+\zeta \sin(\pi x), v_0 = 0, p_0 = \rho_0^\gamma$
	and $\zeta\in(0, 1)$ controls the range of the density.
	The exact solution can be obtained by the method of characteristics, given by $\rho(x,t)=\frac12\left(\rho_0(x_1) + \rho_0(x_2)\right), v(x,t)=\sqrt{3}\left(\rho(x,t)-\rho_0(x_1)\right)$,
	where $x_1$ and $x_2$ are solved from the nonlinear equations $x+\sqrt{3}\rho_0(x_1)t-x_1=0$, $x-\sqrt{3}\rho_0(x_2)t-x_2=0$.
	The problem is solved until $T = 0.1$ with $\zeta=1-10^{-7}$.

As $\zeta=1-10^{-7}$, the minimum density and pressure are $10^{-7}$ and $10^{-21}$ respectively, so that the BP limitings are necessary to run this test case.
The maximal CFL numbers allowing stable simulations are obtained experimentally, which are around $0.47, 0.43, 0.32, 0.18$ for the JS, LLF, SW, VH FVS, respectively,
thus we run the test with the same CFL number as $0.18$.
\Cref{fig:1d_euler_accuracy} shows the errors and corresponding convergence rates for the conservative variables in the $\ell^1$ norm.
It is seen that the JS and all the FVS except for the SW FVS achieve the designed third-order accuracy, showing that our BP limitings do not affect the high-order accuracy.
To examine the reason why the scheme based on the SW FVS is only second-order accurate,
\Cref{fig:1d_euler_accuracy_test1} plots the density and velocity profiles obtained by the SW FVS with $80$ cells.
One can observe some defects in the density when the velocity is zero,
similar to the ``sonic point glitch'' in the literature \cite{Tang_2005_sonic_JoCP}.
One possible reason is that the SW FVS is based on the absolute value of the eigenvalues, and the corresponding mass flux is not differentiable when the velocity is zero \cite{Leer_1982_Flux_InProceedings}.
Such an issue remains to be further explored in the future.

\begin{figure}[htbp]
	\centering
	\includegraphics[width=0.5\linewidth]{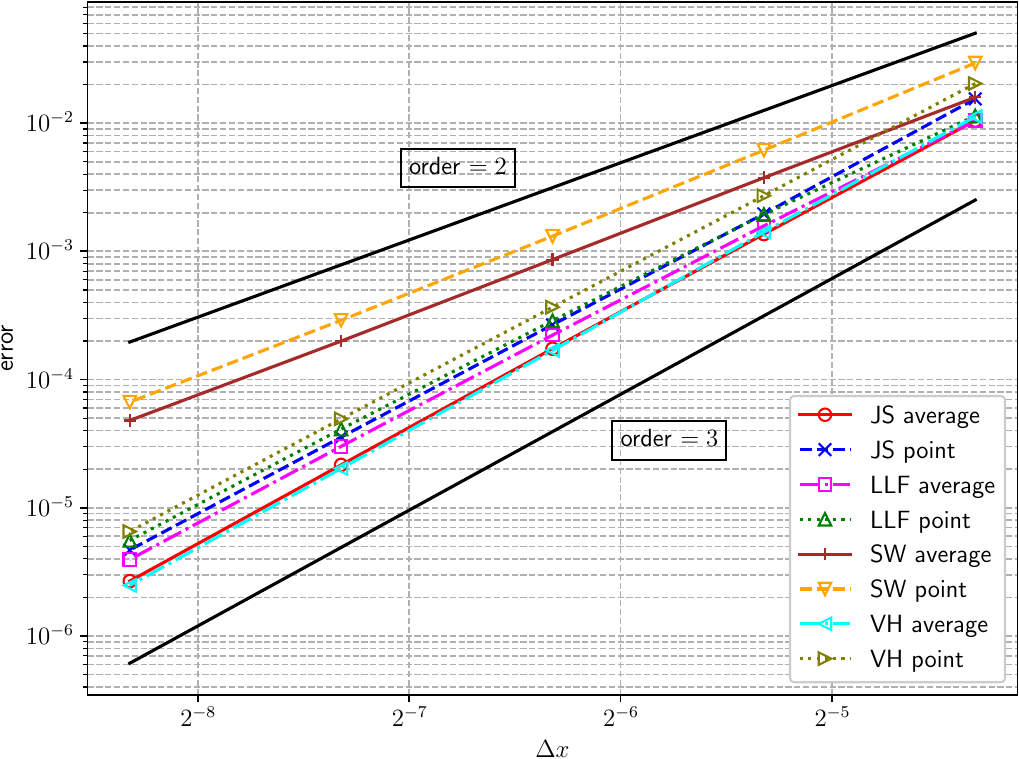}
	\caption{\Cref{ex:1d_accuracy}, the accuracy test for the 1D Euler equations.
	}
	\label{fig:1d_euler_accuracy}
\end{figure}

\begin{figure}[htbp]
	\centering
	\begin{subfigure}[b]{0.35\textwidth}
		\centering
		\includegraphics[width=1.0\linewidth]{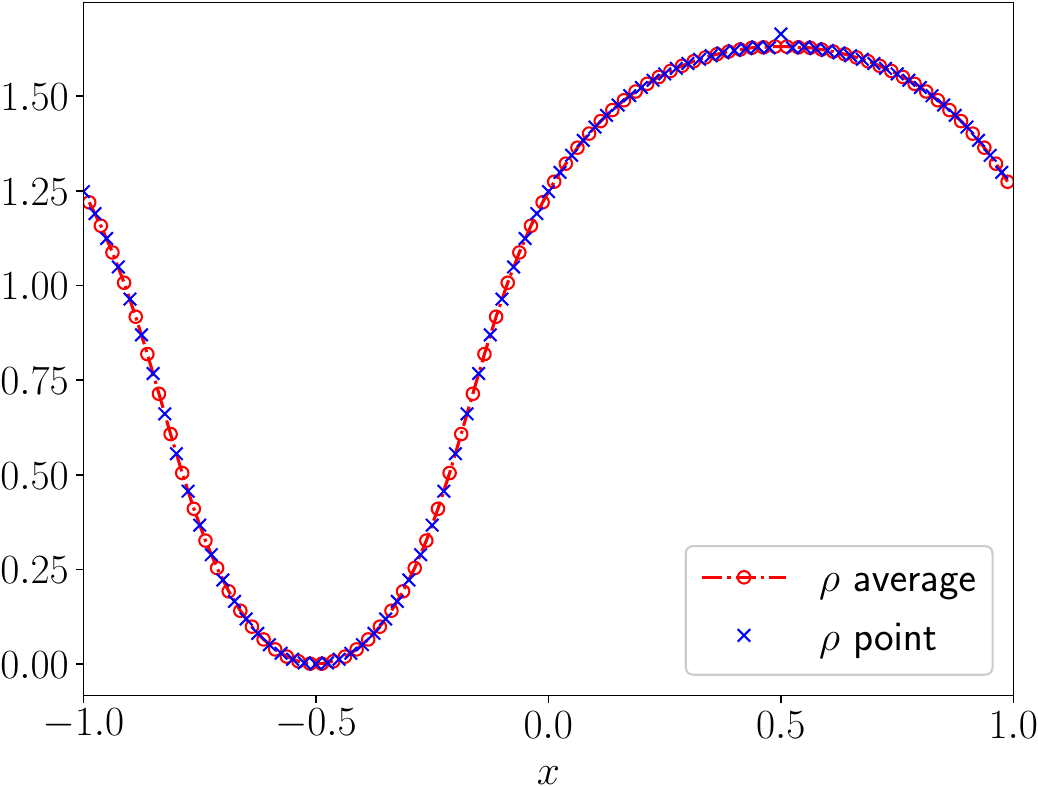}
	\end{subfigure}
	~
	\begin{subfigure}[b]{0.35\textwidth}
		\centering
		\includegraphics[width=1.0\linewidth]{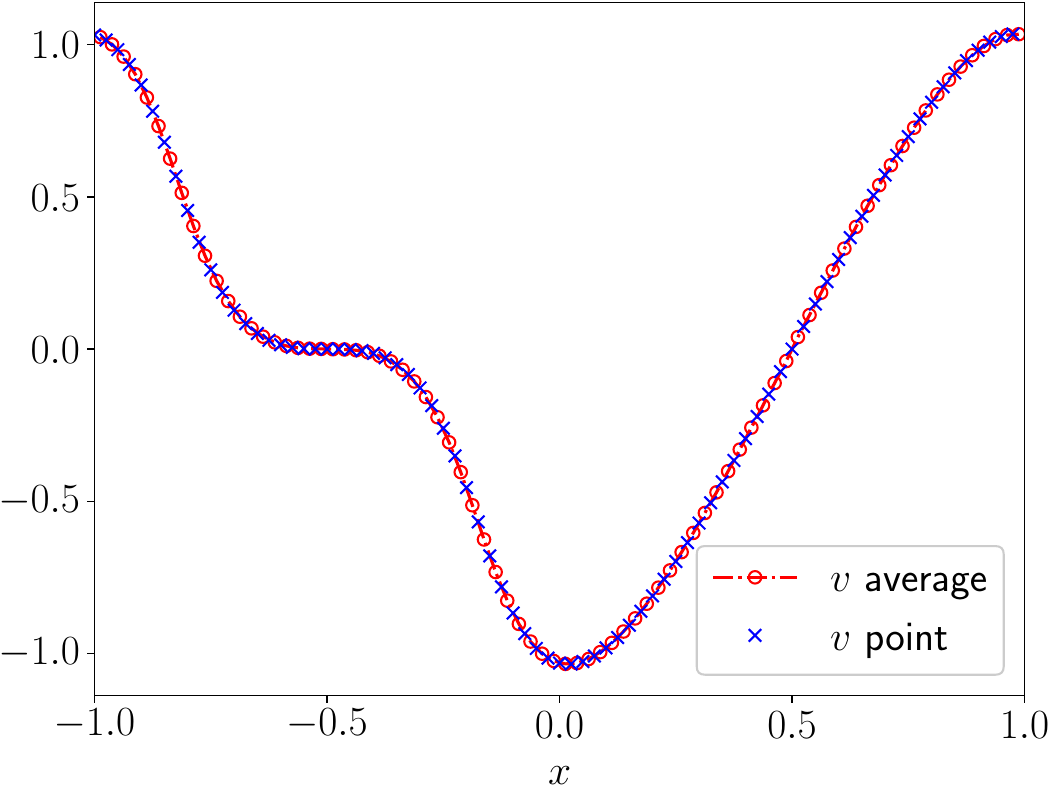}
	\end{subfigure}
	\caption{\Cref{ex:1d_accuracy}, the density (left) and velocity (right) are obtained with the SW FVS and $80$ cells for the 1D Euler equations.
	}
	\label{fig:1d_euler_accuracy_test1}
\end{figure}
\end{example}


\begin{example}[Double rarefaction problem]\label{ex:1d_double_rarefaction}\rm
	The exact solution to this problem contains a vacuum,
	so that it is often used to verify the BP property of numerical methods.
	The test is solved on a domain $[0,1]$ until $T=0.3$ with the initial data	\begin{equation*}
		(\rho, v, p) = \begin{cases}
			(7, -1, 0.2), &\text{if}~~ x<0.5,\\
			(7, 1, 0.2), &\text{otherwise}.\\
		\end{cases}
	\end{equation*}

In this test, the AF method based on any kind of point value update mentioned in this paper gives negative density or pressure without the BP limitings.
\Cref{fig:1d_double_rarefaction_rho} shows the density computed with $400$ cells and the BP limitings for the cell average and point value updates.
The power law reconstruction is not used in this test,
and the CFL number is $0.4$ for all kinds of point value updates,
except for $0.1$ for the VH FVS.
One observes that the BP AF method gets good performance for this example.

\begin{figure}[htbp]
	\centering
	\begin{subfigure}[b]{0.24\textwidth}
		\centering
		\includegraphics[width=\linewidth]{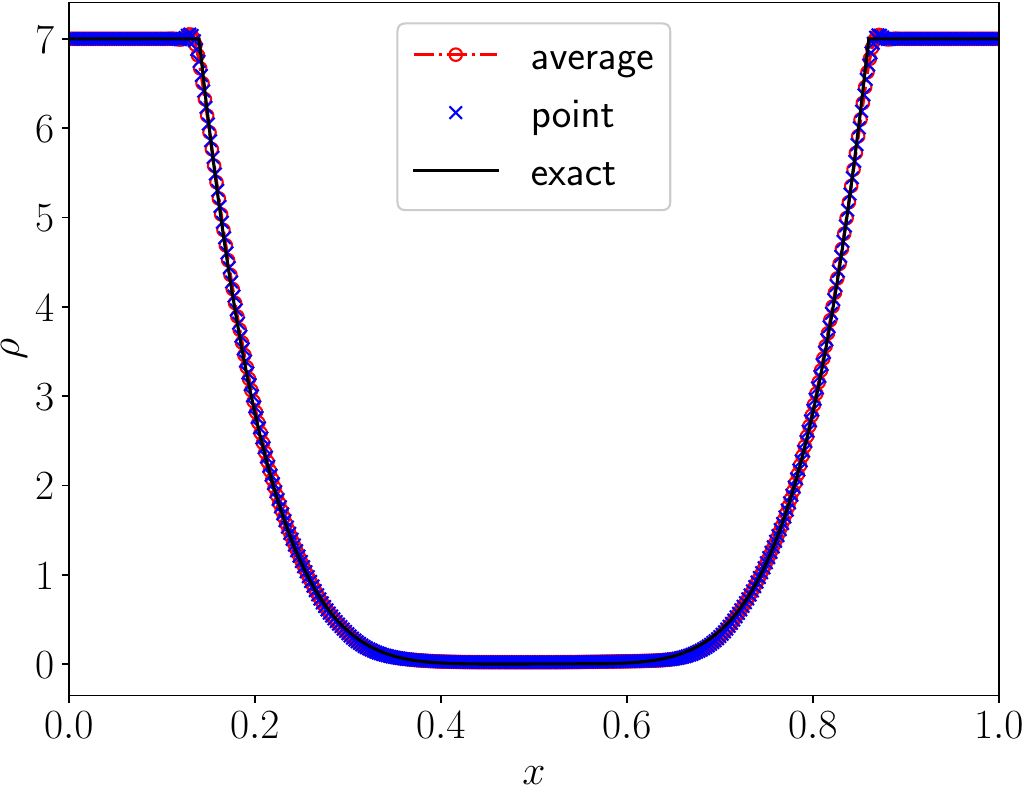}
	\end{subfigure}
	\begin{subfigure}[b]{0.24\textwidth}
		\centering
		\includegraphics[width=\linewidth]{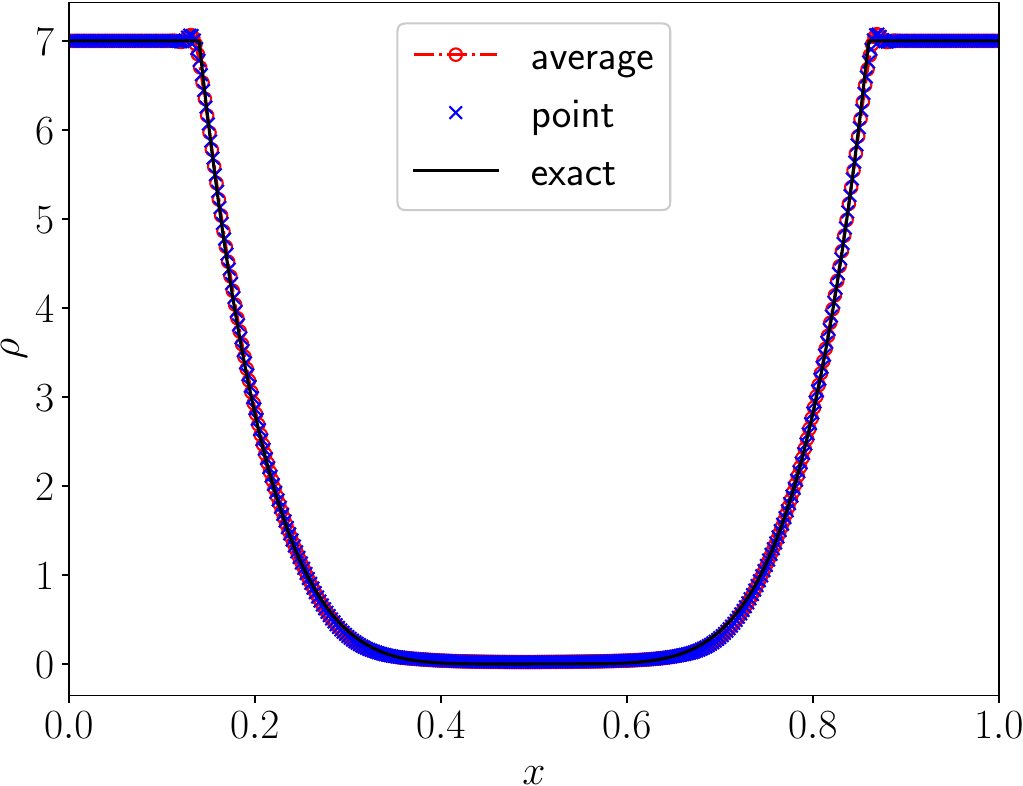}
	\end{subfigure}
	\begin{subfigure}[b]{0.24\textwidth}
		\centering
		\includegraphics[width=\linewidth]{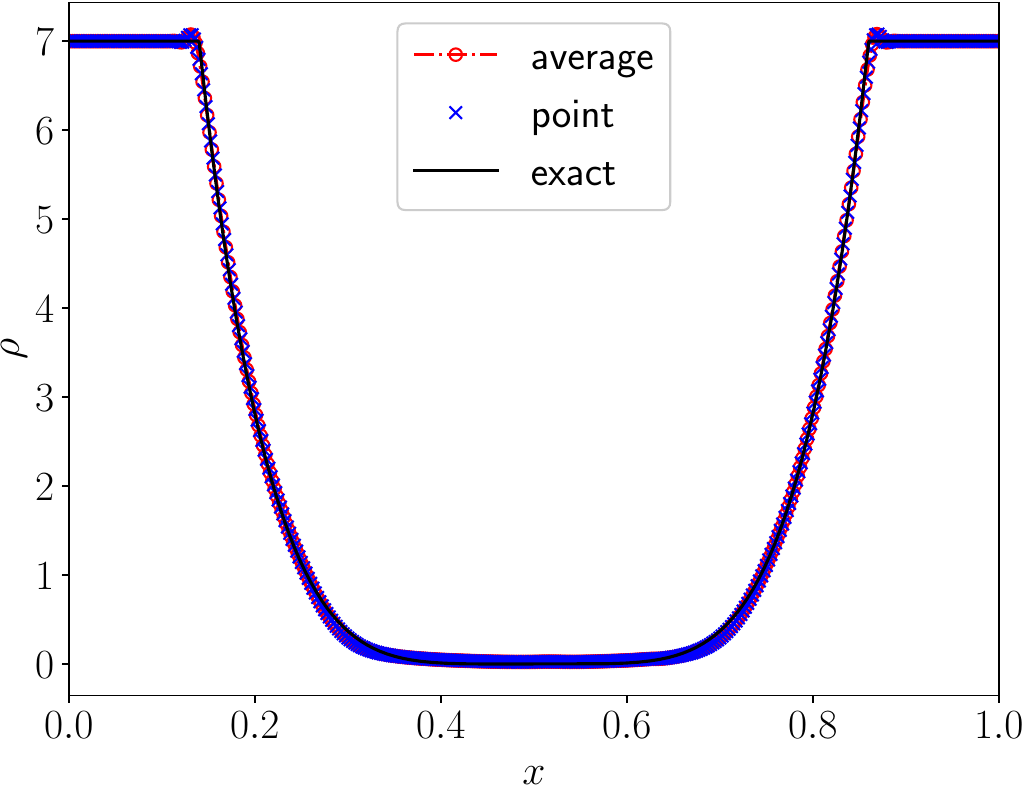}
	\end{subfigure}
	\begin{subfigure}[b]{0.24\textwidth}
		\centering
		\includegraphics[width=\linewidth]{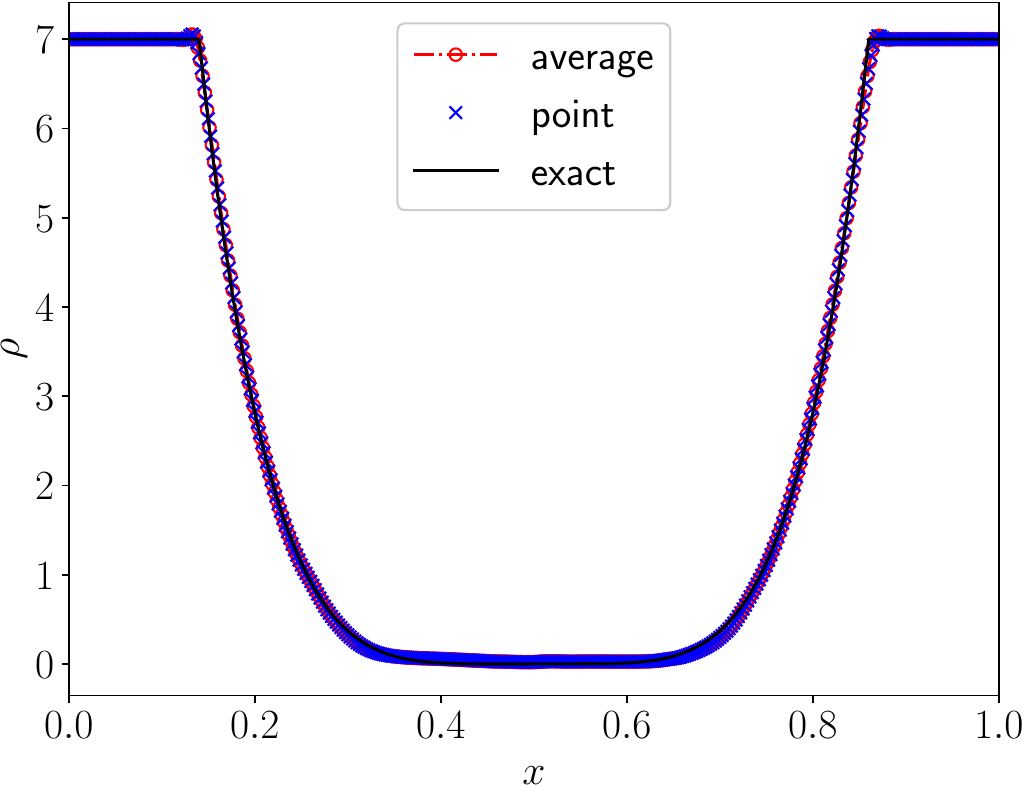}
	\end{subfigure}
	\caption{\Cref{ex:1d_double_rarefaction}, double rarefaction Riemann problem.
		The numerical solutions are computed with BP limitings for the cell average and point value updates on a uniform mesh of $400$ cells.
		The power law reconstruction is not used.
		From left to right: JS, LLF, SW, and VH FVS.}
	\label{fig:1d_double_rarefaction_rho}
\end{figure}
\end{example}

\begin{example}[LeBlanc shock tube]\label{ex:1d_leblanc}\rm
	This is a Riemann problem with an extremely large initial pressure ratio.
	This test is solved until $T=5\times 10^{-6}$ on a domain $[0,1]$ with the initial data
	\begin{equation*}
		(\rho, v, p) = \begin{cases}
			(2, 0, 10^9), &\text{if}~~ x<0.5,\\
			(10^{-3}, 0, 1), &\text{otherwise}.\\
		\end{cases}
	\end{equation*}

Without the BP limitings, the simulation will stop due to negative density or pressure.
\Cref{fig:1d_leblanc_rho} shows the density computed on a uniform mesh of $400$ and $6000$ cells with the BP limitings for the cell average and point value updates.
The CFL number is $0.4$ for the LLF and SW FVS,
and $0.15$ for the JS and VH FVS for stability when the power law reconstruction is not used.
It is seen that the numerical solutions on the coarse mesh deviate from the exact solutions,
which has also been observed in other high-order BP methods, e.g., \cite{Zhang_2010_positivity_JoCP}.
As the number of the mesh cells increases from $400$ to $6000$ (most of, if not all, the numerical methods need a small mesh size to obtain the right shock wave location, not only our AF method),
one can observe from \cref{fig:1d_leblanc_rho} that the numerical solutions converge to the exact solutions with only a few overshoots/undershoots at the contact discontinuity.
The LLF and SW FVS give better results.

To verify whether the power law reconstruction can suppress spurious oscillations and overshoots/undershoots,
we rerun the test with the CFL number $0.1$,
and the density profiles are shown in \cref{fig:1d_leblanc_fine_mesh_cfl0.1}.
It is obvious that only reducing the CFL number does not change the numerical solutions much except that the oscillations near the contact discontinuity based on the VH FVS are damped.
When the power law reconstruction is activated, the overshoots/undershoots are reduced for the JS, LLF, and SW FVS,
while the VH FVS gives worse results even with a smaller CFL number (e.g. $0.02$, not shown here), which needs further investigation.

\begin{figure}[htbp]
	\centering
	\begin{subfigure}[b]{0.24\textwidth}
		\centering
		\includegraphics[width=\linewidth]{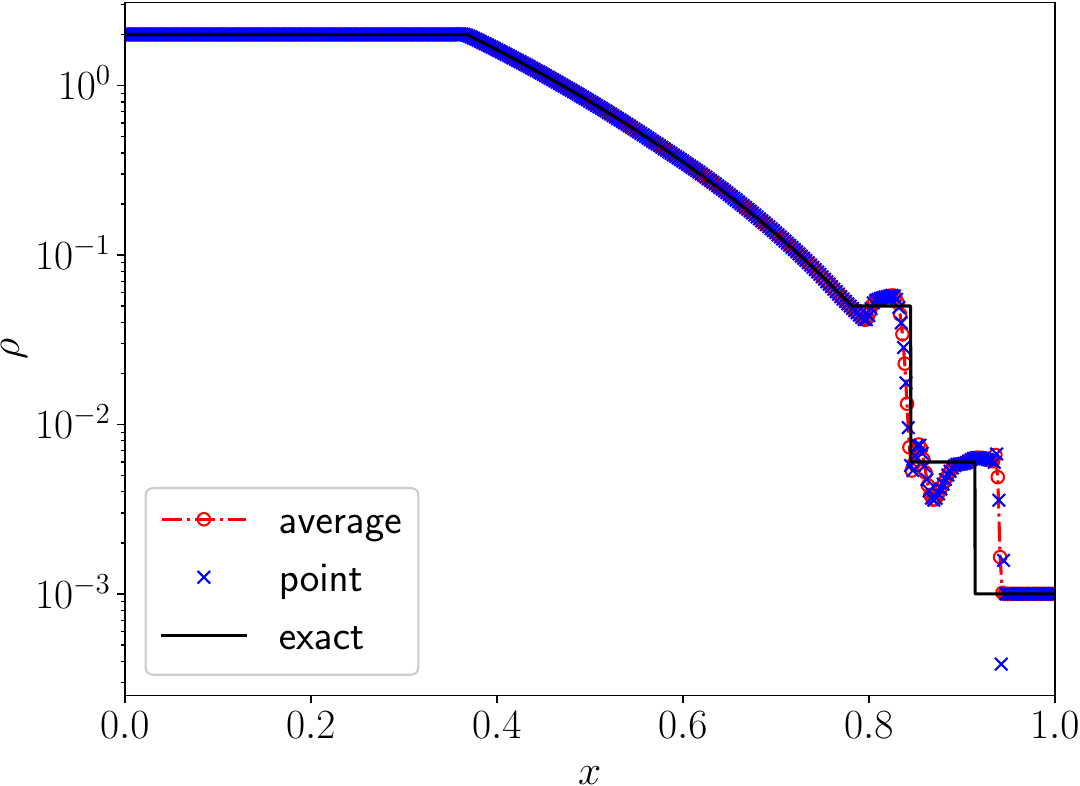}
	\end{subfigure}	
	\begin{subfigure}[b]{0.24\textwidth}
		\centering
		\includegraphics[width=\linewidth]{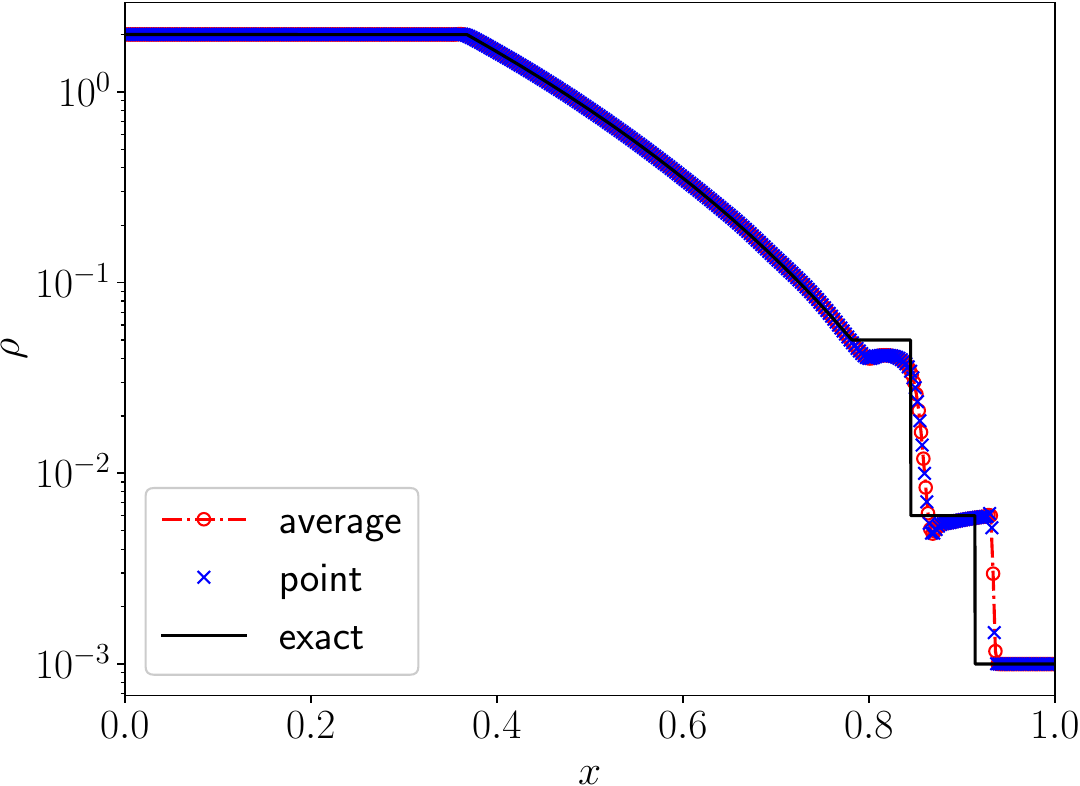}
	\end{subfigure}
	\begin{subfigure}[b]{0.24\textwidth}
		\centering
		\includegraphics[width=\linewidth]{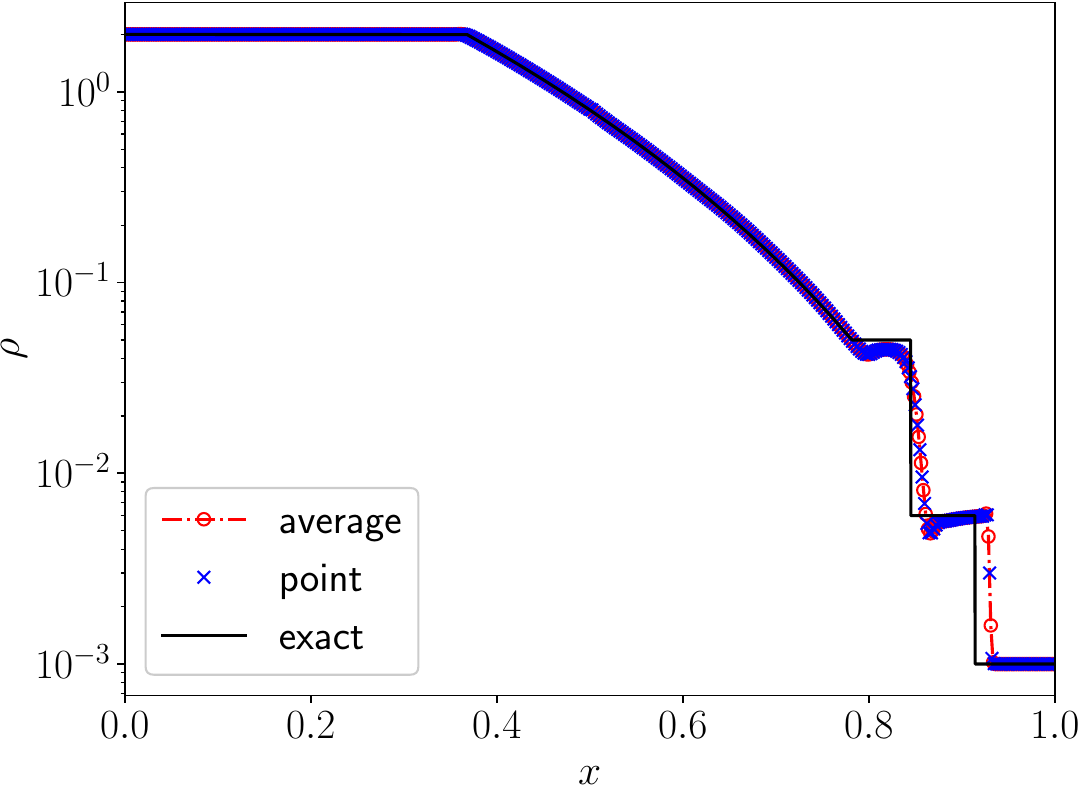}
	\end{subfigure}
	\begin{subfigure}[b]{0.24\textwidth}
		\centering
		\includegraphics[width=\linewidth]{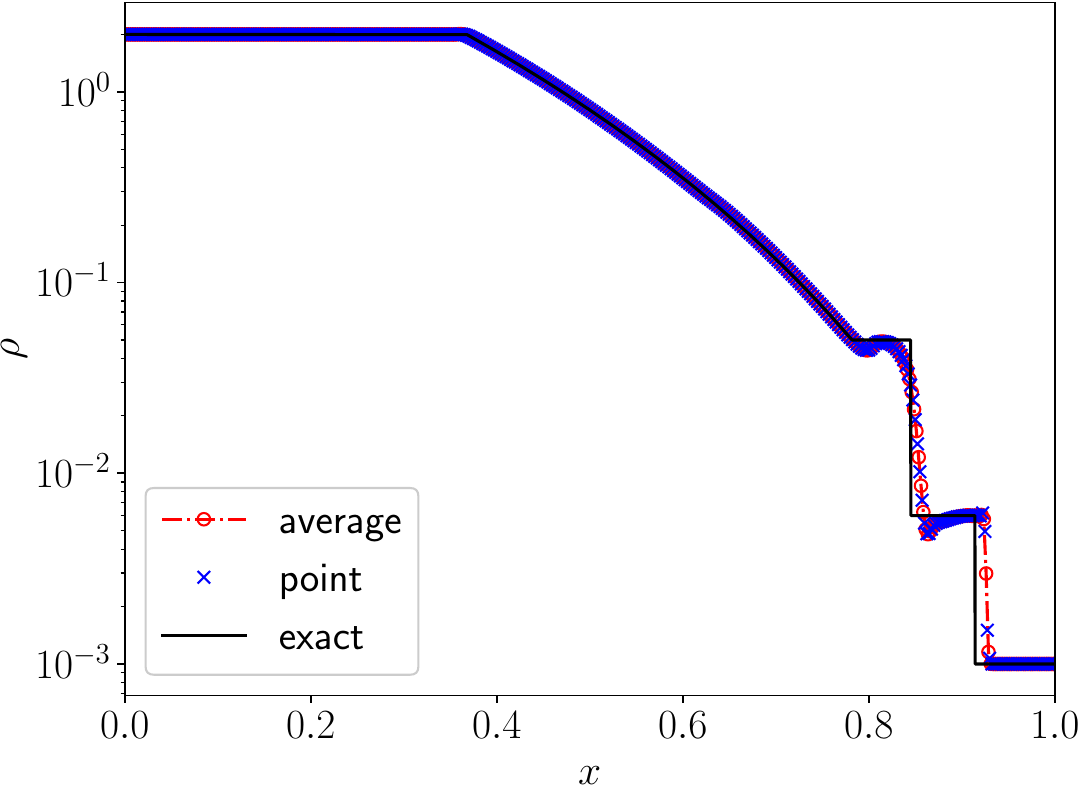}
	\end{subfigure}
	
	\begin{subfigure}[b]{0.24\textwidth}
		\centering
		\includegraphics[width=\linewidth]{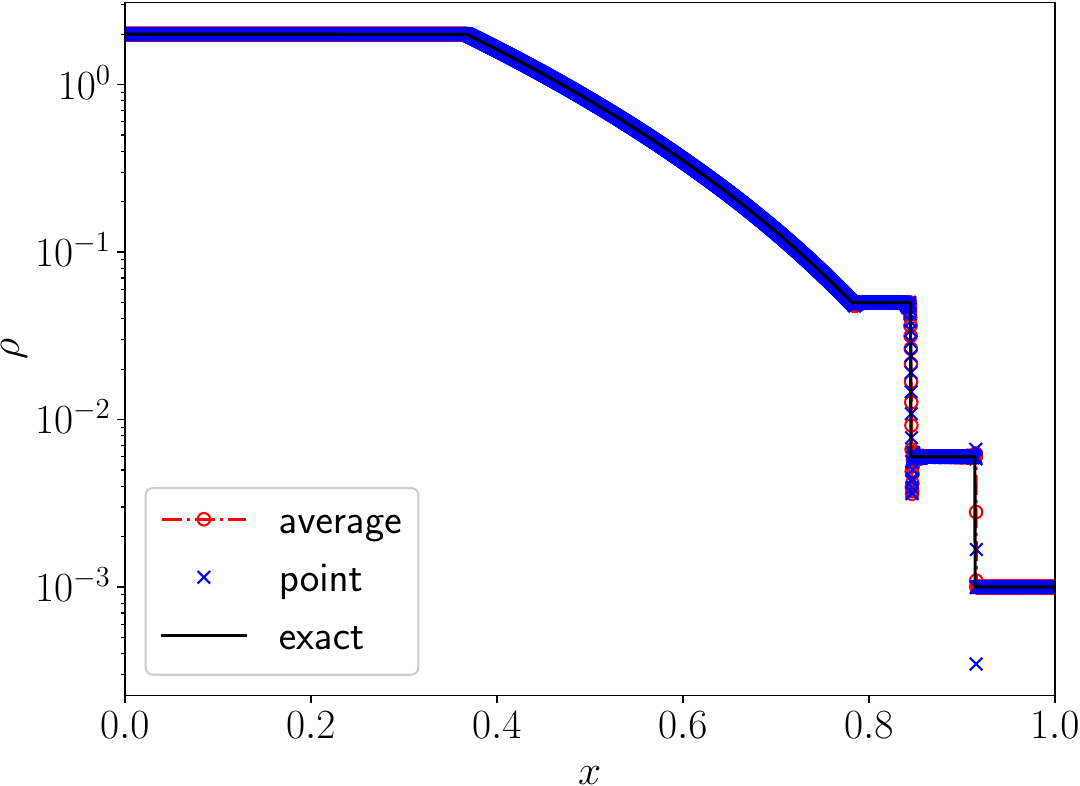}
	\end{subfigure}
	\begin{subfigure}[b]{0.24\textwidth}
		\centering
		\includegraphics[width=\linewidth]{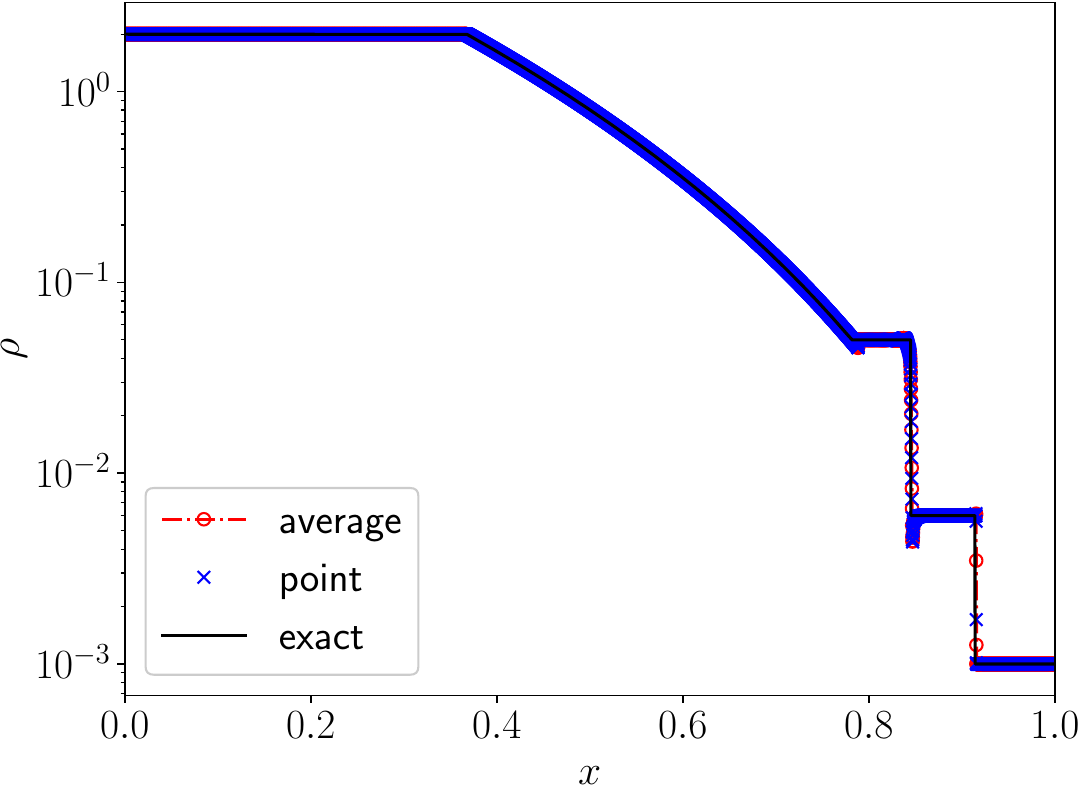}
	\end{subfigure}
	\begin{subfigure}[b]{0.24\textwidth}
		\centering
		\includegraphics[width=\linewidth]{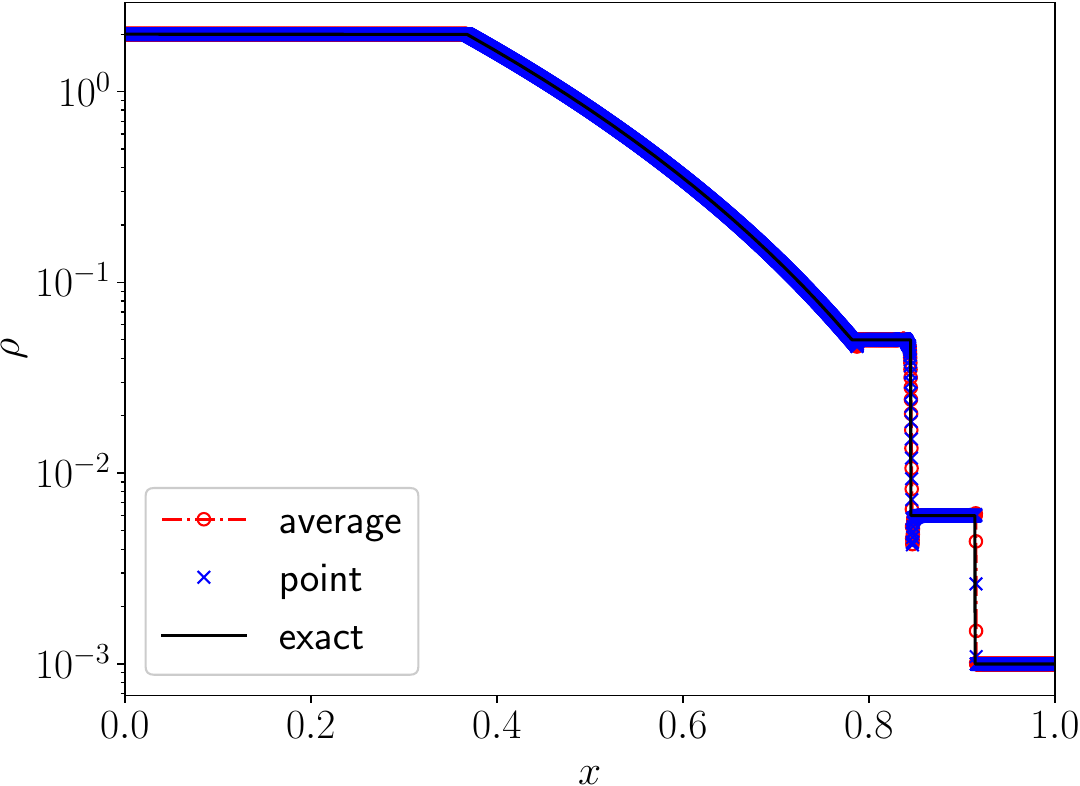}
	\end{subfigure}
	\begin{subfigure}[b]{0.24\textwidth}
		\centering
		\includegraphics[width=\linewidth]{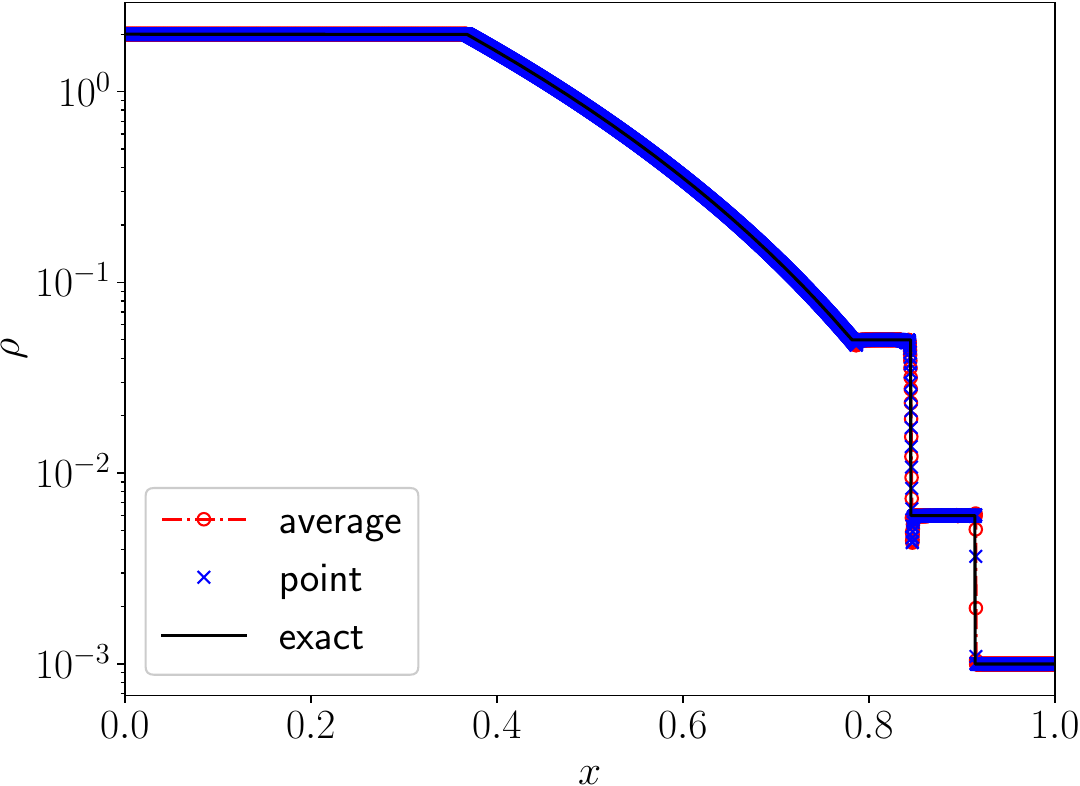}
	\end{subfigure}
	\caption{\Cref{ex:1d_leblanc}, LeBlanc Riemann problem.
		The numerical solutions are computed with the BP limitings for the cell average and point value updates on a uniform mesh of $400$ cells (top) and $6000$ cells (bottom).
		From left to right: JS, LLF, SW, and VH FVS.}
	\label{fig:1d_leblanc_rho}
\end{figure}

\begin{figure}[htbp]
	\centering
	\begin{subfigure}[b]{0.24\textwidth}
		\centering
		\includegraphics[width=\linewidth]{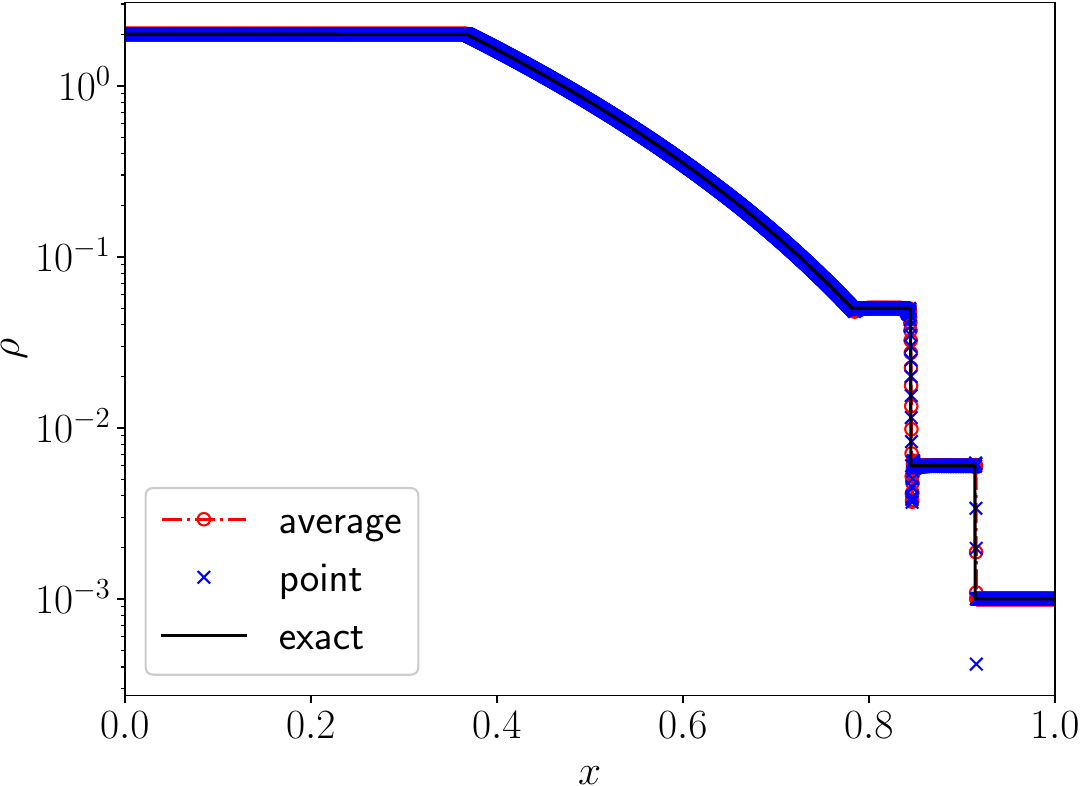}
	\end{subfigure}
	\begin{subfigure}[b]{0.24\textwidth}
		\centering
		\includegraphics[width=\linewidth]{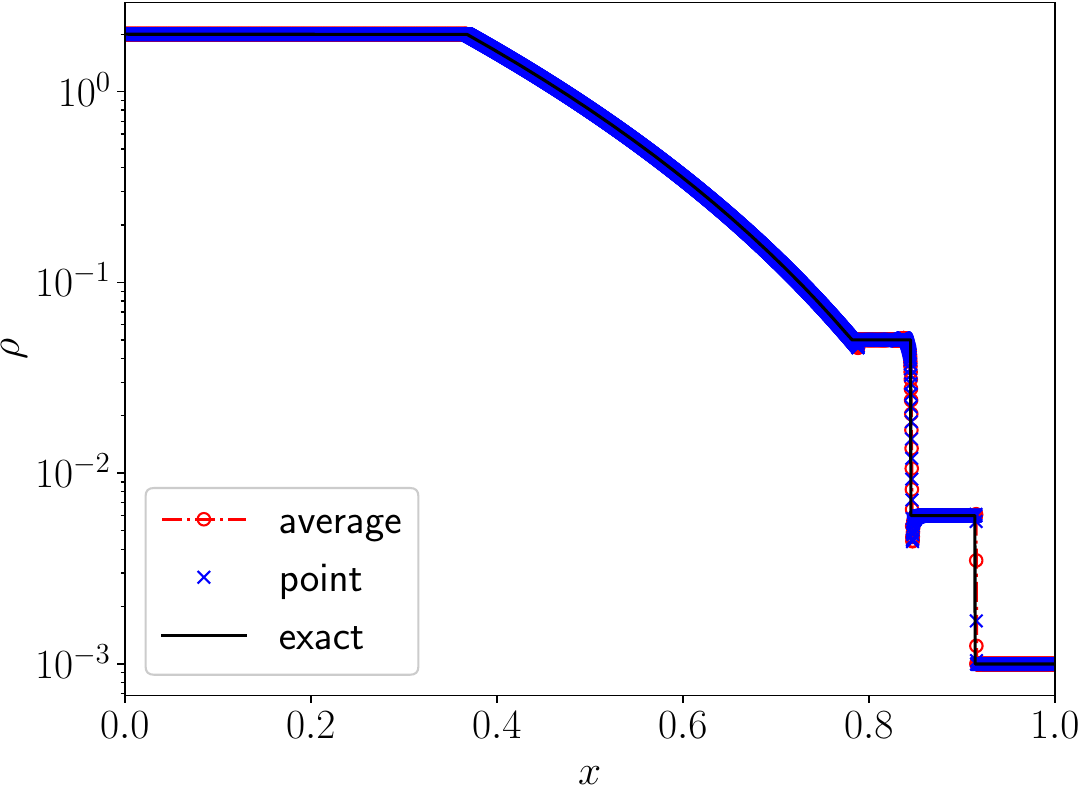}
	\end{subfigure}
	\begin{subfigure}[b]{0.24\textwidth}
		\centering
		\includegraphics[width=\linewidth]{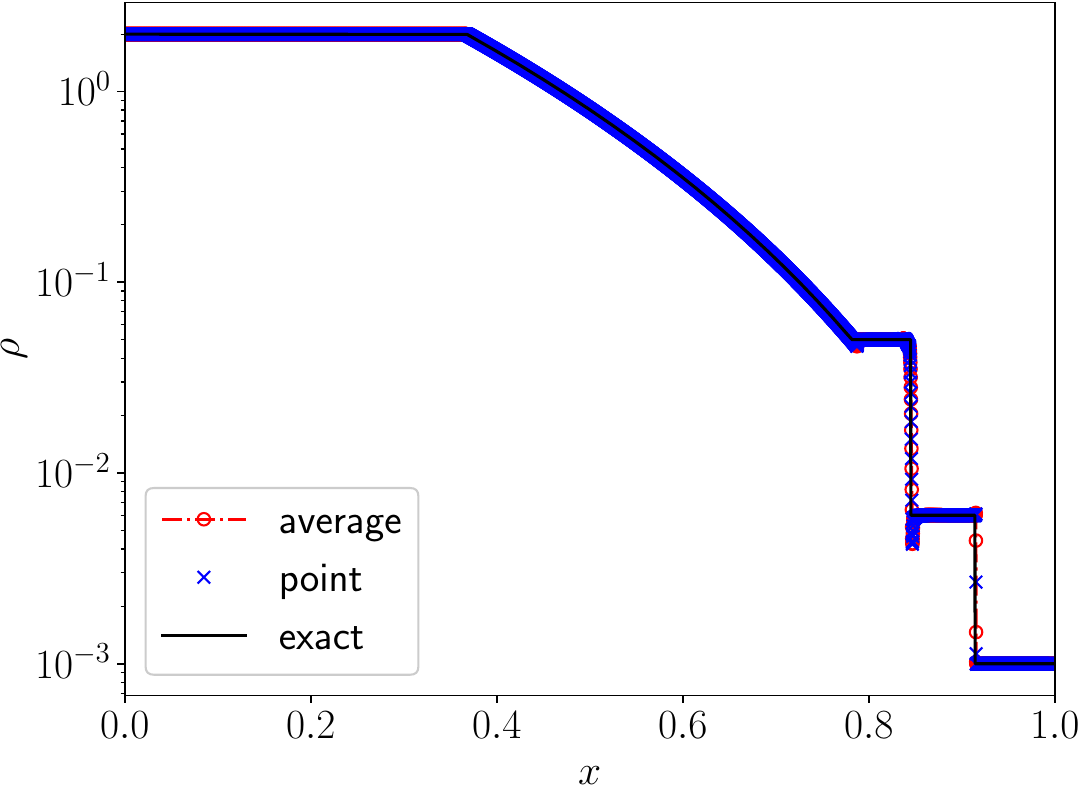}
	\end{subfigure}
	\begin{subfigure}[b]{0.24\textwidth}
		\centering
		\includegraphics[width=\linewidth]{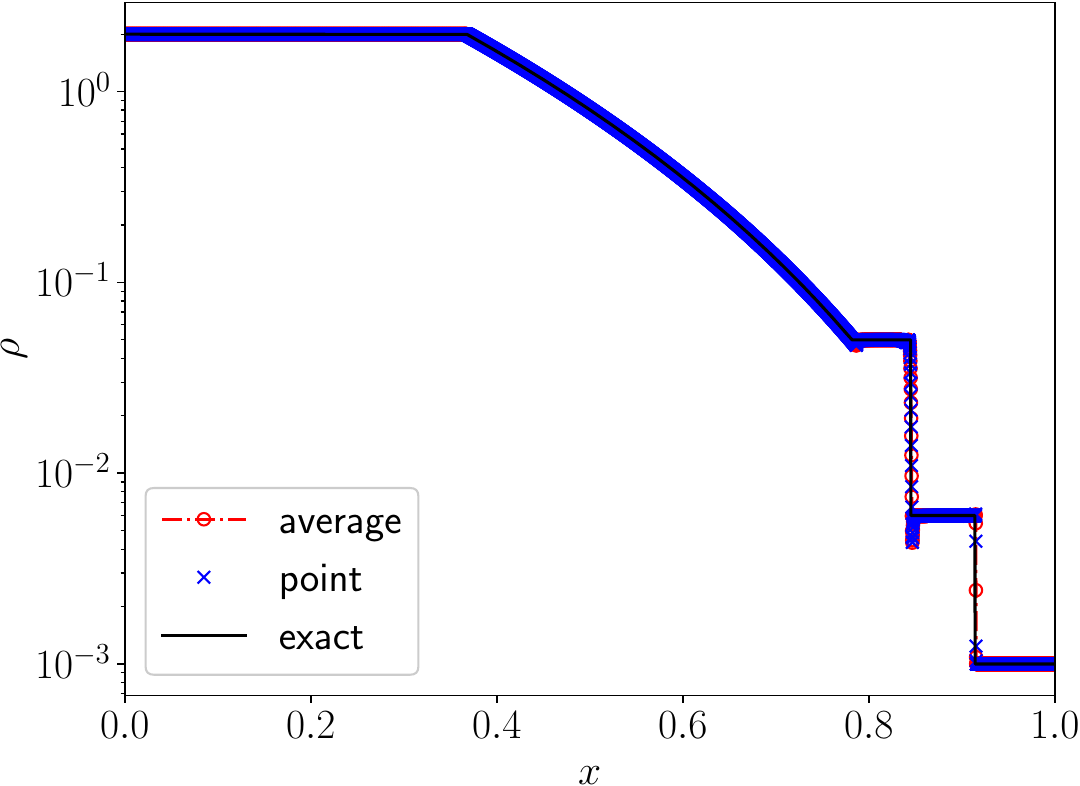}
	\end{subfigure}
	
	\begin{subfigure}[b]{0.24\textwidth}
		\centering
		\includegraphics[width=\linewidth]{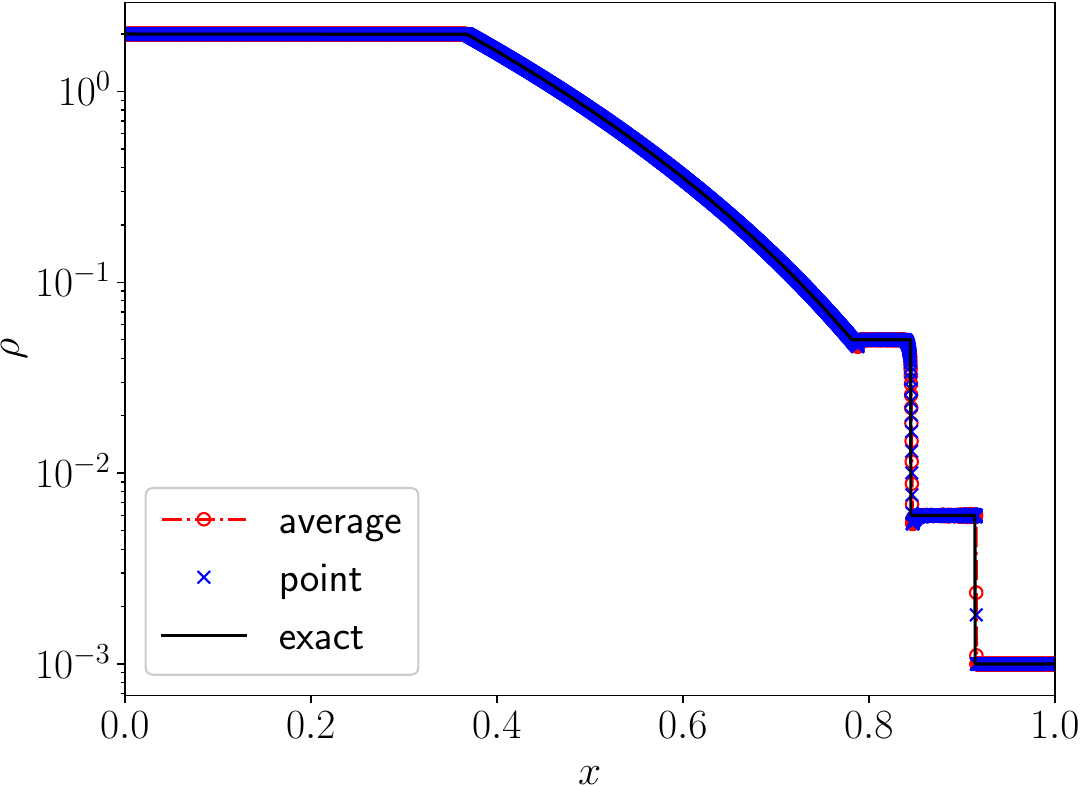}
	\end{subfigure}
	\begin{subfigure}[b]{0.24\textwidth}
		\centering
		\includegraphics[width=\linewidth]{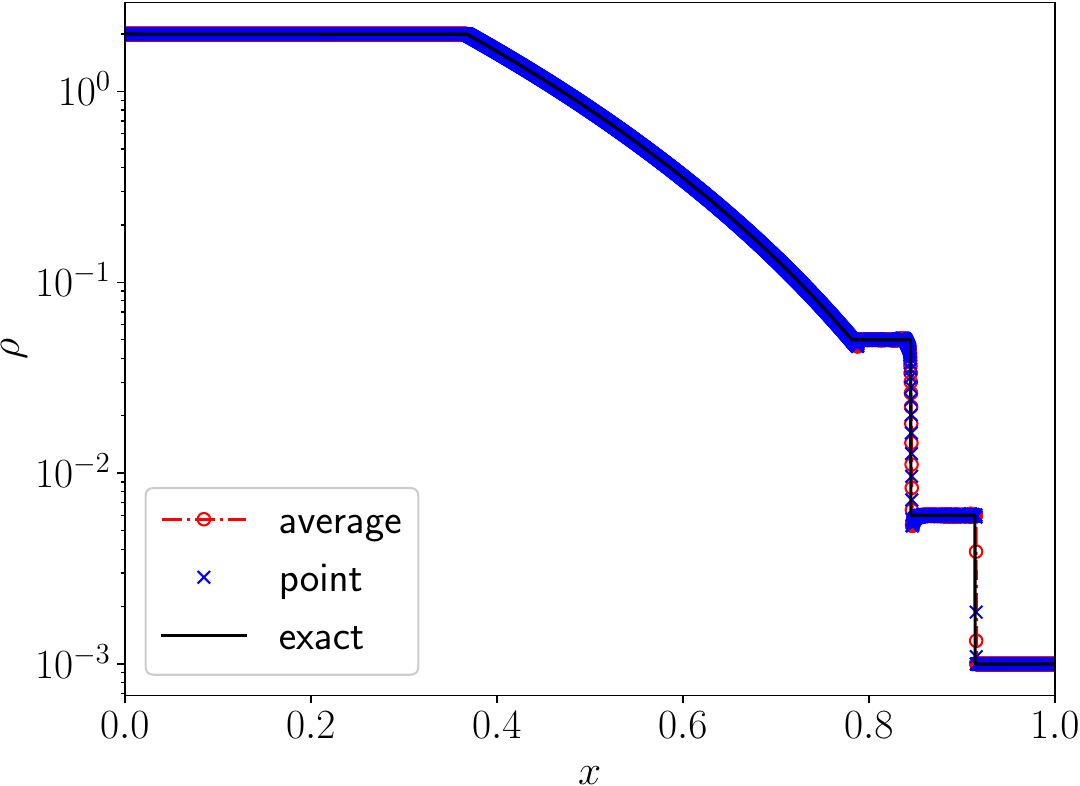}
	\end{subfigure}
	\begin{subfigure}[b]{0.24\textwidth}
		\centering
		\includegraphics[width=\linewidth]{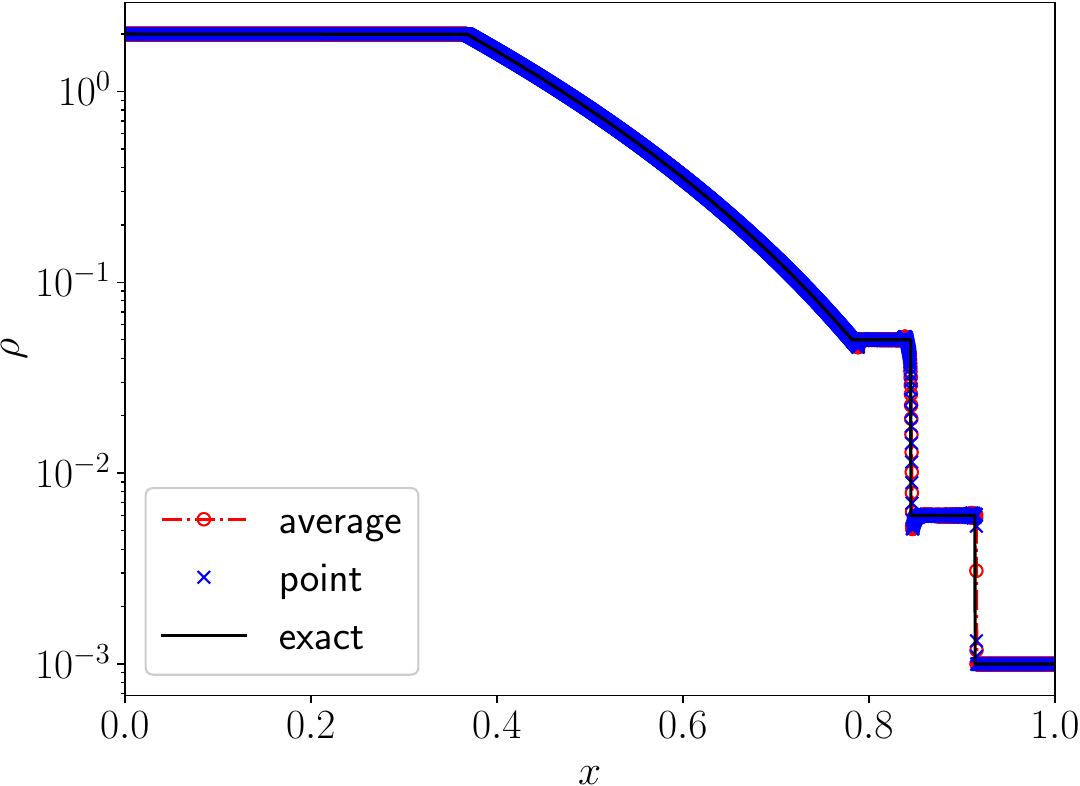}
	\end{subfigure}
	\begin{subfigure}[b]{0.24\textwidth}
		\centering
		\includegraphics[width=\linewidth]{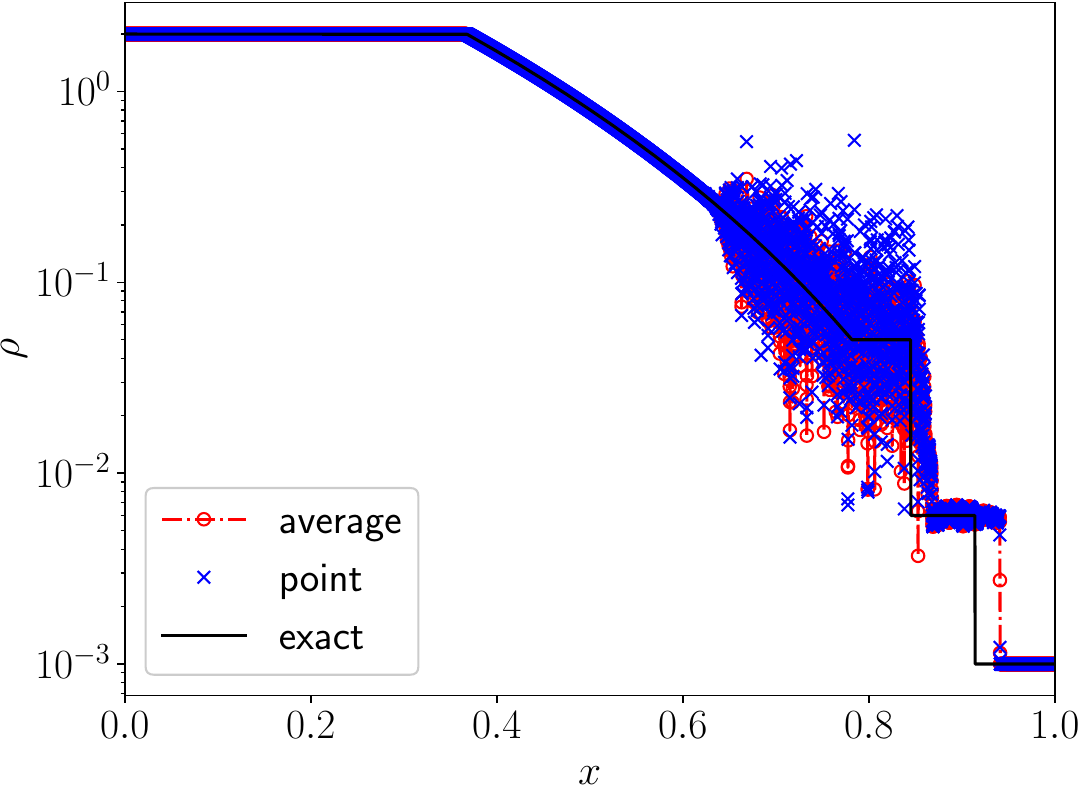}
	\end{subfigure}
	\caption{\Cref{ex:1d_leblanc}, LeBlanc Riemann problem.
		The numerical solutions are computed with the BP limitings for the cell average and point value updates on a uniform mesh of $6000$ cells.
		From left to right: JS, LLF, SW, and VH FVS.
		The CFL number is $0.1$ and the power law reconstruction is not
		activated (top) and activated (bottom).}
	\label{fig:1d_leblanc_fine_mesh_cfl0.1}
\end{figure}

\begin{figure}[htbp]
	\centering
	\begin{subfigure}[b]{0.24\textwidth}
		\centering
		\includegraphics[width=\linewidth]{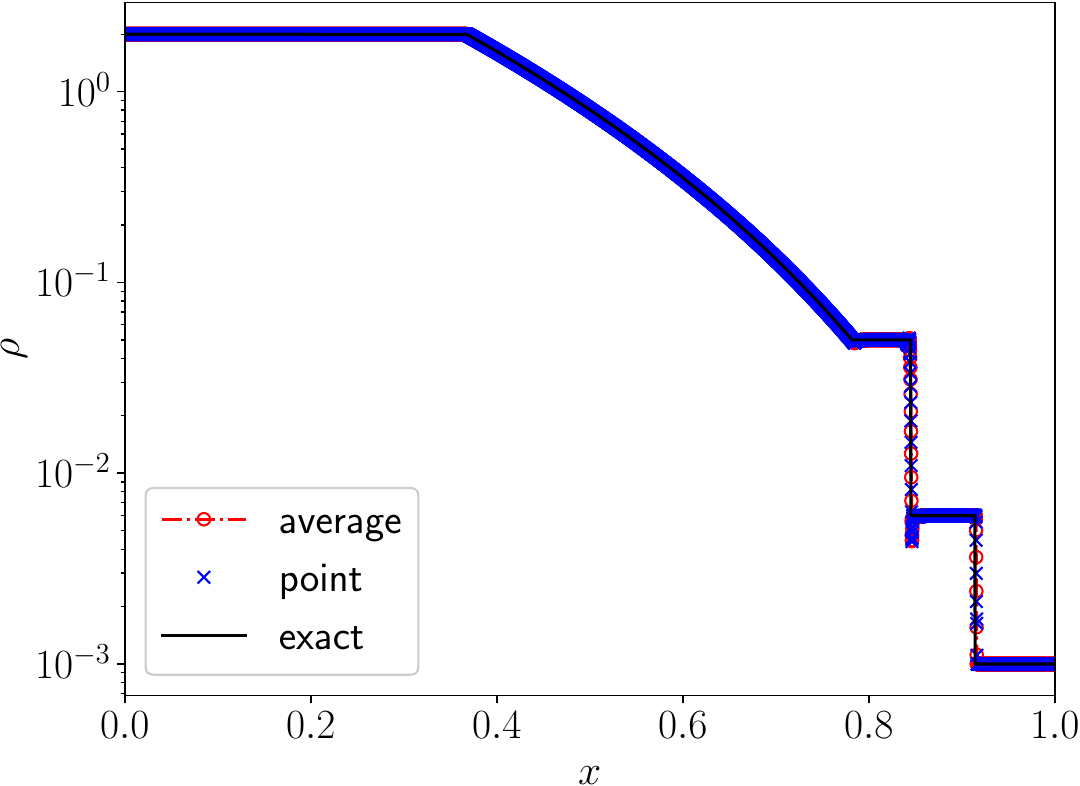}
	\end{subfigure}
	\begin{subfigure}[b]{0.24\textwidth}
		\centering
		\includegraphics[width=\linewidth]{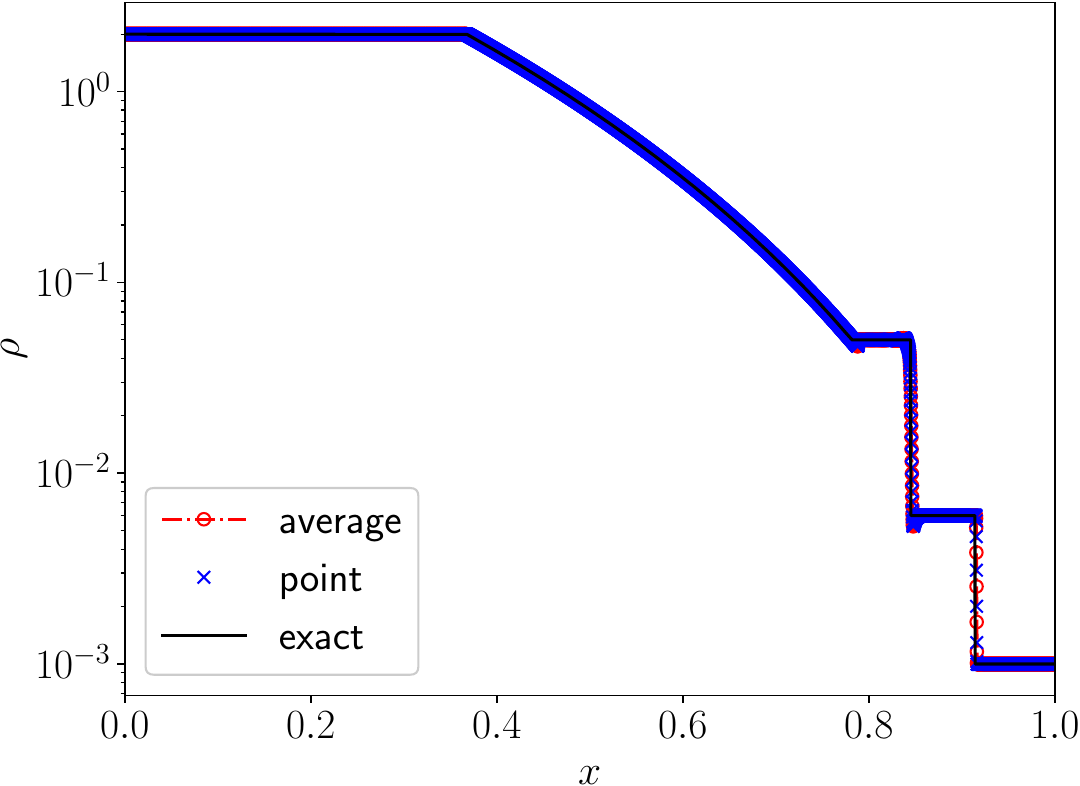}
	\end{subfigure}
	\begin{subfigure}[b]{0.24\textwidth}
		\centering
		\includegraphics[width=\linewidth]{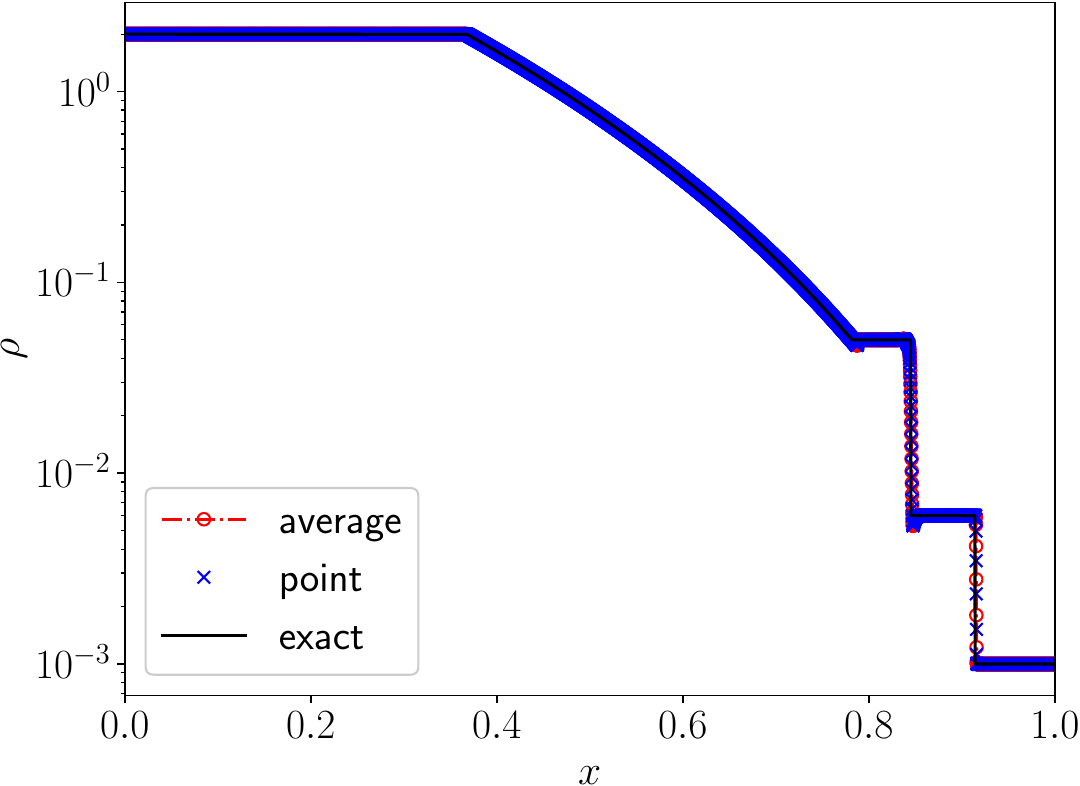}
	\end{subfigure}
	\begin{subfigure}[b]{0.24\textwidth}
		\centering
		\includegraphics[width=\linewidth]{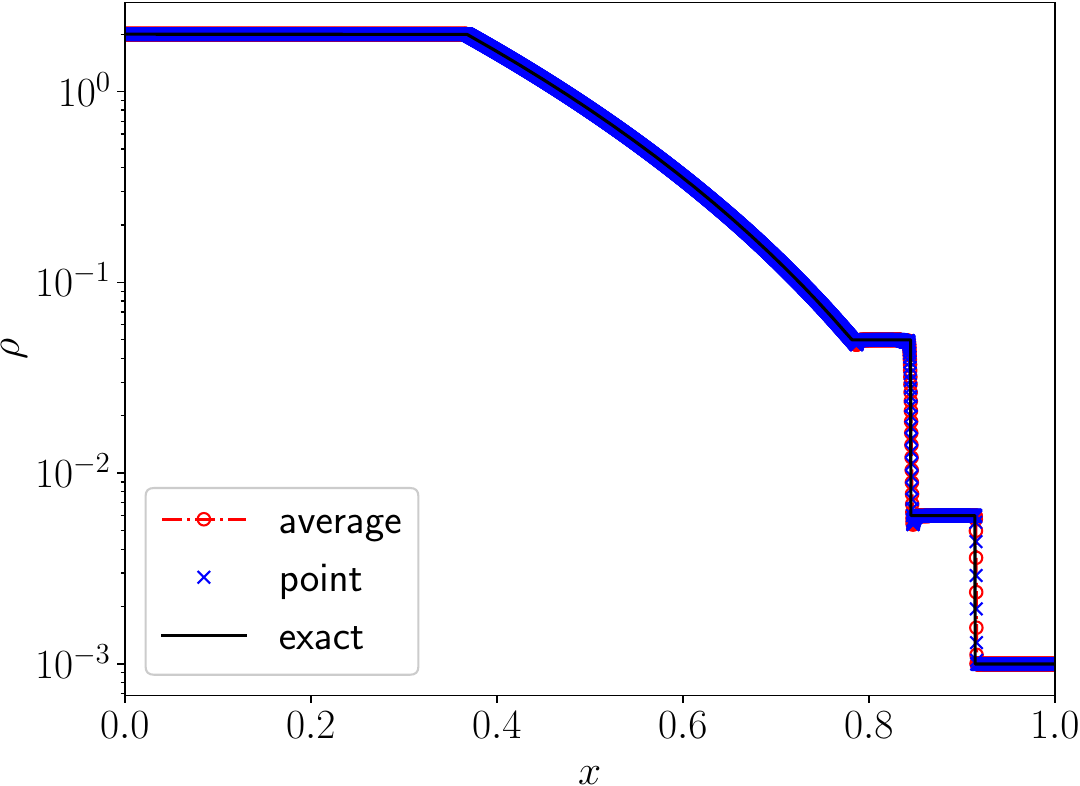}
	\end{subfigure}
	\caption{\Cref{ex:1d_leblanc}, LeBlanc Riemann problem.
		The numerical solutions are computed with the BP limitings for the cell average and point value updates on a uniform mesh of $6000$ cells.
		From left to right: JS, LLF, SW, and VH FVS.
		The CFL number is $0.4,0.4,0.4,0.1$ and the shock sensor-based limiting ($\kappa=10$) is used.}
	\label{fig:1d_leblanc_fine_mesh_shock_sensor}
\end{figure}
\end{example}

\begin{example}[Sedov problem]\label{ex:1d_sedov}\rm
	In this problem, a volume of uniform density and temperature is initialized, and a large quantity of thermal energy is injected at the center, developing into a blast wave that evolves in time in a self-similar fashion \cite{Sedov_1959_Similarity_book}.
	An exact analytical solution based on self-similarity arguments is available \cite{Kamm_2007_efficient}, which contains very low
	density with strong shocks.
	For the background value, the initial density is one, velocity is zero, and total energy is $10^{-12}$ everywhere except that
	in the centered cell, the total energy of the cell average and point values at two cell interfaces are $3.2\times 10^6/\Delta x$ with $\Delta x = 4/N$ with $N$ the number of cells, which is used to emulate a $\delta$-function at the center.
	It should be noted that if the two point values at the interfaces of the centered cell are initialized with the background value, the transonic issue appears for the JS.
	The test is solved until $T= 10^{-3}$.

This test is run with $N=801$ cells, and the density plots in the right half domain are shown in \cref{fig:1d_sedov}.
The BP limitings are adopted for the cell average and point value updates,
while the power law reconstruction is not used.
The maximal CFL numbers for different point value updates to be stable are also listed in the caption,
i.e., $0.1$ for the JS, $0.4$, $0.3$, and $0.3$ for the LLF, SW, and VH FVS, respectively. 
The numerical solutions obtained by the three FVS are nearly the same,
while there are some defects in the solution based on the JS.
Thus the LLF FVS is superior to others regarding the time step size and the shock-capturing ability.

\begin{figure}[htbp]
	\centering
	\begin{subfigure}[b]{0.24\textwidth}
		\centering
		\includegraphics[width=\linewidth]{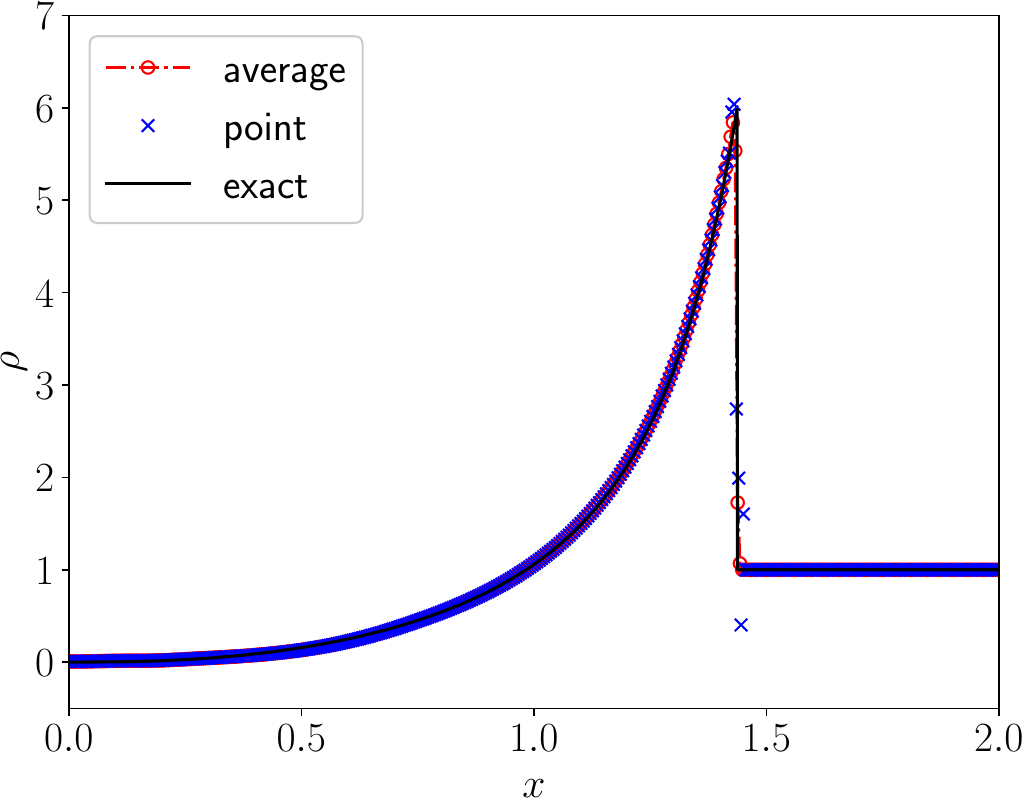}
	\end{subfigure}
	\begin{subfigure}[b]{0.24\textwidth}
		\centering
		\includegraphics[width=\linewidth]{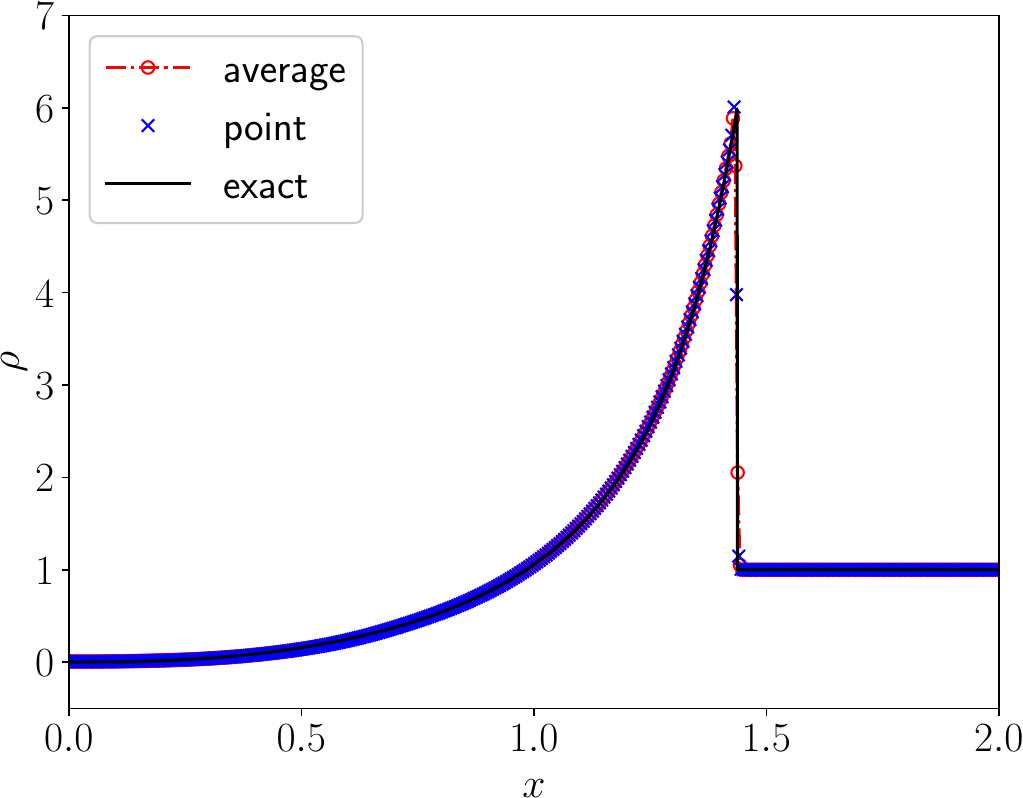}
	\end{subfigure}
	\begin{subfigure}[b]{0.24\textwidth}
		\centering
		\includegraphics[width=\linewidth]{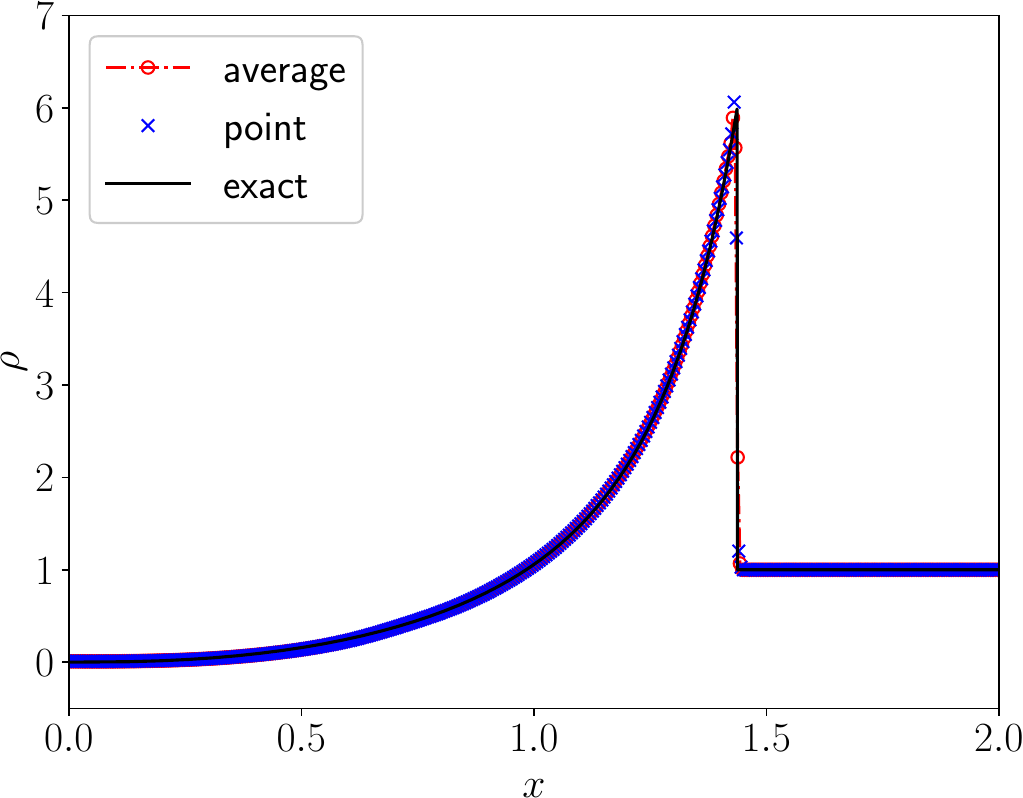}
	\end{subfigure}
	\begin{subfigure}[b]{0.24\textwidth}
		\centering
		\includegraphics[width=\linewidth]{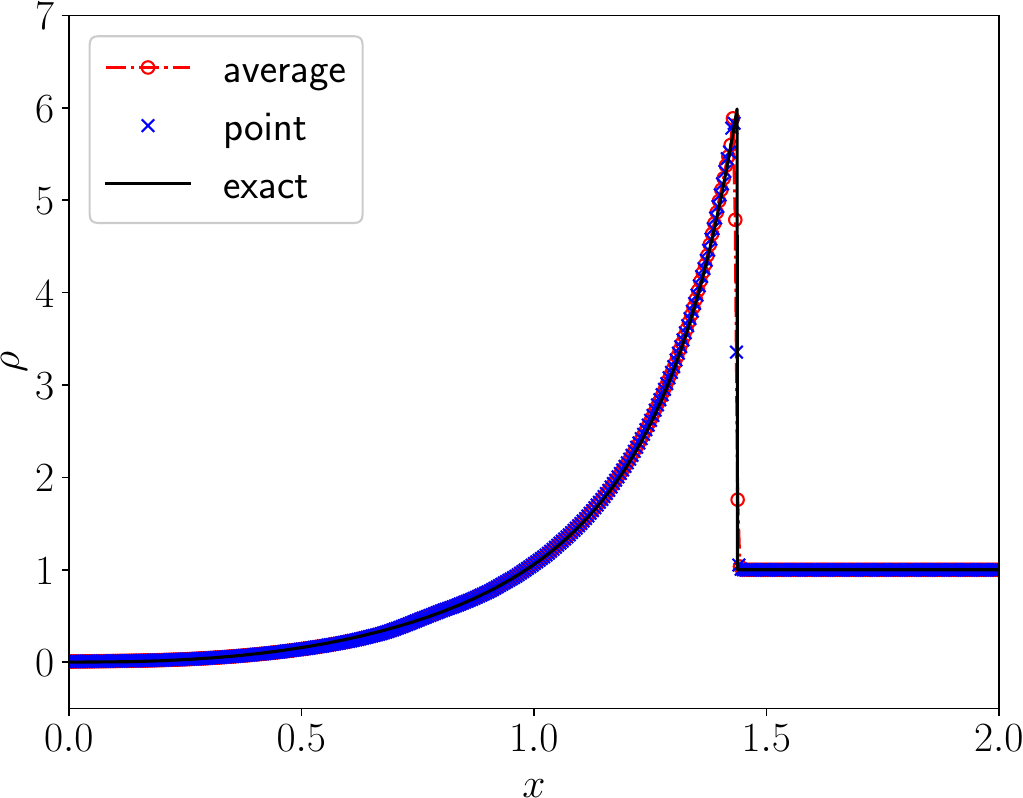}
	\end{subfigure}
	\caption{\Cref{ex:1d_sedov}, Sedov problem.
		The numerical solutions are computed with the BP limitings for the cell average and point value updates on a uniform mesh of $801$ cells,
        without the power law reconstruction.
	The CFL number is (from left to right): $0.1$ for the JS,
		$0.4$ for the LLF FVS,
		$0.3$ for the SW FVS,
		$0.25$ for the VH FVS.}
	\label{fig:1d_sedov}
\end{figure}
\end{example}

\begin{example}[Blast wave interaction \cite{Woodward_1984_numerical_JoCP}]\label{ex:1d_blast_wave}\rm
	This test describes the interaction of two strong shocks in the domain $[0,1]$ with reflective boundary conditions.
	The test is solved until $T=0.038$.

Due to the low-pressure region, the schemes blow up without the BP limitings.
\Cref{fig:1d_blast_wave_coarse_mesh} shows the density profiles and corresponding enlarged views in $x\in[0.62, 0.82]$ obtained by using the BP limitings on a uniform mesh of $800$ cells,
in which the power law reconstruction is not activated.
It is seen that the numerical solutions are close to the reference solution,
although there are some oscillations in the enlarged views.
Then the power law reconstruction is additionally adopted to see if it can suppress the oscillations.
The results with the CFL number $0.1$ and a refined mesh of $1600$ cells are shown in \cref{fig:1d_blast_wave_fine_mesh_pwl},
from which one can observe that the oscillations reduce,
and the LLF FVS gives the best result.

\begin{figure}[htbp]
	\centering
	\begin{subfigure}[b]{0.24\textwidth}
		\centering
		\includegraphics[width=1.0\linewidth]{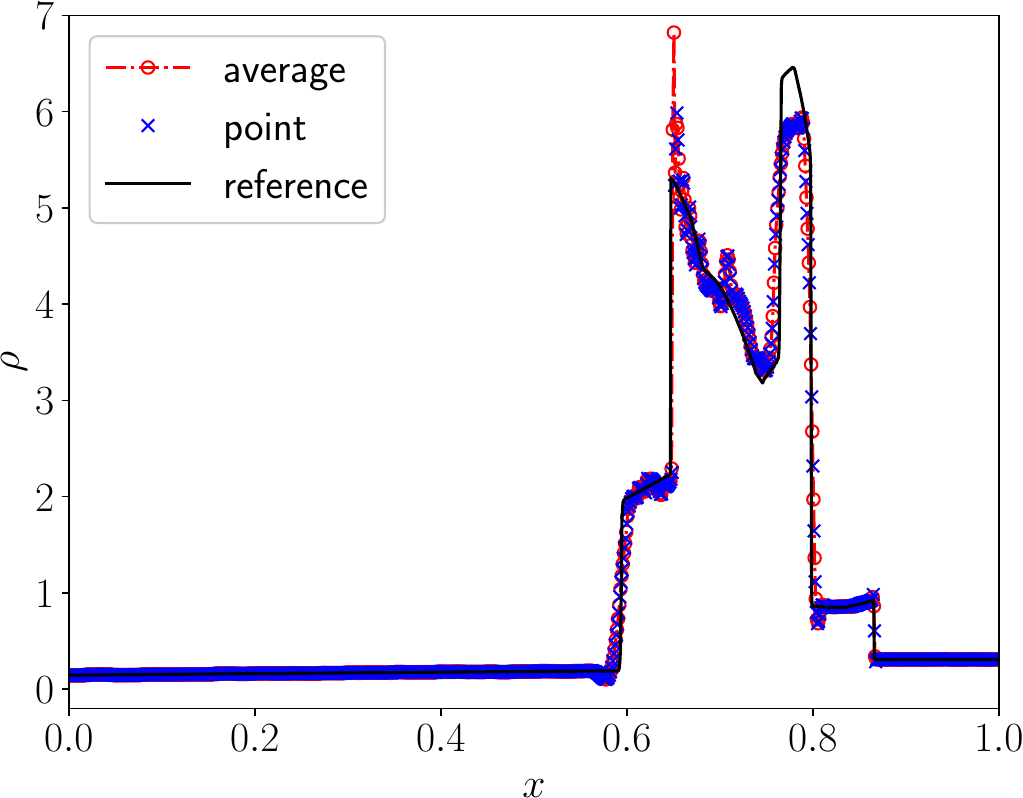}
	\end{subfigure}
	\begin{subfigure}[b]{0.24\textwidth}
		\centering
		\includegraphics[width=1.0\linewidth]{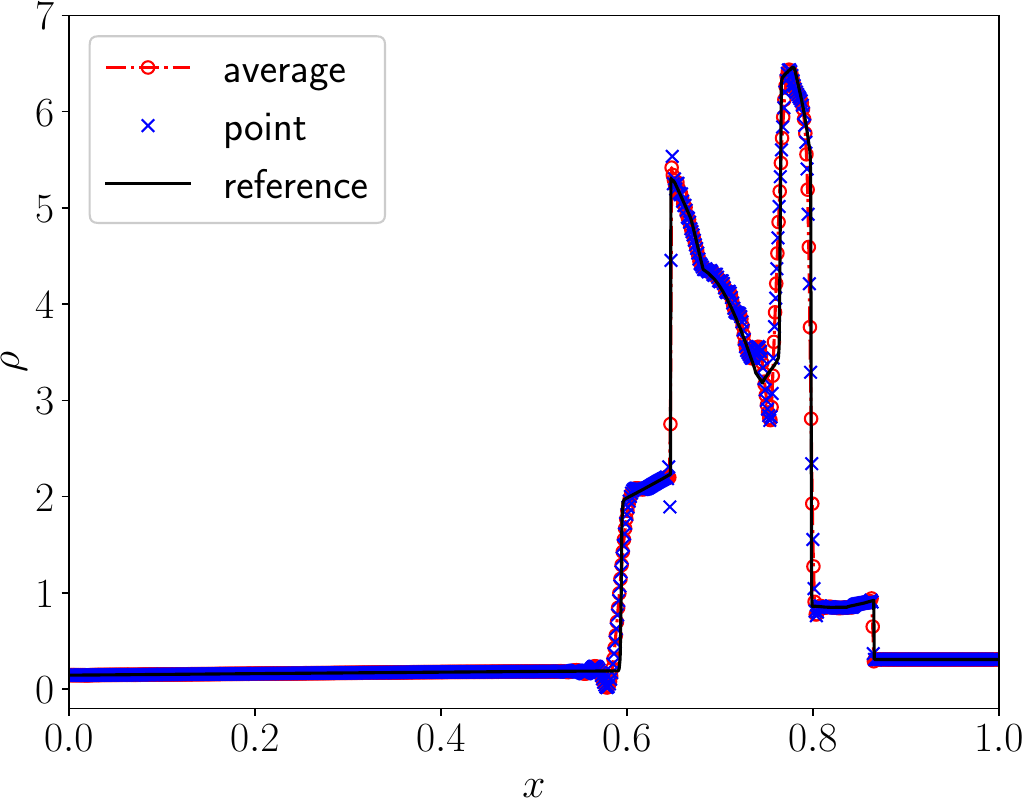}
	\end{subfigure}
	\begin{subfigure}[b]{0.24\textwidth}
		\centering
		\includegraphics[width=1.0\linewidth]{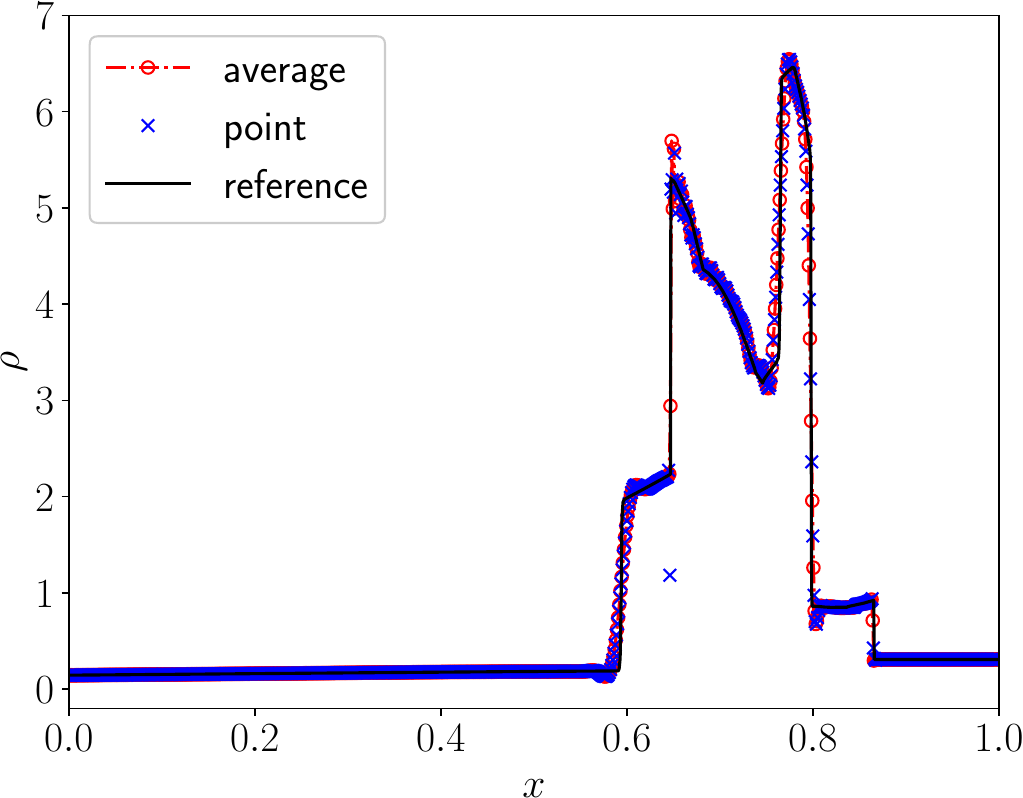}
	\end{subfigure}
	\begin{subfigure}[b]{0.24\textwidth}
		\centering
		\includegraphics[width=1.0\linewidth]{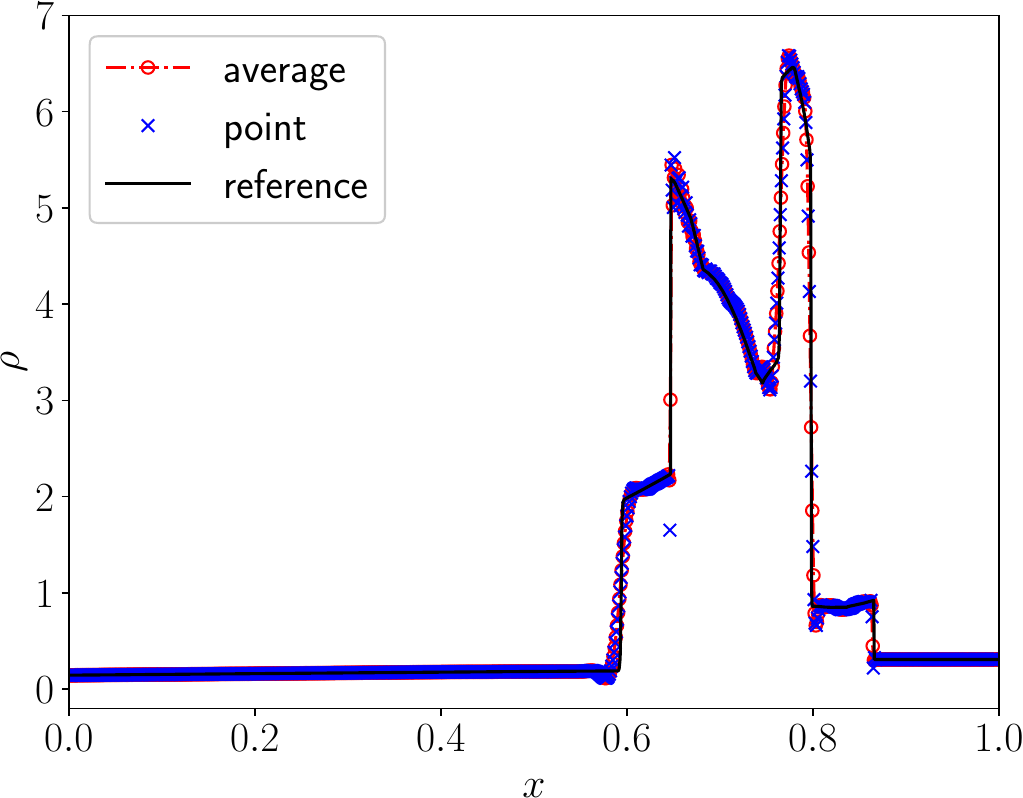}
	\end{subfigure}
	
	\begin{subfigure}[b]{0.24\textwidth}
		\centering
		\includegraphics[width=1.0\linewidth]{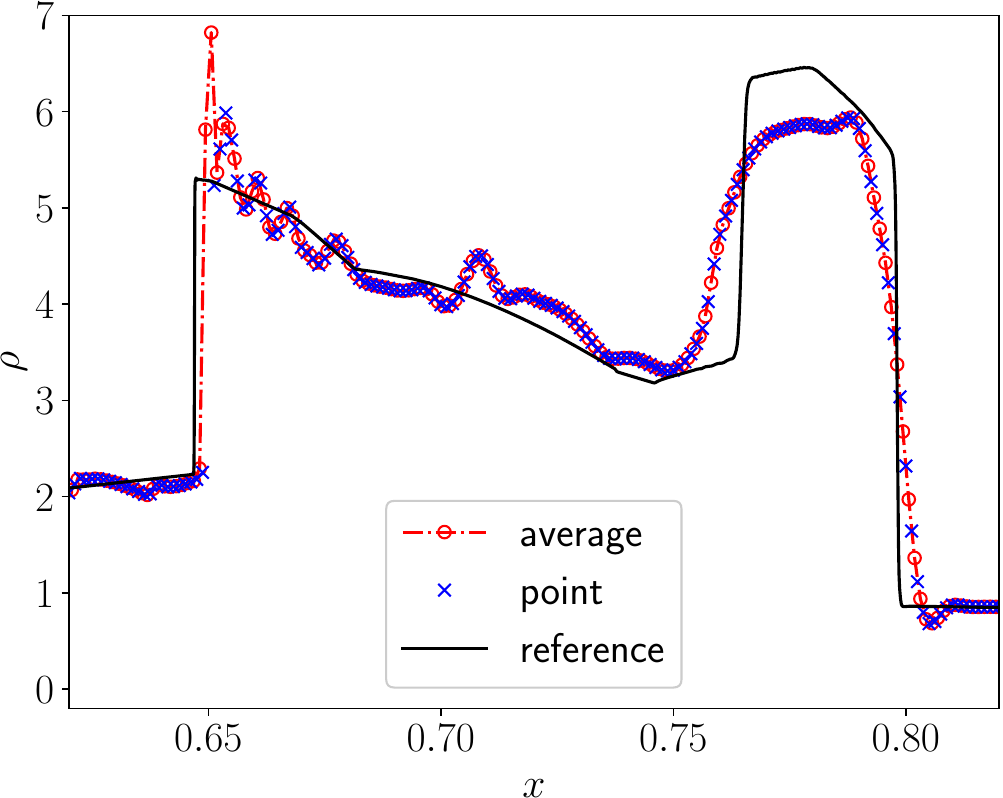}
	\end{subfigure}
	\begin{subfigure}[b]{0.24\textwidth}
		\centering
		\includegraphics[width=1.0\linewidth]{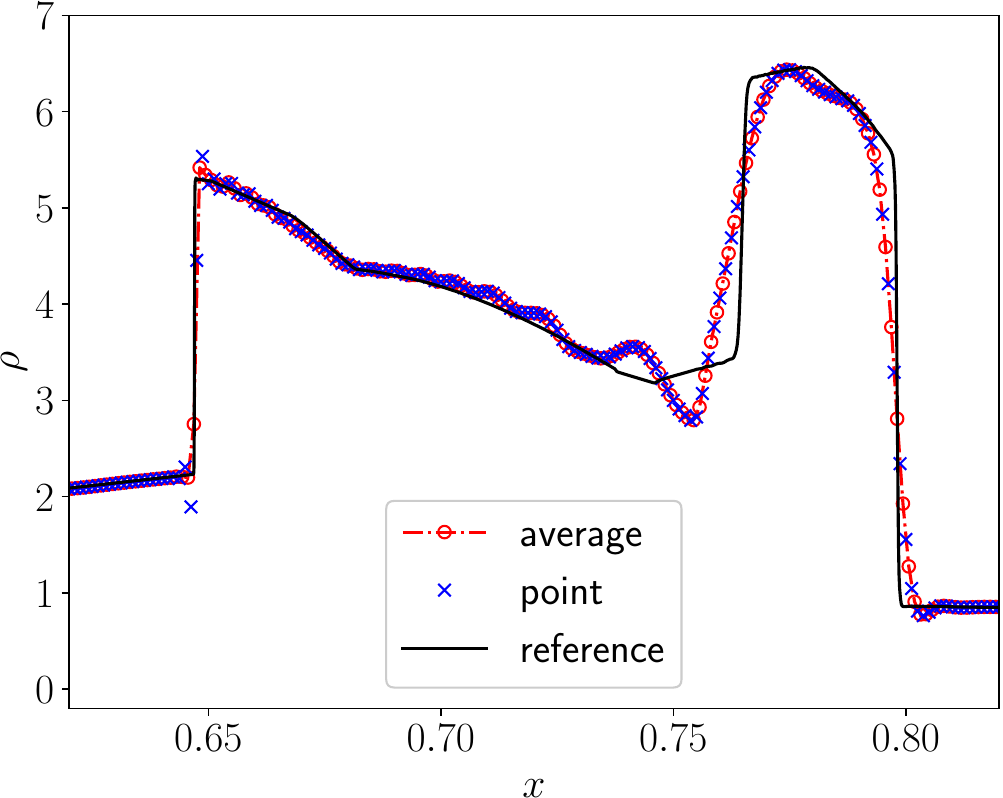}
	\end{subfigure}
	\begin{subfigure}[b]{0.24\textwidth}
		\centering
		\includegraphics[width=1.0\linewidth]{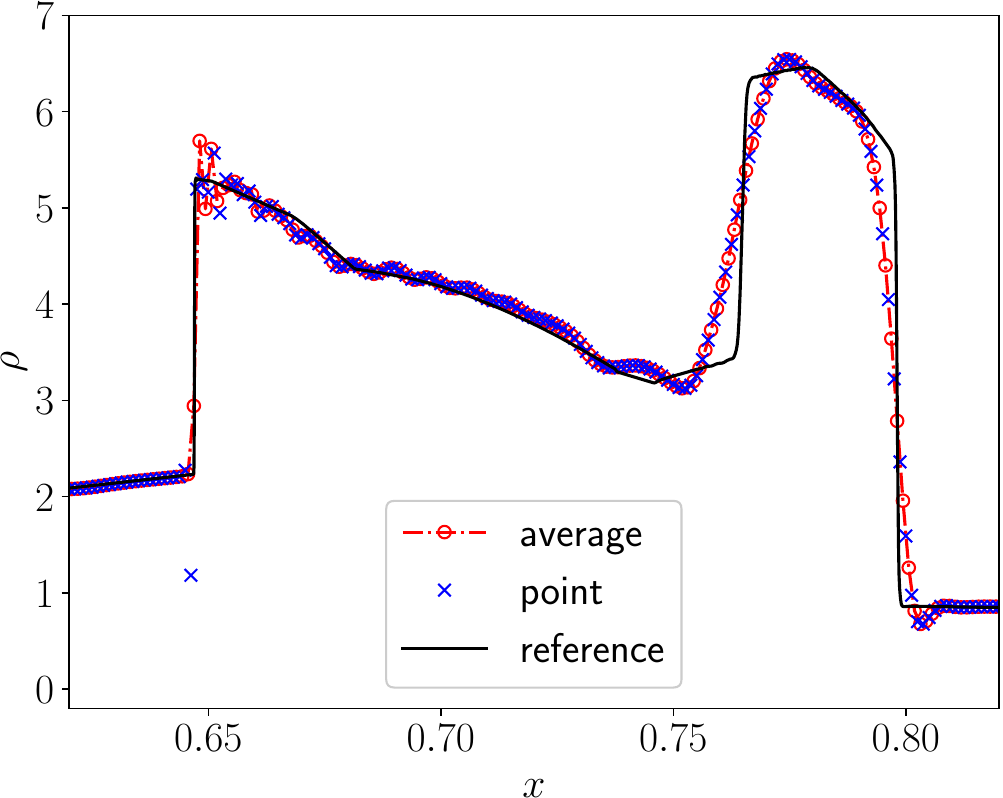}
	\end{subfigure}
	\begin{subfigure}[b]{0.24\textwidth}
		\centering
		\includegraphics[width=1.0\linewidth]{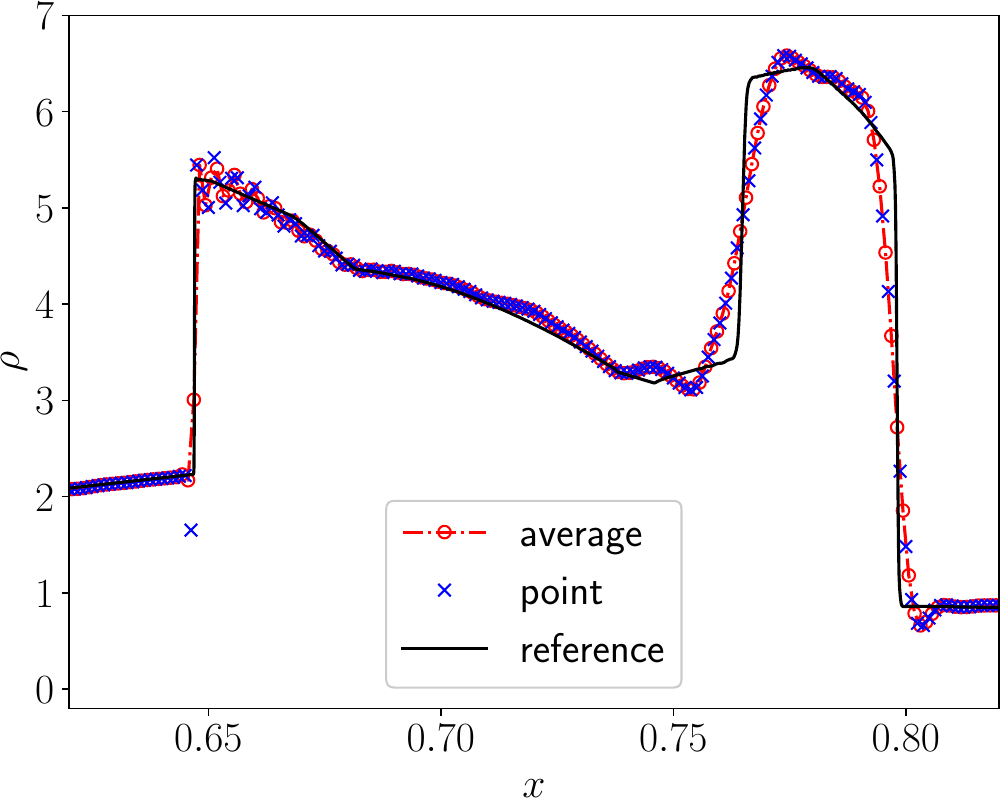}
	\end{subfigure}
	\caption{\Cref{ex:1d_blast_wave}, blast wave interaction.
		The numerical solutions are computed with the BP limitings for the cell average and point value updates on a uniform mesh of $800$ cells.
		The power law reconstruction is not used,
		and from left to right: the CFL number is $0.4$, $0.4$, $0.4$, $0.35$ for the JS, LLF, SW, and VH FVS, respectively.
		The corresponding enlarged views in $[0.62, 0.82]$ are shown in the bottom row.}
	\label{fig:1d_blast_wave_coarse_mesh}
\end{figure}

\begin{figure}[htbp]
	\centering
	\begin{subfigure}[b]{0.24\textwidth}
		\centering
		\includegraphics[width=1.0\linewidth]{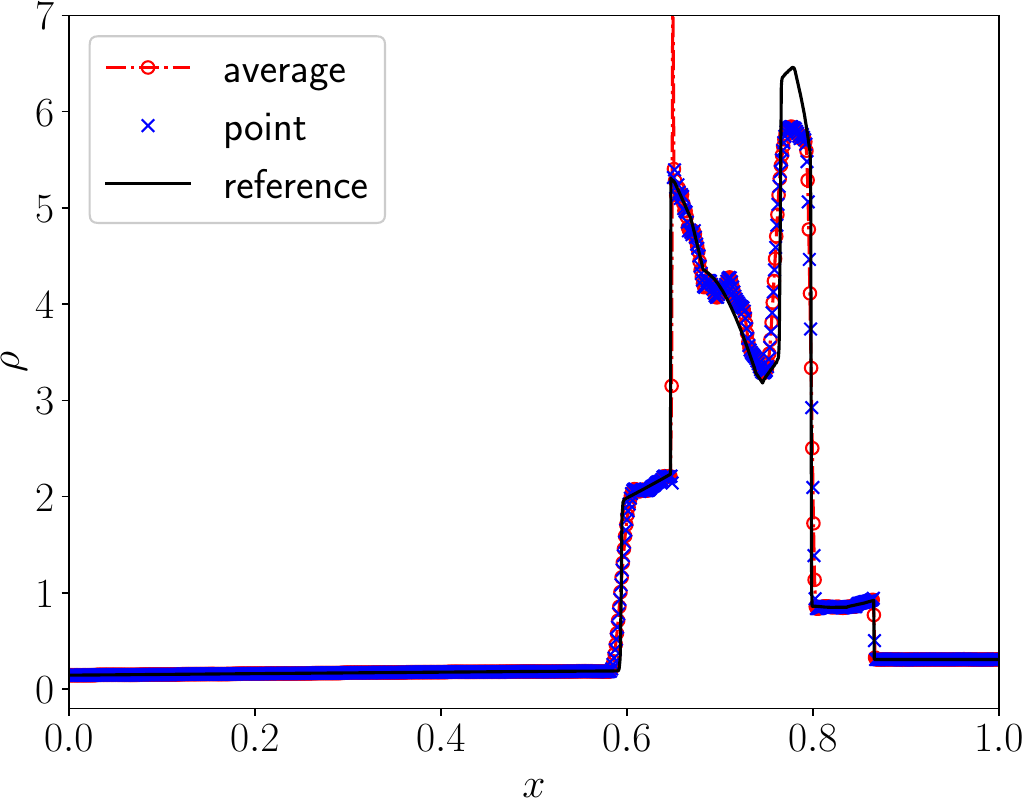}
	\end{subfigure}
	\begin{subfigure}[b]{0.24\textwidth}
		\centering
		\includegraphics[width=1.0\linewidth]{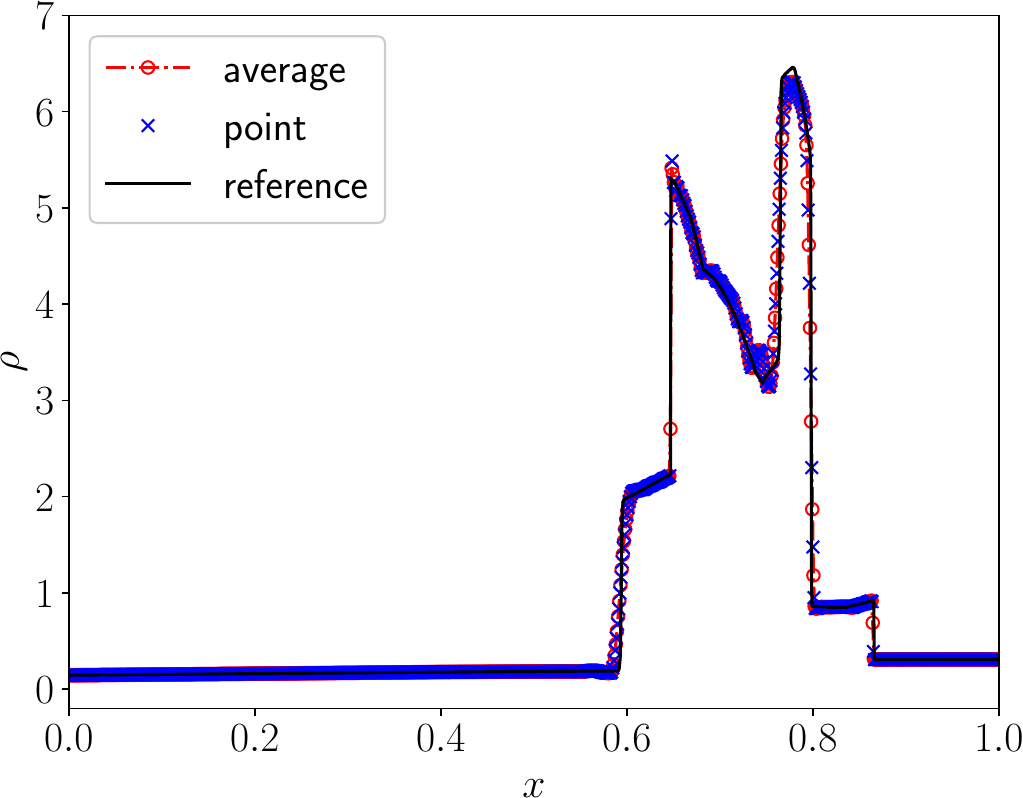}
	\end{subfigure}
	\begin{subfigure}[b]{0.24\textwidth}
		\centering
		\includegraphics[width=1.0\linewidth]{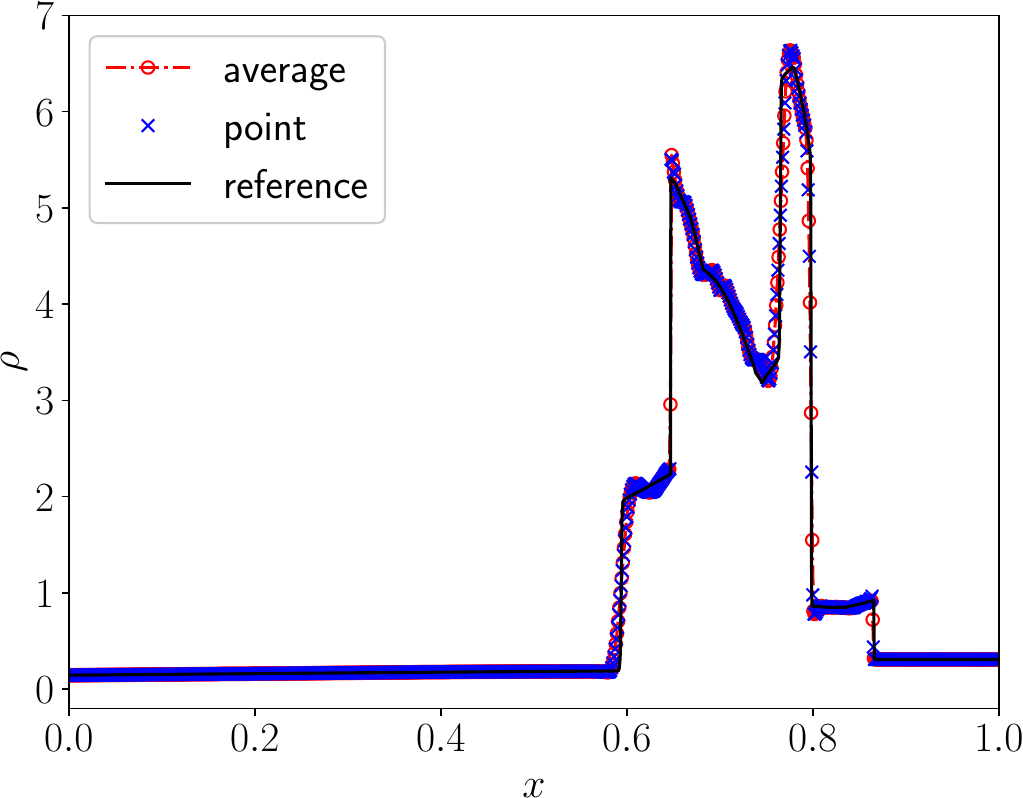}
	\end{subfigure}
	\begin{subfigure}[b]{0.24\textwidth}
		\centering
		\includegraphics[width=1.0\linewidth]{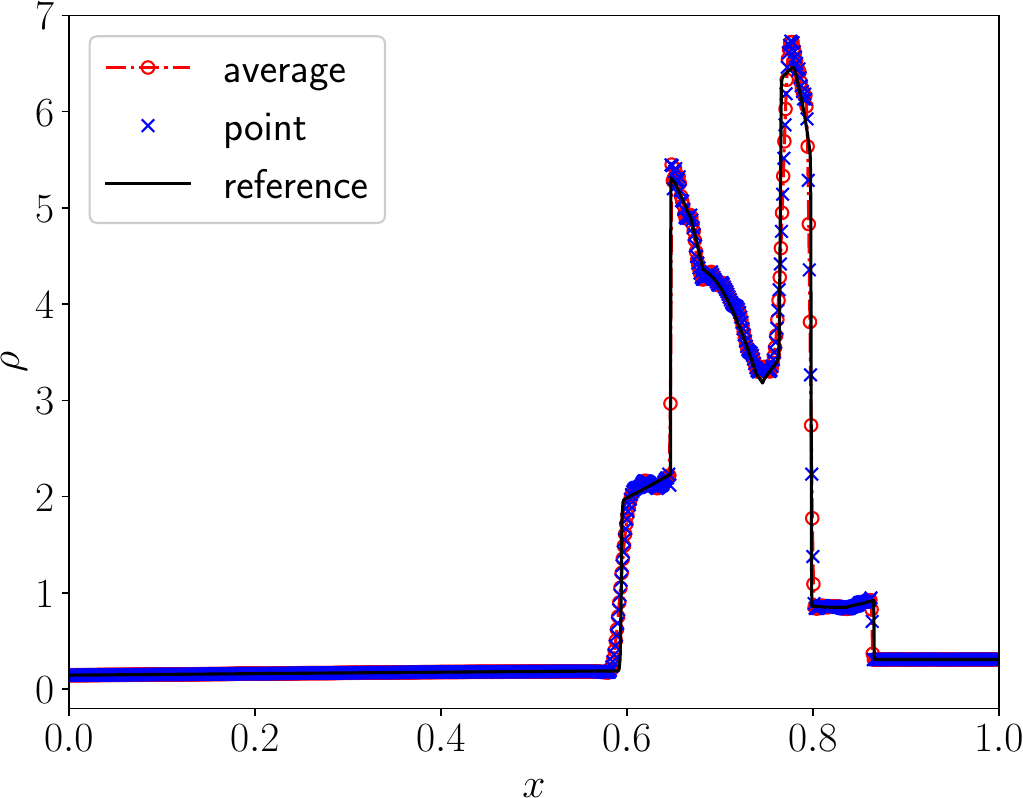}
	\end{subfigure}
	
	\begin{subfigure}[b]{0.24\textwidth}
		\centering
		\includegraphics[width=1.0\linewidth]{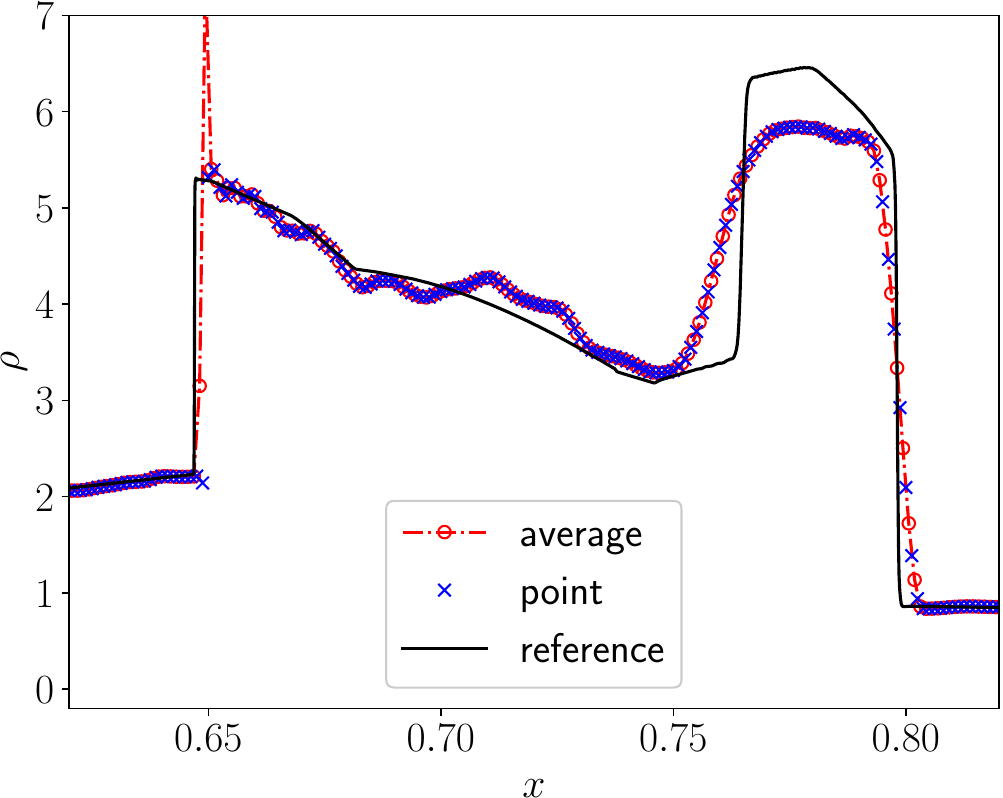}
	\end{subfigure}
	\begin{subfigure}[b]{0.24\textwidth}
		\centering
		\includegraphics[width=1.0\linewidth]{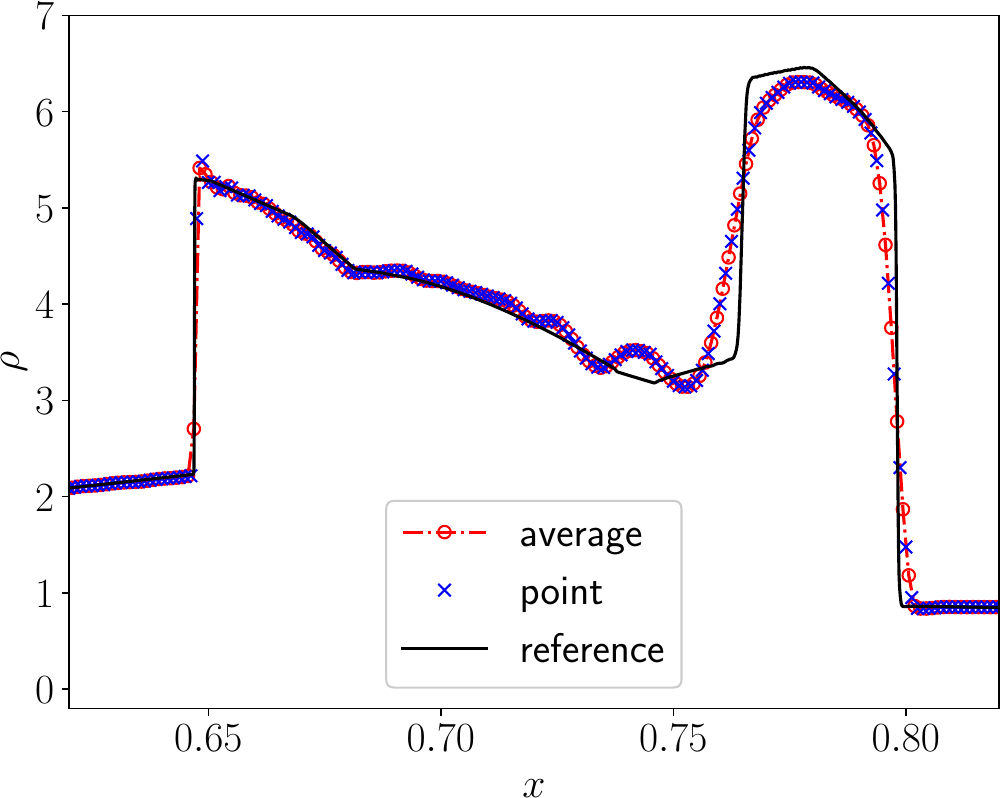}
	\end{subfigure}
	\begin{subfigure}[b]{0.24\textwidth}
		\centering
		\includegraphics[width=1.0\linewidth]{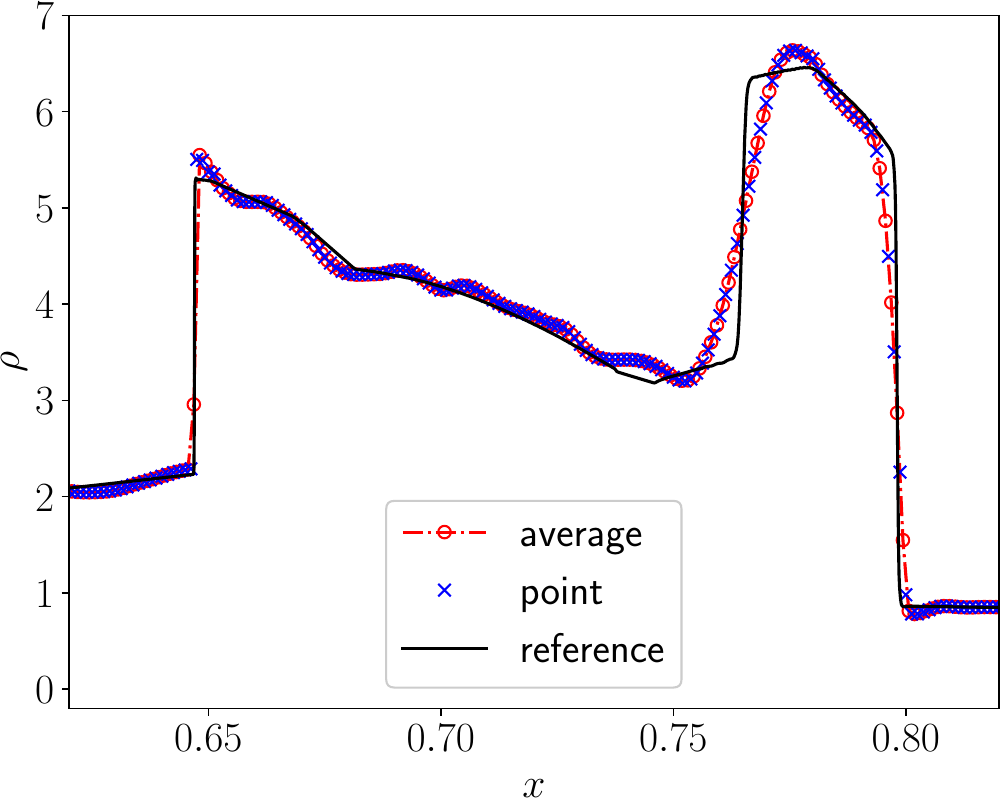}
	\end{subfigure}
	\begin{subfigure}[b]{0.24\textwidth}
		\centering
		\includegraphics[width=1.0\linewidth]{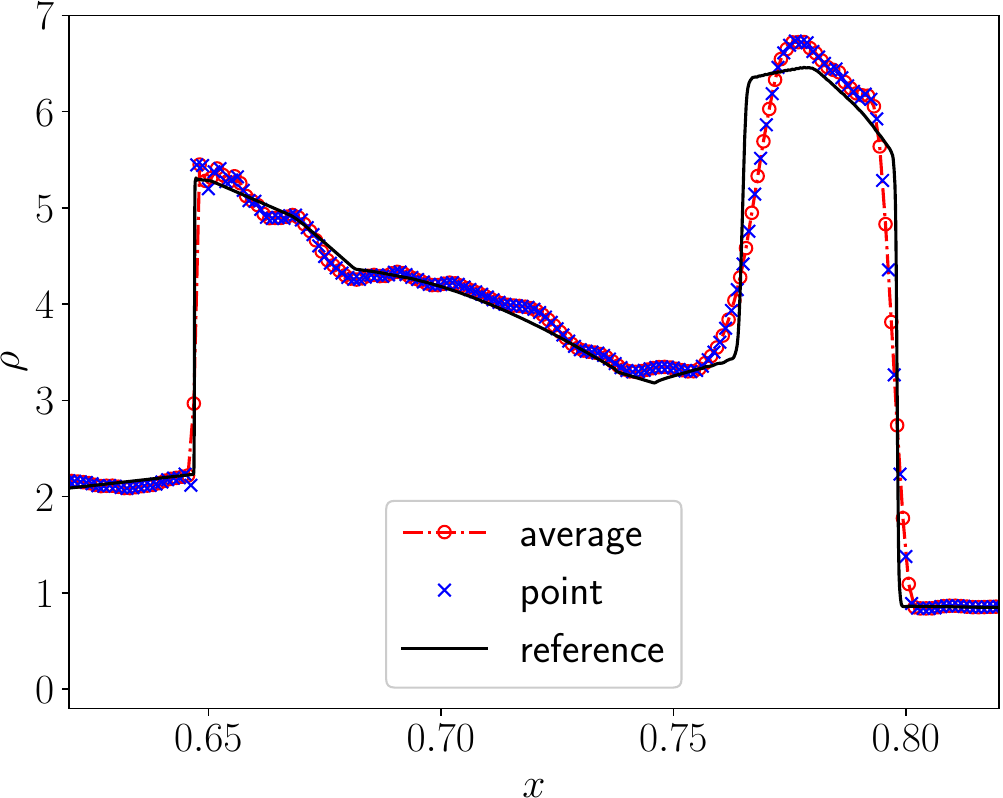}
	\end{subfigure}
	\caption{\Cref{ex:1d_blast_wave}, blast wave interaction.
		The numerical solutions are computed with the power law reconstruction and the BP limitings for the cell average and point values update on a uniform mesh of $800$ cells.
		The CFL number is $0.1$ for all the point value updates,
		and the corresponding enlarged views in $[0.62, 0.82]$ are shown in the bottom row.
		From left to right: JS, LLF, SW, and VH FVS.
		}
	\label{fig:1d_blast_wave_fine_mesh_pwl}
\end{figure}

\begin{figure}[htbp]
	\centering
	\begin{subfigure}[b]{0.24\textwidth}
		\centering
		\includegraphics[width=1.0\linewidth]{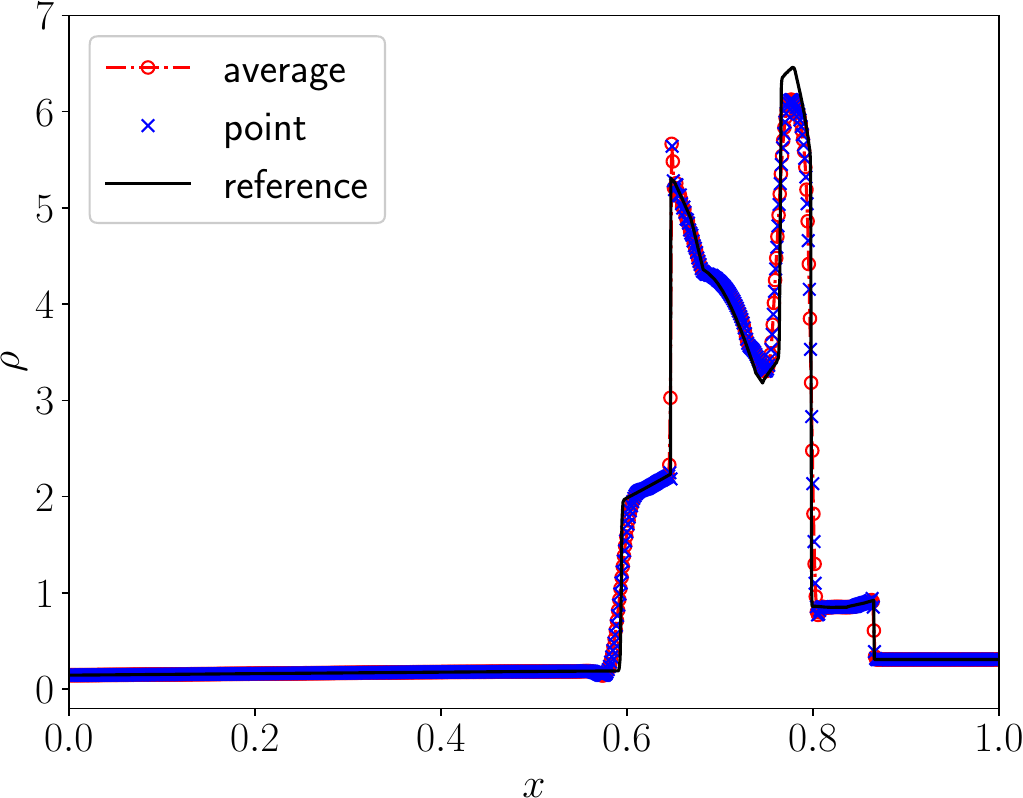}
	\end{subfigure}
	\begin{subfigure}[b]{0.24\textwidth}
		\centering
		\includegraphics[width=1.0\linewidth]{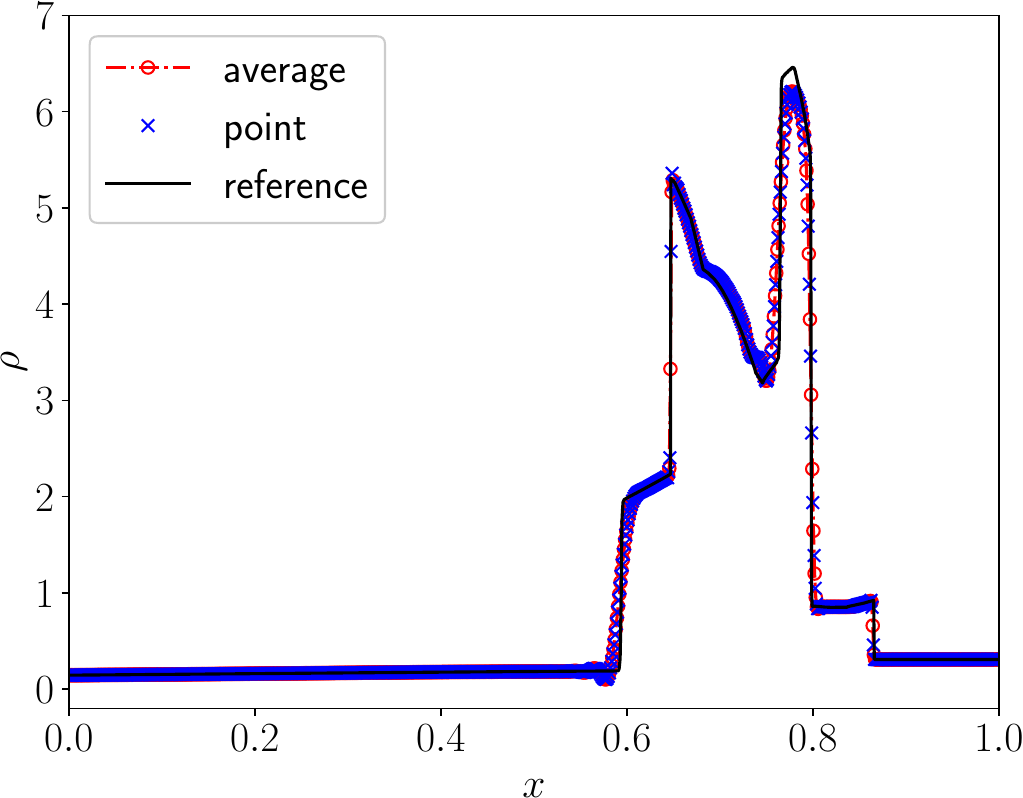}
	\end{subfigure}
	\begin{subfigure}[b]{0.24\textwidth}
		\centering
		\includegraphics[width=1.0\linewidth]{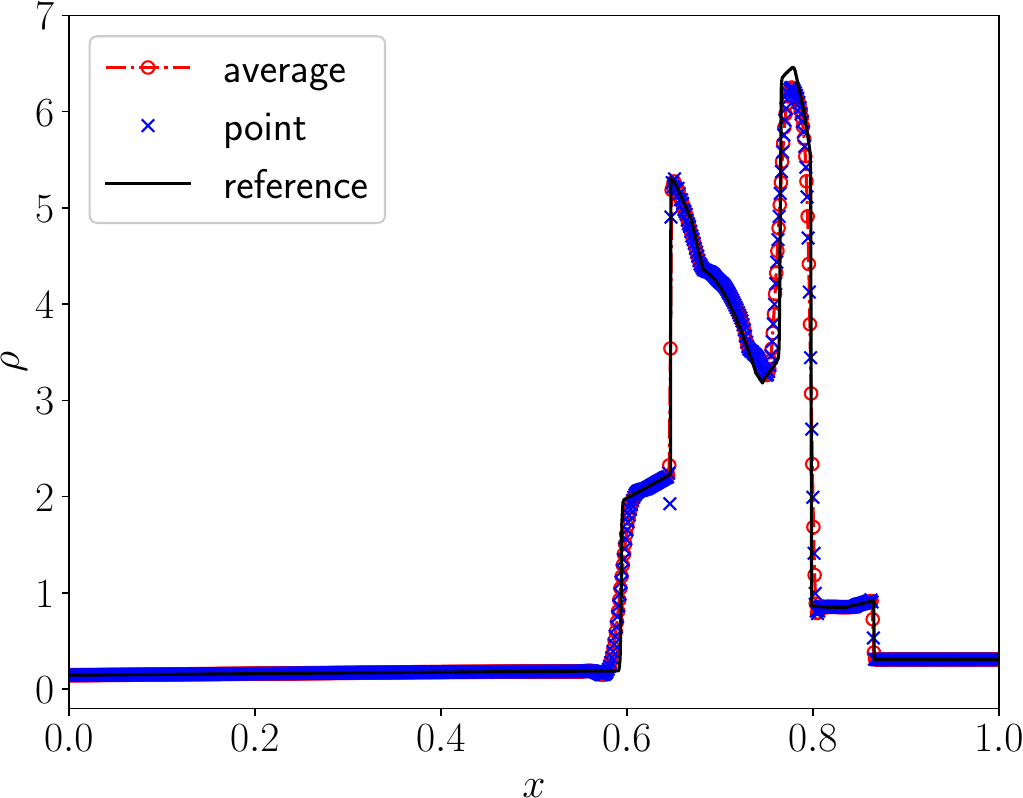}
	\end{subfigure}
	\begin{subfigure}[b]{0.24\textwidth}
		\centering
		\includegraphics[width=1.0\linewidth]{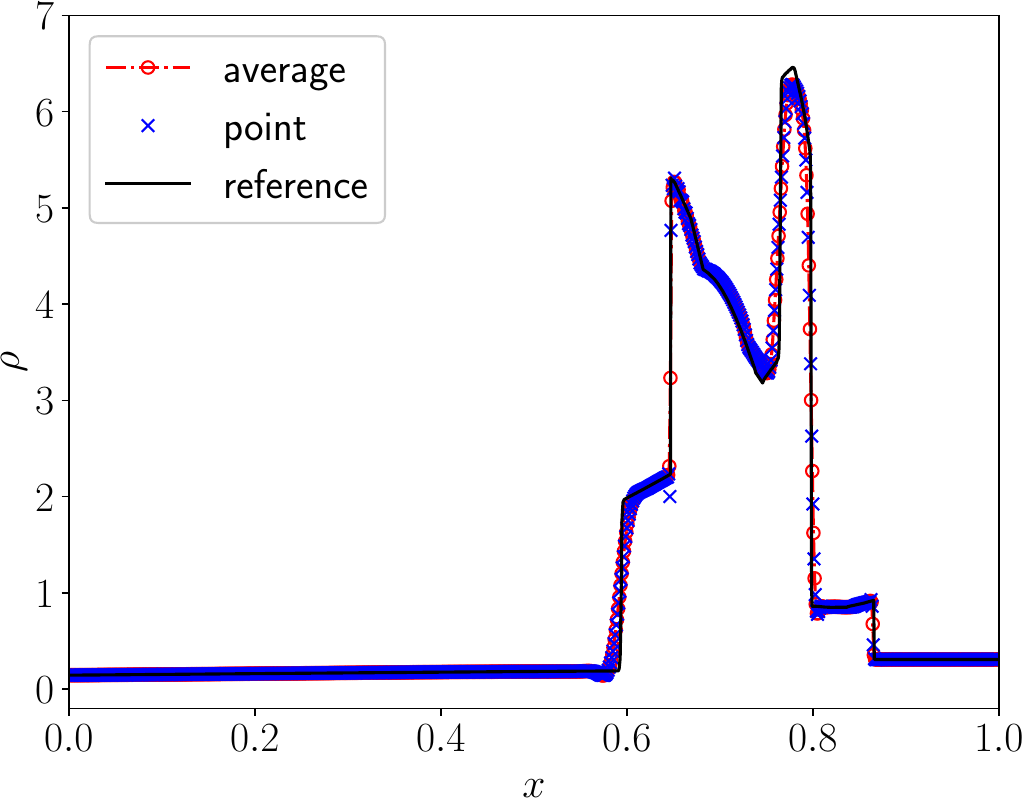}
	\end{subfigure}
	
	\begin{subfigure}[b]{0.24\textwidth}
		\centering
		\includegraphics[width=1.0\linewidth]{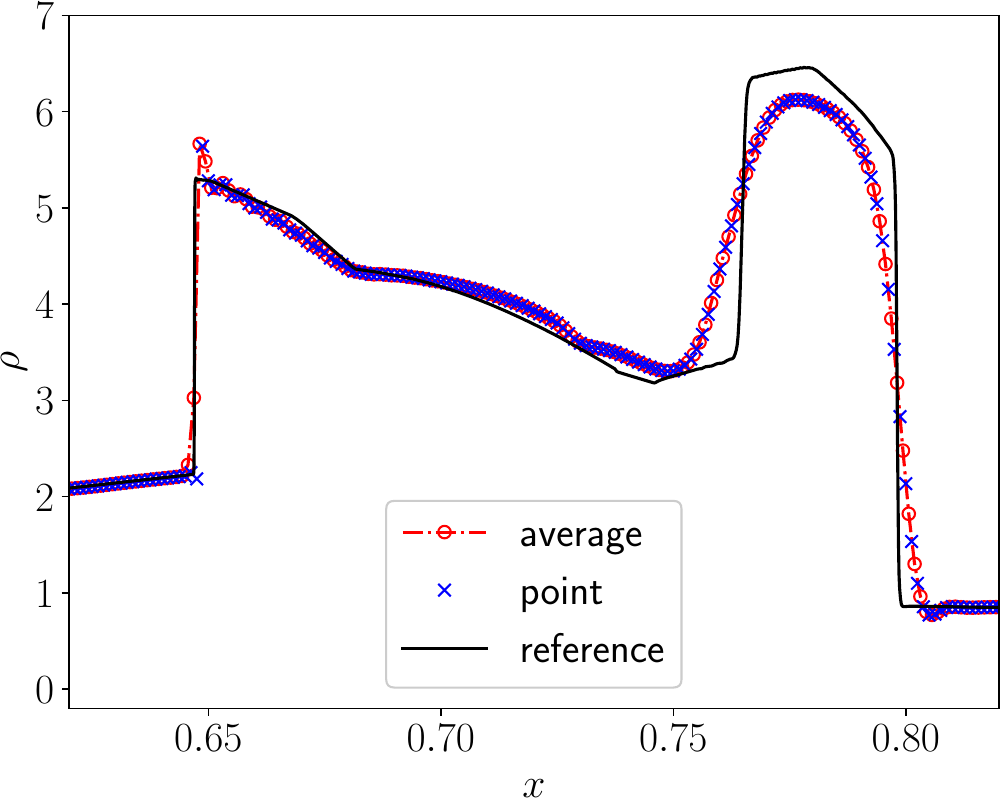}
	\end{subfigure}
	\begin{subfigure}[b]{0.24\textwidth}
		\centering
		\includegraphics[width=1.0\linewidth]{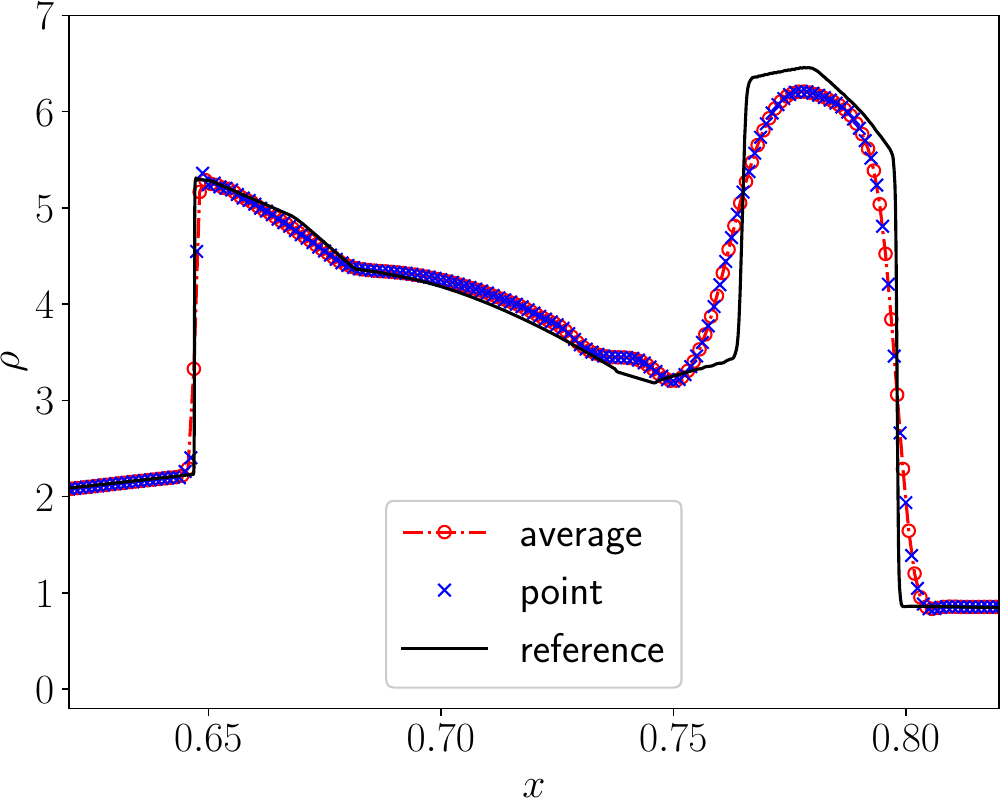}
	\end{subfigure}
	\begin{subfigure}[b]{0.24\textwidth}
		\centering
		\includegraphics[width=1.0\linewidth]{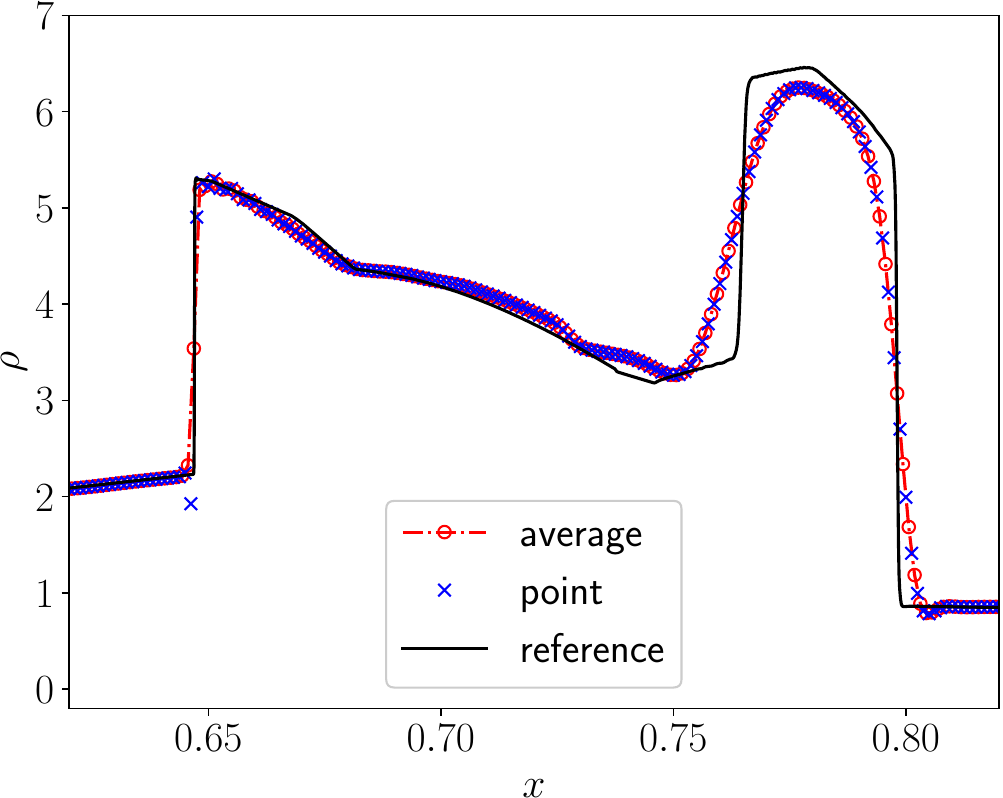}
	\end{subfigure}
	\begin{subfigure}[b]{0.24\textwidth}
		\centering
		\includegraphics[width=1.0\linewidth]{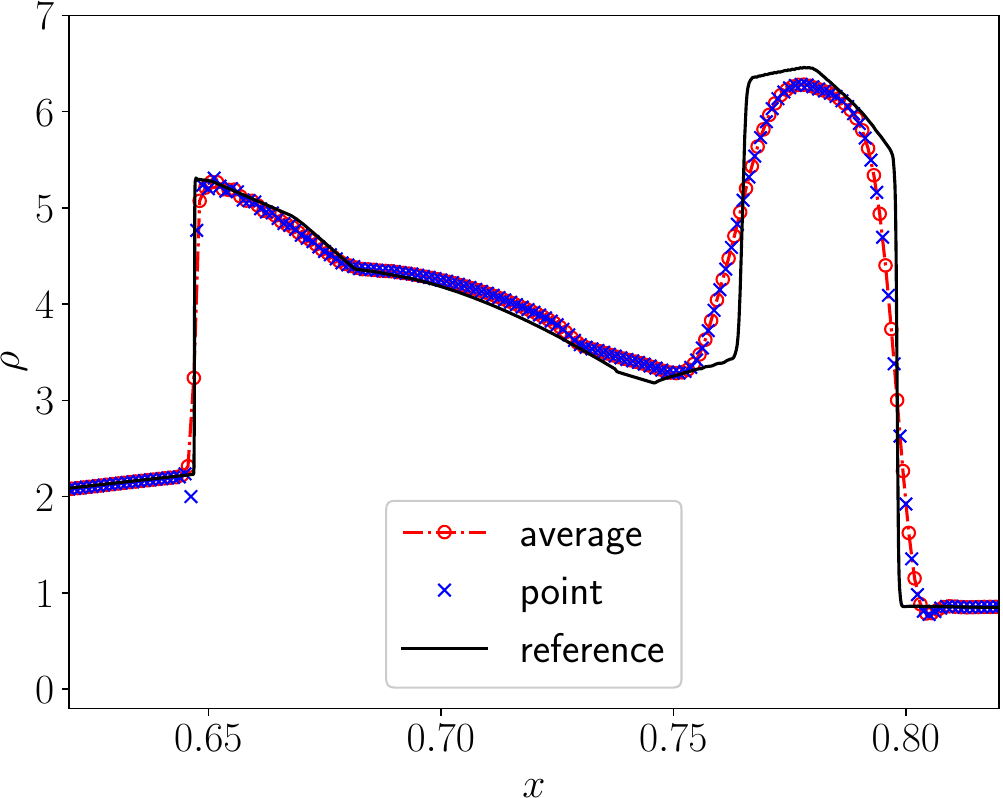}
	\end{subfigure}
	\caption{\Cref{ex:1d_blast_wave}, blast wave interaction.
		The numerical solutions are computed with the BP limitings for the cell average and point value updates on a uniform mesh of $800$ cells.
		The shock sensor-based limiting with $\kappa=1$ is used.
		The corresponding enlarged views in $[0.62, 0.82]$ are shown in the bottom row.}
	\label{fig:1d_blast_wave_coarse_mesh_kappa=1}
\end{figure}
\end{example}

\begin{remark}\rm
   In the numerical tests, the maximal CFL numbers for stability are obtained by experiments.
   Note that the constraints \cref{eq:1d_convex_combination_dt} and \cref{eq:1d_pnt_llf_dt} are used to guarantee the BP property, while the reduction of the CFL numbers is due to the stability issue for different FVS and power law reconstruction.
\end{remark}

\section{Conclusion}\label{sec:conclusion}
In the active flux (AF) methods, the way how point values at cell interfaces are updated is essential to achieve stability and high-order accuracy.
The point value update based on Jacobian splitting (JS) may lead to the so-called transonic issue for nonlinear problems due to inaccurate estimation of the upwind direction.
This paper proposed to use the flux vector splitting (FVS) for the point value update instead of the JS,
which keeps the continuous reconstruction as the original AF methods,
and offers a natural and uniform remedy to the transonic issue.
To further improve the robustness of the AF methods, this paper developed bound-preserving (BP) AF methods for general one-dimensional hyperbolic conservation laws,
achieved by blending the high-order AF methods with the first-order local Lax-Friedrichs (LLF) or Rusanov methods for both the cell average and point value updates,
where the convex limiting and scaling limiter were employed, respectively.
For scalar conservation laws, the blending coefficient was determined based on the global or local maximum principle,
while for the compressible Euler equations, it was obtained by enforcing the positivity of density and pressure.
Some challenging benchmark tests were conducted based on different choices of the point value update, including the JS, LLF, Steger-Warming, and Van Leer-H\"anel FVS.
The numerical results confirmed the accuracy, BP property, and shock-capturing ability of our methods,
and also showed that the LLF FVS is generally superior to others in terms of the CFL number and shock-capturing ability.
Our future work will include, among others, extending the current BP limitings to two-dimensional cases.
We may also explore other ways to further suppress oscillations for the Euler equations \cite{Duan_2024_Active_SJoSCa}.


\section*{Acknowledgement}

JD was supported by an Alexander von Humboldt Foundation Research fellowship CHN-1234352-HFST-P. CK and WB acknowledge funding by the Deutsche Forschungsgemeinschaft (DFG, German Research Foundation) within \textit{SPP 2410 Hyperbolic Balance Laws in Fluid Mechanics: Complexity, Scales, Randomness (CoScaRa)}, project number 525941602.

\newcommand{\etalchar}[1]{$^{#1}$}


\appendix

\section{1D power law reconstruction for point value update}\label{sec:1d_pwl_limiters}
When the numerical solutions contain discontinuities, the computation of the derivatives \cref{eq:parabolic_av} or \cref{eq:parabolic_pnt} based on the parabolic reconstructions may cause oscillations. Thus, it is reasonable to seek finite difference approximations based on differentiating a modified reconstruction with improved monotonicity properties.
This section only considers the scalar case and can be extended to systems of equations in a component-wise fashion.

The power law reconstruction proposed in \cite{Barsukow_2021_active_JoSC} can be used to replace the original parabolic reconstruction to achieve monotonicity on some occasions.
It is shown in Theorem 5 in \cite{Barsukow_2021_active_JoSC} that the extremum is not avoidable in the cell $I_i=[x_{\xl}, x_{\xr}]$ for continuous reconstructions if the cell average lies outside the range of the point values $(\bar{u}_i - u_{\xl})(u_{\xr} - \bar{u}_i) < 0$.
The parabola is monotone, and thus no action is required when $(2u_{\xl} + u_{\xr})/3 < \bar{u}_i < (u_{\xl} + 2u_{\xr})/3$
or $(2u_{\xl} + u_{\xr})/3 > \bar{u}_i > (u_{\xl} + 2u_{\xr})/3$.
Upon defining $r = \dfrac{u_{i+1/2} - \bar{u}_i}{\bar{u}_i - u_{i-1/2 }}$, one can equivalently express that the parabola is monotone when $1/2 < r < 2$.
In both these cases, the parabolic reconstruction is used,
and the derivatives are obtained by \cref{eq:parabolic_av} or \cref{eq:parabolic_pnt}.
Otherwise, the following power law reconstruction is used.

\begin{proposition}[Barsukow \cite{Barsukow_2021_active_JoSC}]\rm
	The power law reconstruction
	\begin{equation}\label{eq:pwl_function}
		\begin{cases}
			u_{\texttt{pwl},1}(x) = u_{\xl} + (u_{\xr} - u_{\xl})\left(\dfrac{x - x_i}{\Delta x_i} + \dfrac{1}{2}\right)^r, &\text{if}~~ r > 2\\
			u_{\texttt{pwl},2}(x) = u_{\xr} - (u_{\xr} - u_{\xl})\left(\dfrac{1}{2} - \dfrac{x - x_i}{\Delta x_i}\right)^{1/r}, &\text{if}~~ 0 < r < 1/2\\
		\end{cases}
	\end{equation}
	is monotone and satisfies
	\begin{equation*}
		u_{\texttt{pwl}, l}(x_{\xl}) = u_{\xl},~
		u_{\texttt{pwl}, l}(x_{\xr}) = u_{\xr},~
		\dfrac{1}{\Delta x_i}\int_{I_i} u_{\texttt{pwl}, l}(x) ~\dd x = \bar{u}_{i},~ l=1,2.
	\end{equation*}
\end{proposition}

\begin{figure}[hptb!]
	\centering
	\begin{subfigure}[b]{0.35\textwidth}
		\centering
		\includegraphics[width=1.0\linewidth]{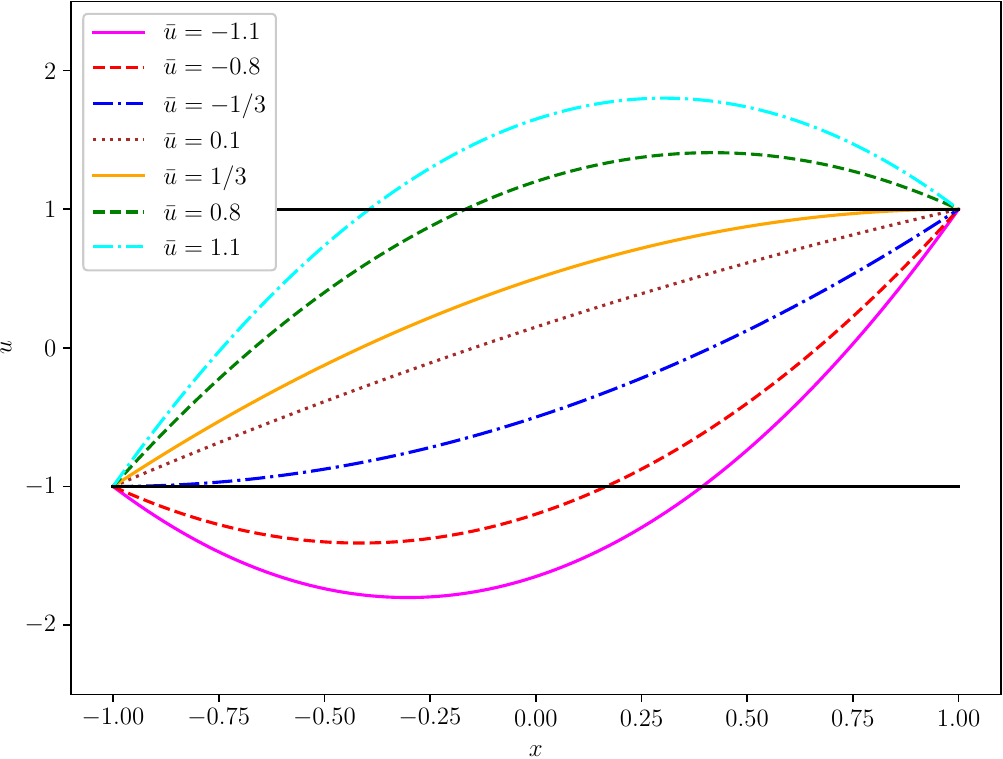}
	\end{subfigure}
	~
	\begin{subfigure}[b]{0.35\textwidth}
		\centering
		\includegraphics[width=1.0\linewidth]{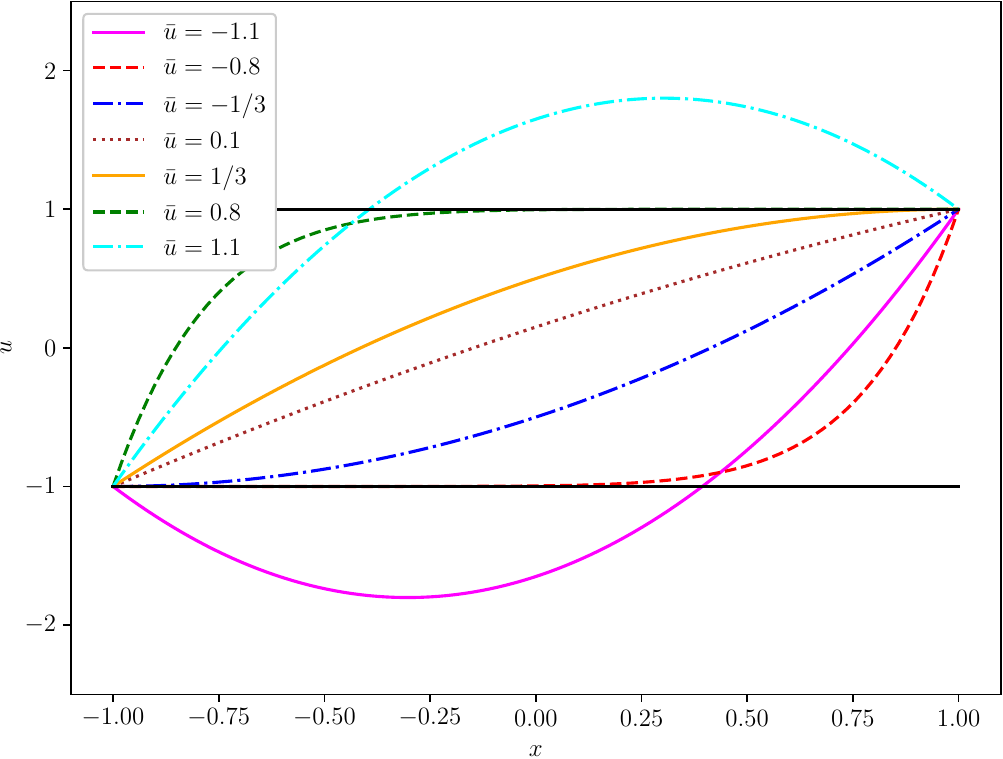}
	\end{subfigure}
	\caption{The parabolic \cref{eq:parabolic_reconstruction} and power law reconstruction \cref{eq:pwl_function} obtained with different cell averages $\{-1.1, -0.8, -1/3, ~0.1, ~1/3,~ 0.8, ~1.1\}$, and fixed point values as $-1$ and $1$ at the left and right interfaces.}
	\label{fig:1d_limited_reconstructions}
\end{figure}

A comparison between the parabolic reconstruction \cref{eq:parabolic_reconstruction} and power law reconstruction \cref{eq:pwl_function} is given in \cref{fig:1d_limited_reconstructions} with point values fixed as $-1$ and $1$ at the interfaces, and different cell averages $\{-1.1, -0.8, -1/3, 0.1, 1/3, 0.8, 1.1\}$.
One can observe monotone profiles for the power law reconstruction when the cell average lies between the two point values.
Based on \cref{eq:pwl_function}, the derivatives can be computed directly
\begin{equation*}
	\begin{cases}
		u'_{\texttt{pwl},1}(x) = \dfrac{u_{\xr} - u_{\xl}}{\Delta x_i}r\left(\dfrac{x - x_i}{\Delta x_i} + \dfrac{1}{2}\right)^{r - 1},   &\text{if}~~ r > 2, \\
		u'_{\texttt{pwl},2}(x) = \dfrac{u_{\xr} - u_{\xl}}{\Delta x_i}\dfrac{1}{r}\left(\dfrac{1}{2} - \dfrac{x - x_i}{\Delta x_i}\right)^{1/r - 1},  &\text{if}~~ 0 < r < 1/2. \\
	\end{cases}
\end{equation*}
At the left interface, the derivative is
\begin{equation}\label{eq:1d_pwl_left}
	\begin{cases}
		u'_{\texttt{pwl},1}(x_{\xl}^+) = 0,  &\text{if}~~ r > 2, \\
		u'_{\texttt{pwl},2}(x_{\xl}^+) = \dfrac{u_{\xr} - u_{\xl}}{\Delta x_i}\dfrac{1}{r},  &\text{if}~~ 0 < r < 1/2, \\
	\end{cases}
\end{equation}
and at the right interface, the derivative is
\begin{equation}\label{eq:1d_pwl_right}
	\begin{cases}
		u'_{\texttt{pwl},1}(x_{\xr}^-) = \dfrac{u_{\xr} - u_{\xl}}{\Delta x_i}r,  &\text{if}~~ r > 2, \\
		u'_{\texttt{pwl},2}(x_{\xr}^-) = 0,  &\text{if}~~ 0 < r < 1/2. \\
	\end{cases}
\end{equation}
To avoid computational issues, when $r\not\in[1/50, 50]$, the parabolic reconstruction is adopted directly.

For the FVS, as the cell average of the flux can be obtained through Simpson's rule,
$\bar{f}_i = (f_{\xl} + 4f_i + f_{\xr}) / 6$,
the flux derivatives can be computed by \cref{eq:1d_pwl_left}-\cref{eq:1d_pwl_right}.

\begin{remark}\rm
	In \cite{Abgrall_2023_Extensions_EMMaNA}, it is mentioned that if the signs of the derivatives of the parabolic reconstruction and the first-order reconstruction are the same, then the parabolic reconstruction is adopted.
	This strategy is not employed in this paper as the numerical results may be worse.
\end{remark}

\section{Additional numerical results}

\begin{example}[Shu-Osher shock-entropy wave interaction]\label{ex:1d_shock_entropy}
	This test is used to check a scheme’s ability to resolve a complex solution with both strong and weak shocks and highly oscillatory but smooth waves.
	The initial data are
	\begin{equation*}
		(\rho, v, p) = \begin{cases}
			(3.857143, ~2.629369, ~10.33333), &\text{if}~ x<-4,\\
			(1 + 0.2\sin(5x), ~0, ~1), &\text{otherwise},\\
		\end{cases}
	\end{equation*}
	on the domain $[-5,5]$ with $\gamma=1.4$.
	This test is solved until $T=1.8$.

The reference solution is obtained with the fifth-order WENO finite difference scheme on a mesh of $2000$ grid points.
The solutions computed with the CFL number $0.3$ based on the JS and different FVS without limiting on a mesh of $400$ cells are displayed in \cref{fig:1d_shock_entropy}.
There are very minor differences between the JS and FVS in the enlarged view.
We also check the maximal CFL numbers for each kind of point value update such that the simulation is stable,
which are around $0.43$, $0.42$, $0.43$, $0.31$ for the JS, LLF, SW, and VH FVS.
If the power law reconstruction is activated for computing the derivatives in the point value update,
the corresponding CFL numbers should be reduced to achieve stability,
which are around $0.13$, $0.15$, $0.16$, $0.14$.

\begin{figure}[hptb!]
	\centering
	\begin{subfigure}[b]{0.35\textwidth}
		\centering
		\includegraphics[width=1.0\linewidth]{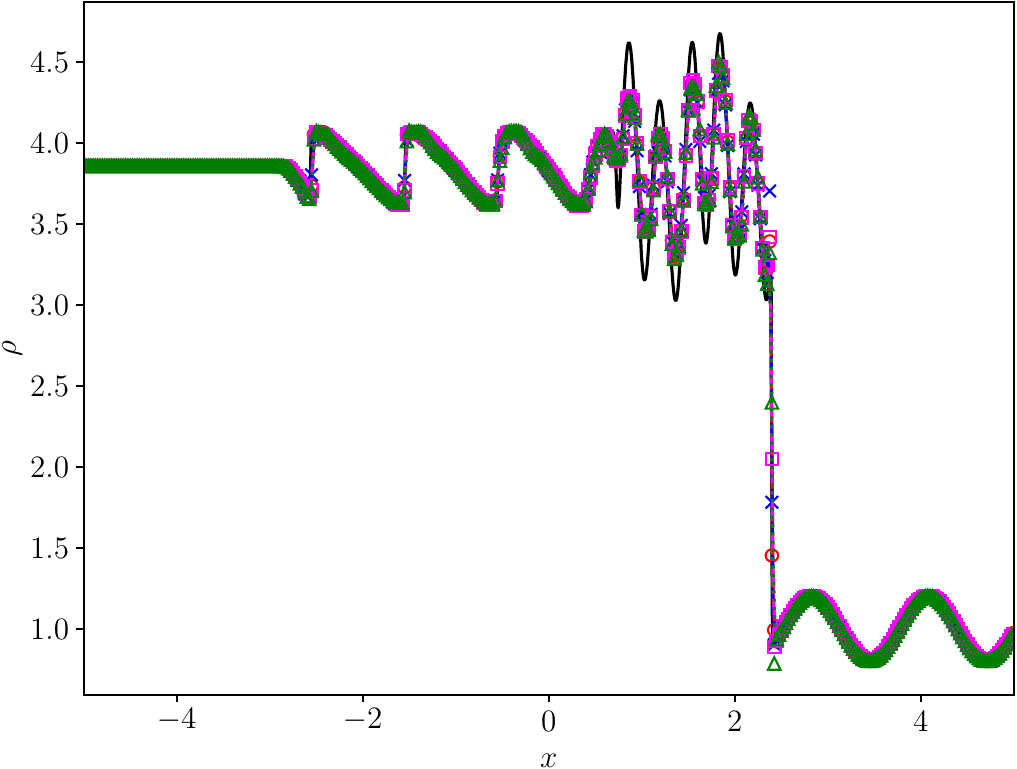}
	\end{subfigure}
	\begin{subfigure}[b]{0.35\textwidth}
		\centering
		\includegraphics[width=1.0\linewidth]{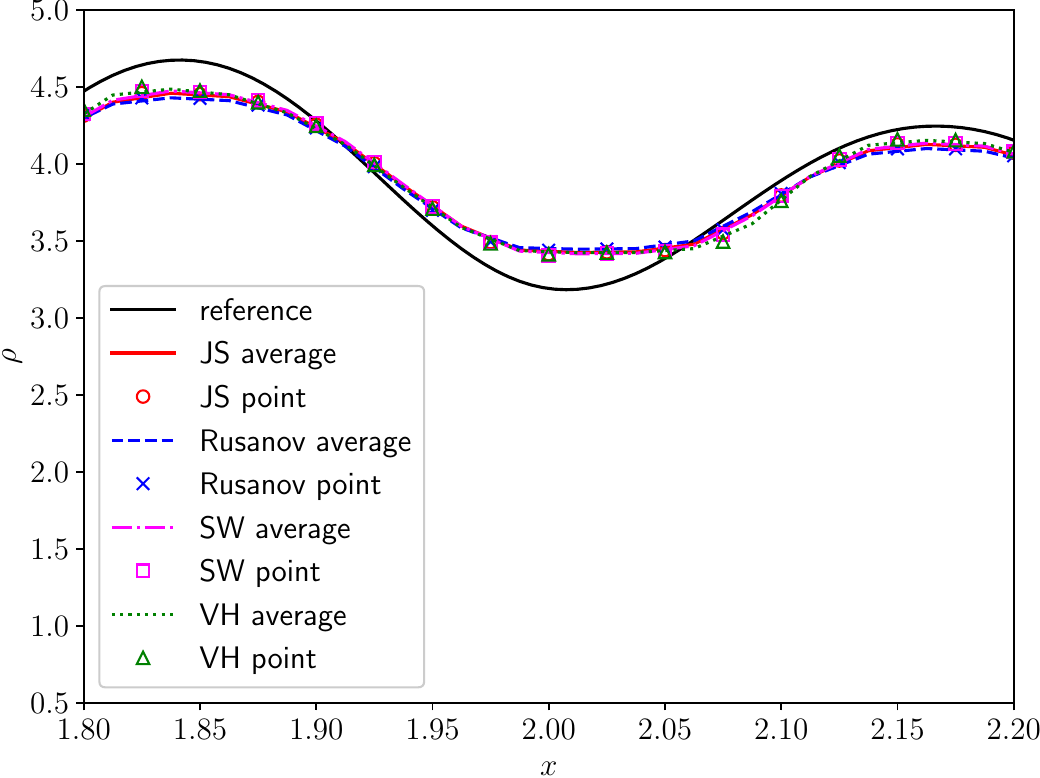}
	\end{subfigure}
	\caption{\Cref{ex:1d_shock_entropy}, shock-entropy wave interaction.
		The numerical solutions are obtained by using the JS and different FVS without limiting on a uniform mesh of $400$ cells.
		The enlarged view in $x\in[1.8, 2.2]$ is shown on the right.}
	\label{fig:1d_shock_entropy}
\end{figure}
\end{example}

\begin{example}[Double rarefaction problem]\label{ex:1d_double_rarefaction_supp}
\cref{fig:1d_double_rarefaction_v_p} shows the velocity and pressure computed with $400$ cells and the BP limitings for the cell average and point value updates, while without the power law reconstruction.

\begin{figure}[hptb!]
	\centering
	\begin{subfigure}[b]{0.24\textwidth}
		\centering
		\includegraphics[width=\linewidth]{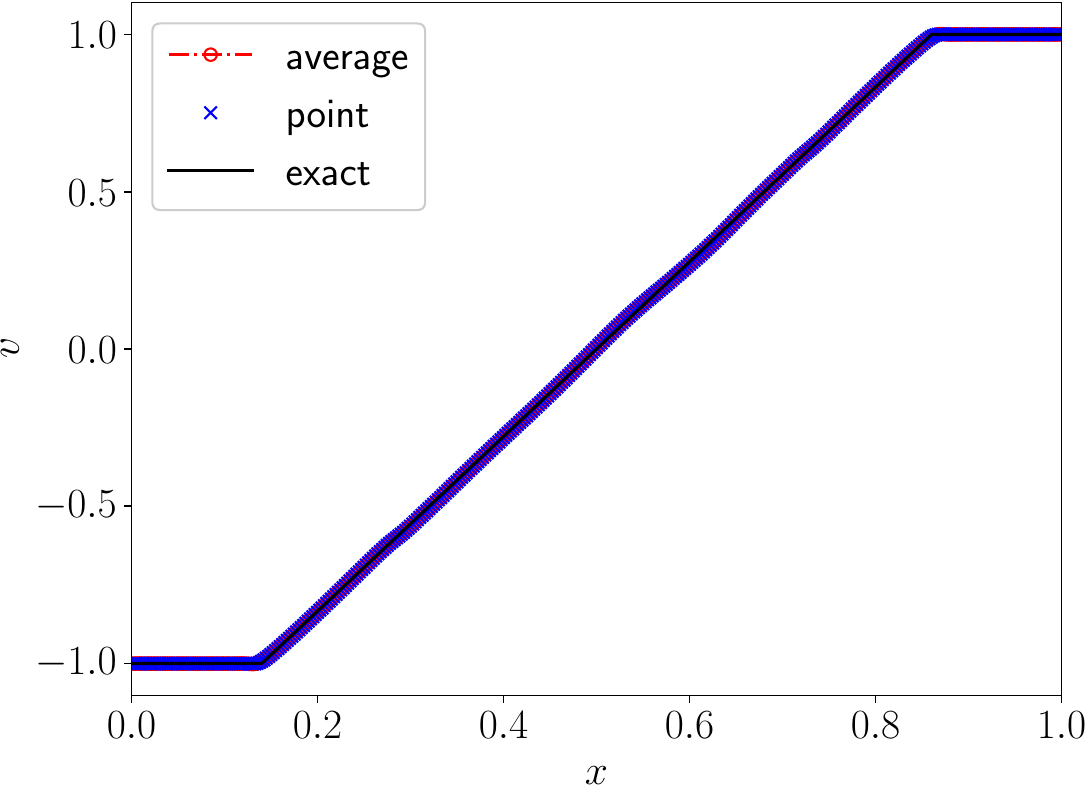}
	\end{subfigure}
	\begin{subfigure}[b]{0.24\textwidth}
		\centering
		\includegraphics[width=\linewidth]{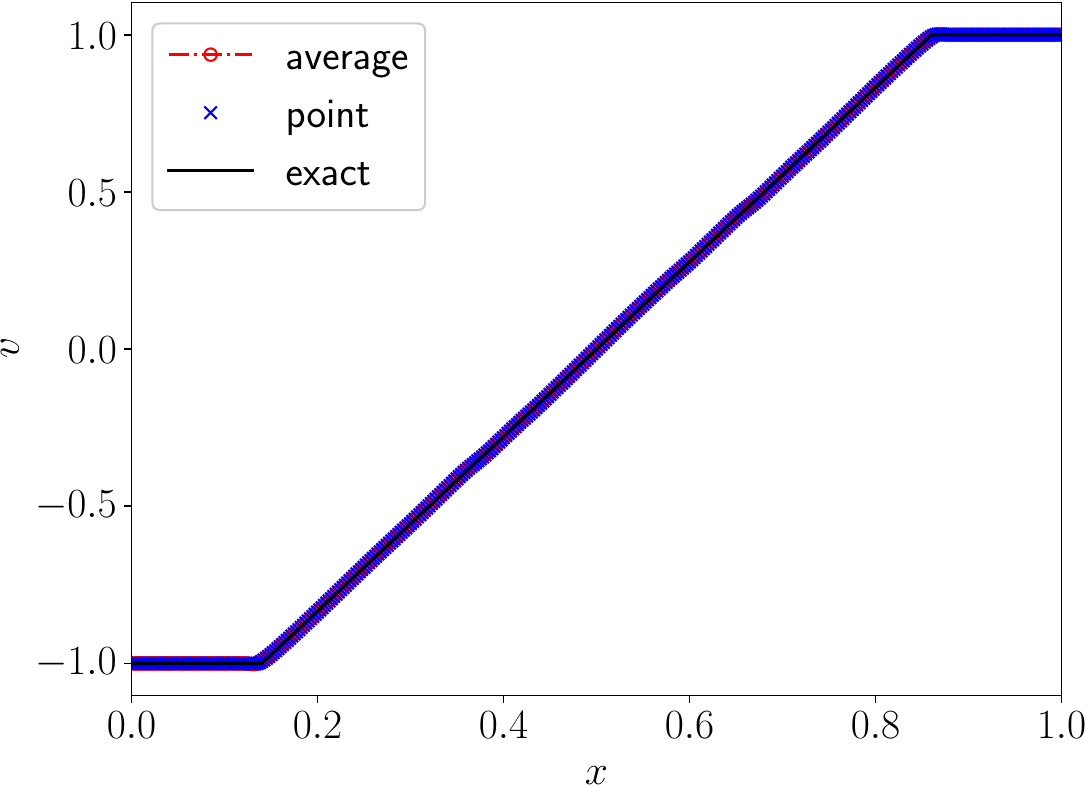}
	\end{subfigure}
	\begin{subfigure}[b]{0.24\textwidth}
		\centering
		\includegraphics[width=\linewidth]{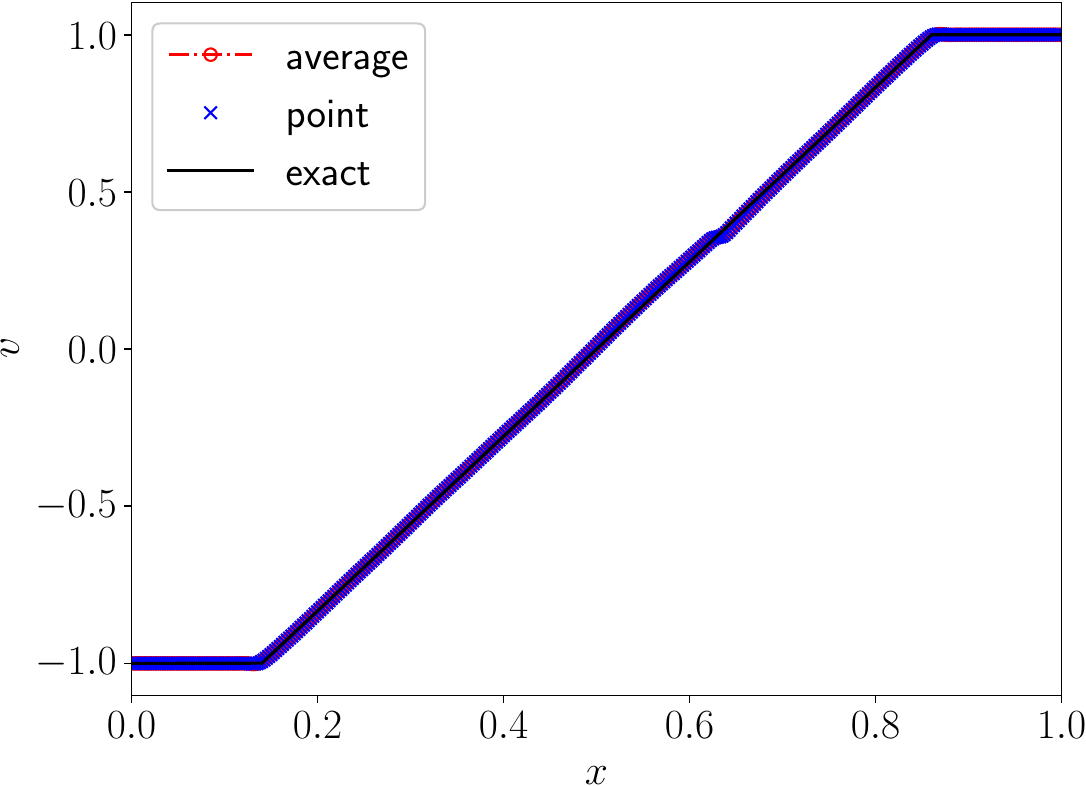}
	\end{subfigure}
	\begin{subfigure}[b]{0.24\textwidth}
		\centering
		\includegraphics[width=\linewidth]{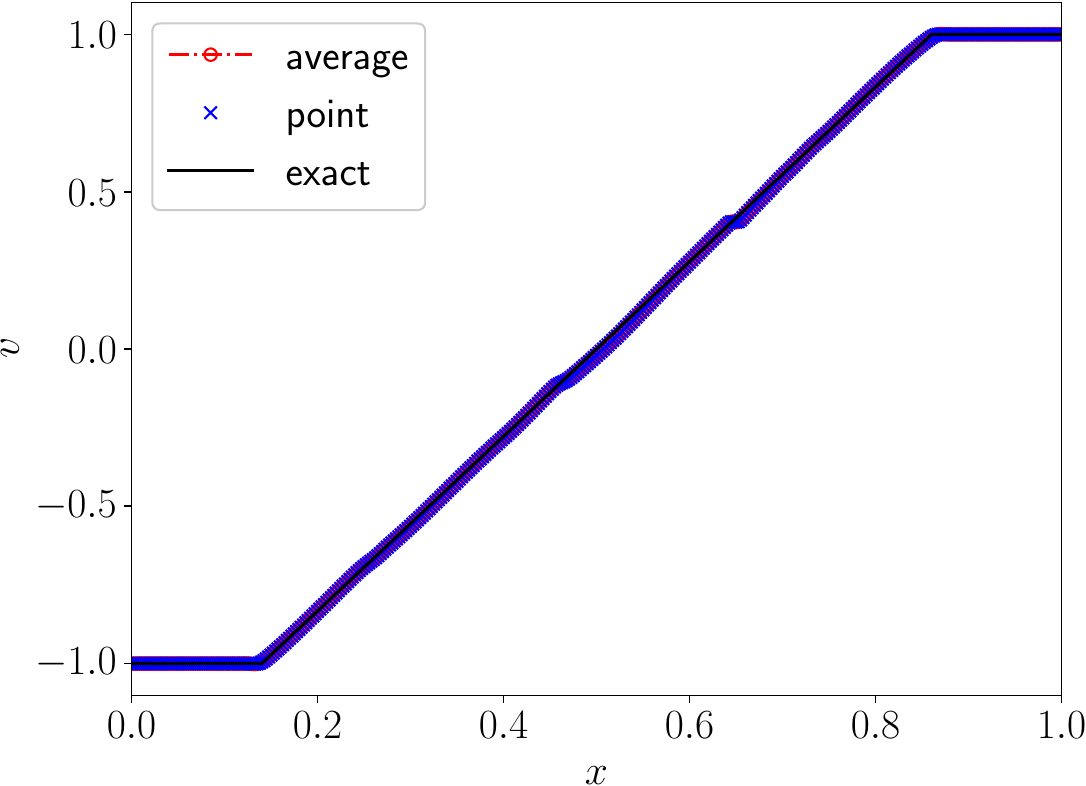}
	\end{subfigure}
	
	\begin{subfigure}[b]{0.24\textwidth}
		\centering
		\includegraphics[width=\linewidth]{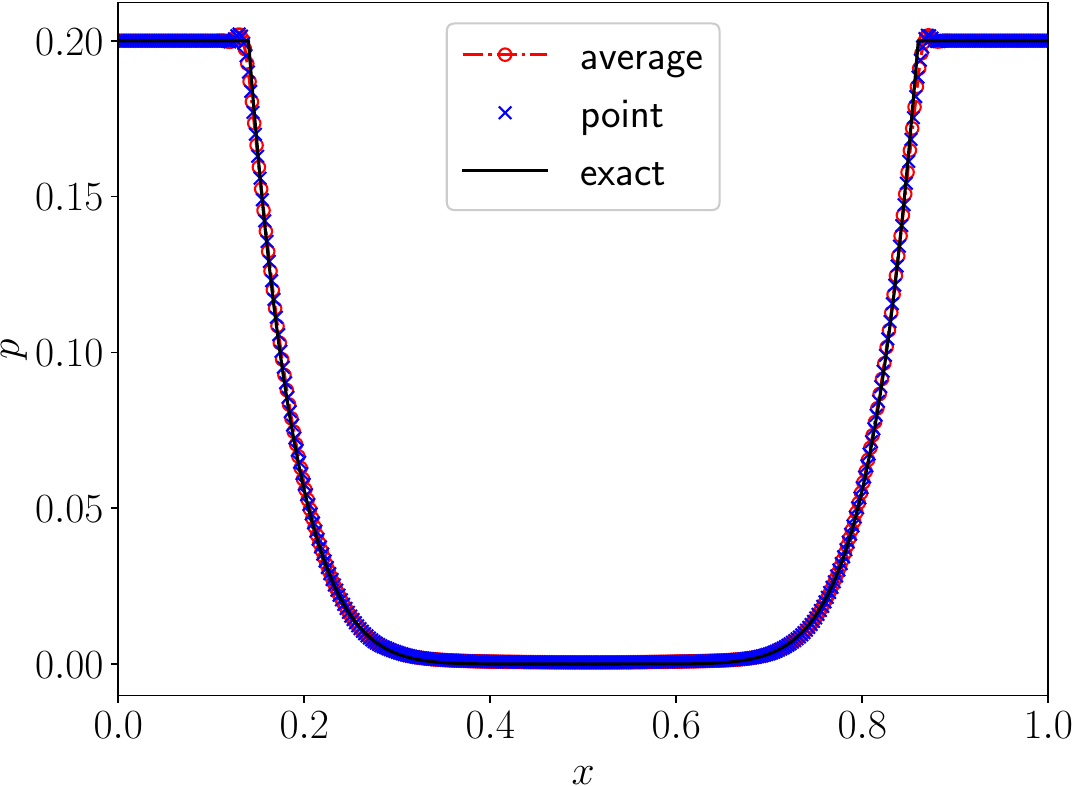}
	\end{subfigure}
	\begin{subfigure}[b]{0.24\textwidth}
		\centering
		\includegraphics[width=\linewidth]{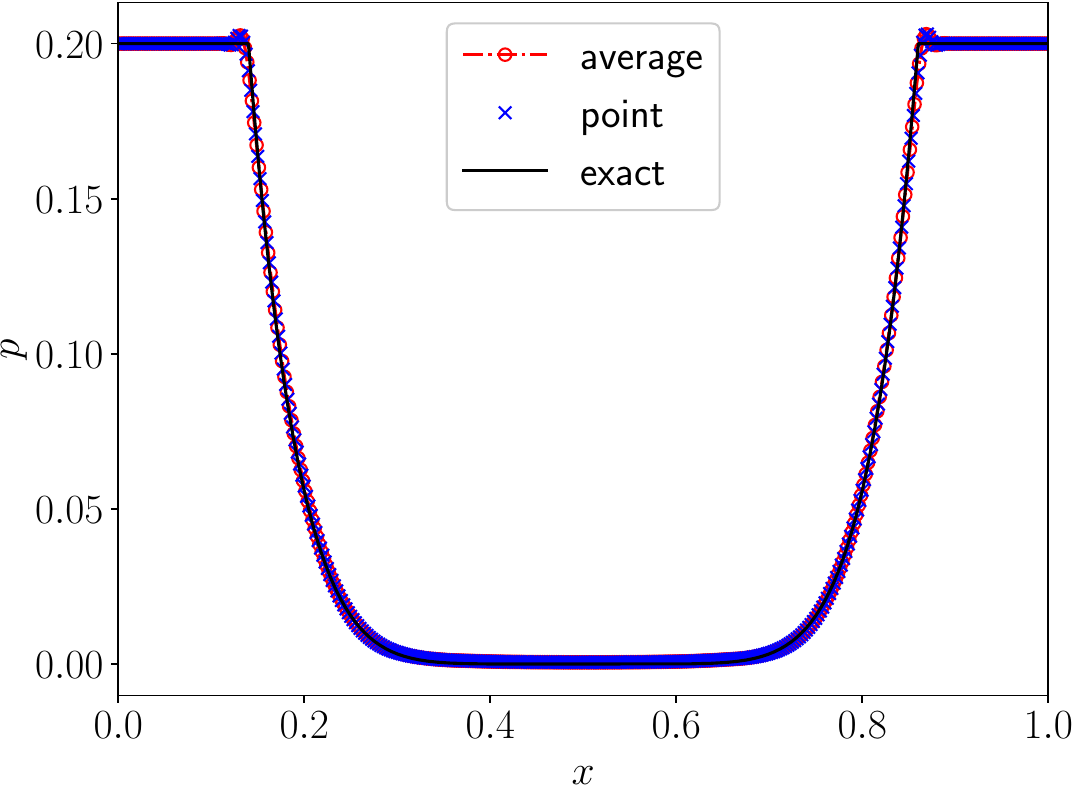}
	\end{subfigure}
	\begin{subfigure}[b]{0.24\textwidth}
		\centering
		\includegraphics[width=\linewidth]{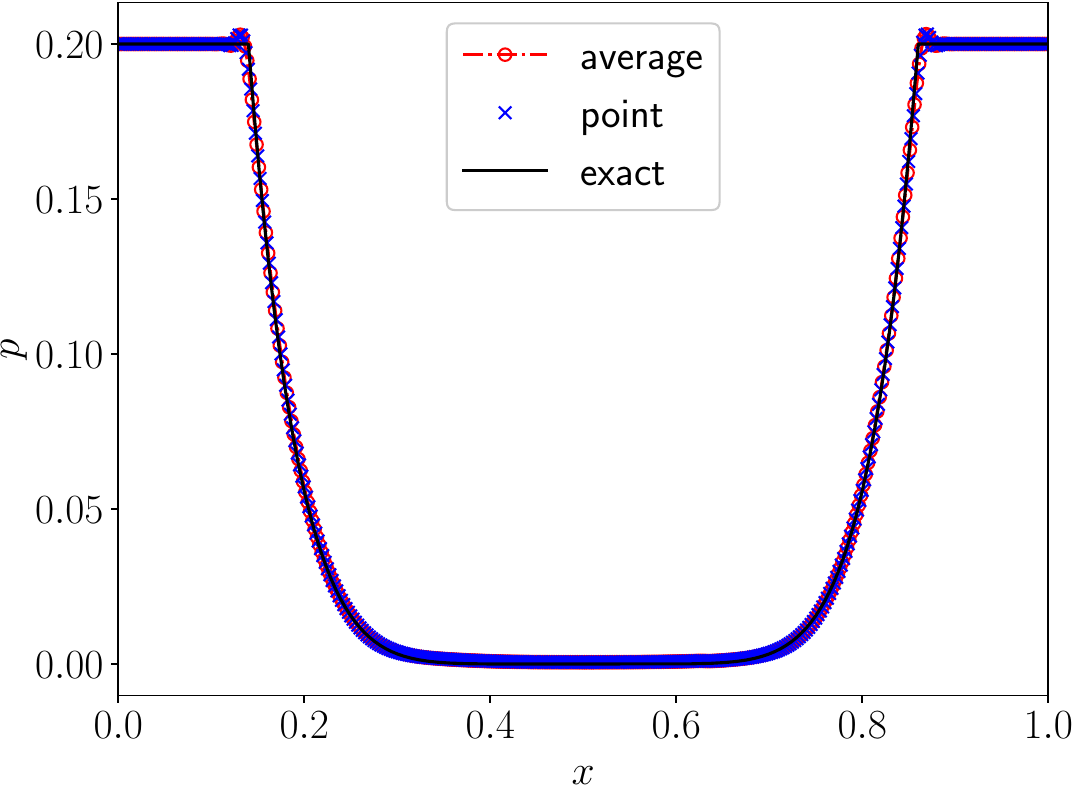}
	\end{subfigure}
	\begin{subfigure}[b]{0.24\textwidth}
		\centering
		\includegraphics[width=\linewidth]{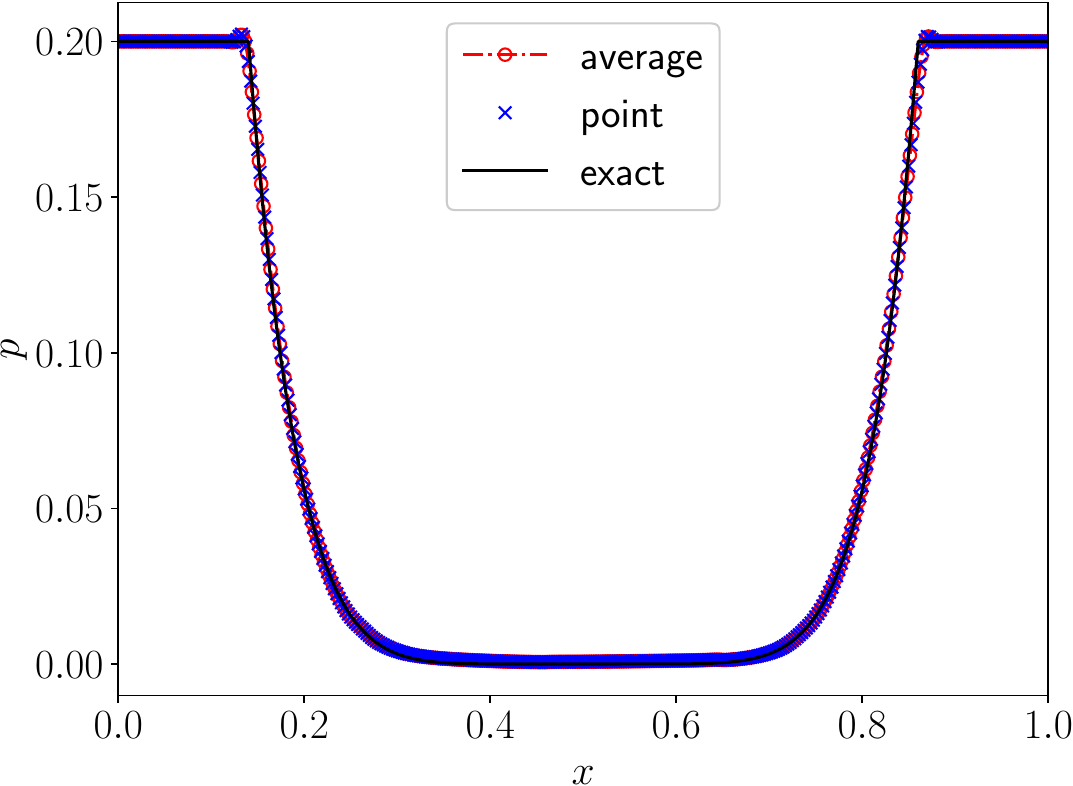}
	\end{subfigure}
	\caption{\Cref{ex:1d_double_rarefaction_supp}, double rarefaction Riemann problem.
		The velocity and pressure are computed with BP limitings for the cell average and point value updates on a uniform mesh of $400$ cells.
		The power law reconstruction is not used.
		From left to right: JS, LLF, SW, and VH FVS.}
	\label{fig:1d_double_rarefaction_v_p}
\end{figure}
\end{example}

\begin{example}[LeBlanc shock tube]\label{ex:1d_leblanc_supp}
\Cref{fig:1d_leblanc_coarse_mesh_v_p} shows the velocity and pressure computed on a uniform mesh of $400$ cells and the BP limitings for the cell average and point value updates,
and \cref{fig:1d_leblanc_fine_mesh_cfl0.4_v_p} shows the corresponding results with $6000$ cells.

\begin{figure}[hptb!]
	\centering
	\begin{subfigure}[b]{0.24\textwidth}
		\centering
		\includegraphics[width=\linewidth]{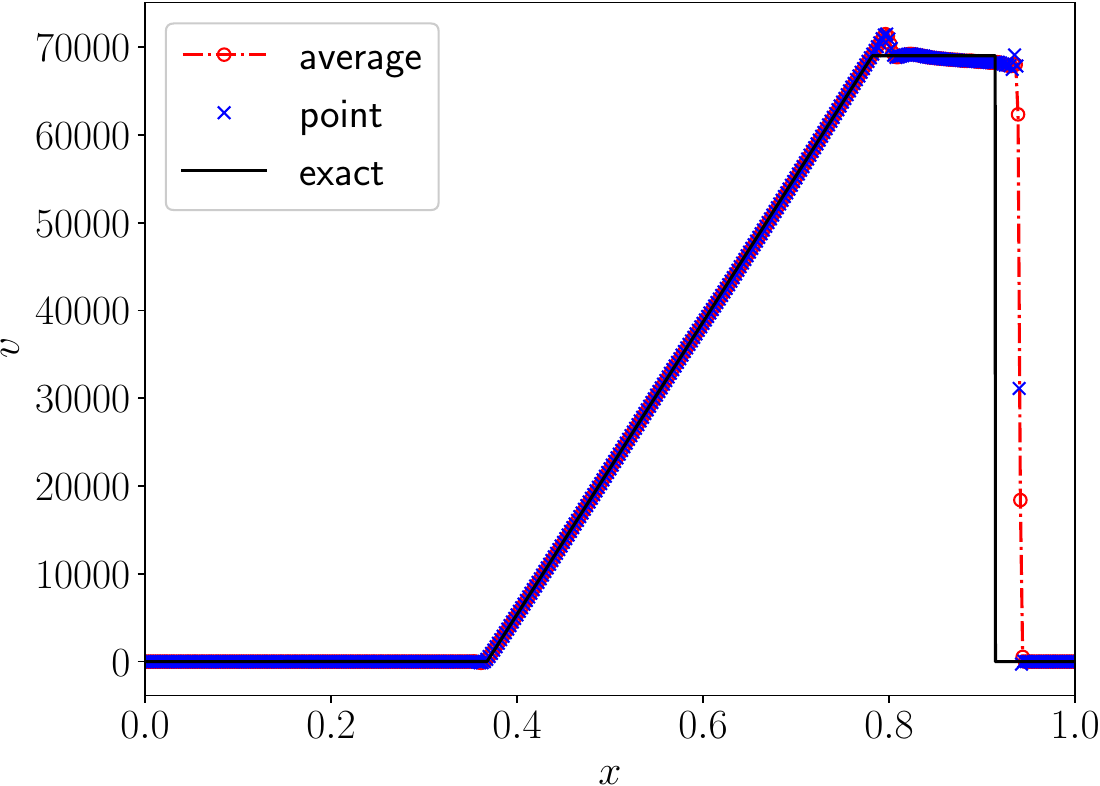}
	\end{subfigure}
	\begin{subfigure}[b]{0.24\textwidth}
		\centering
		\includegraphics[width=\linewidth]{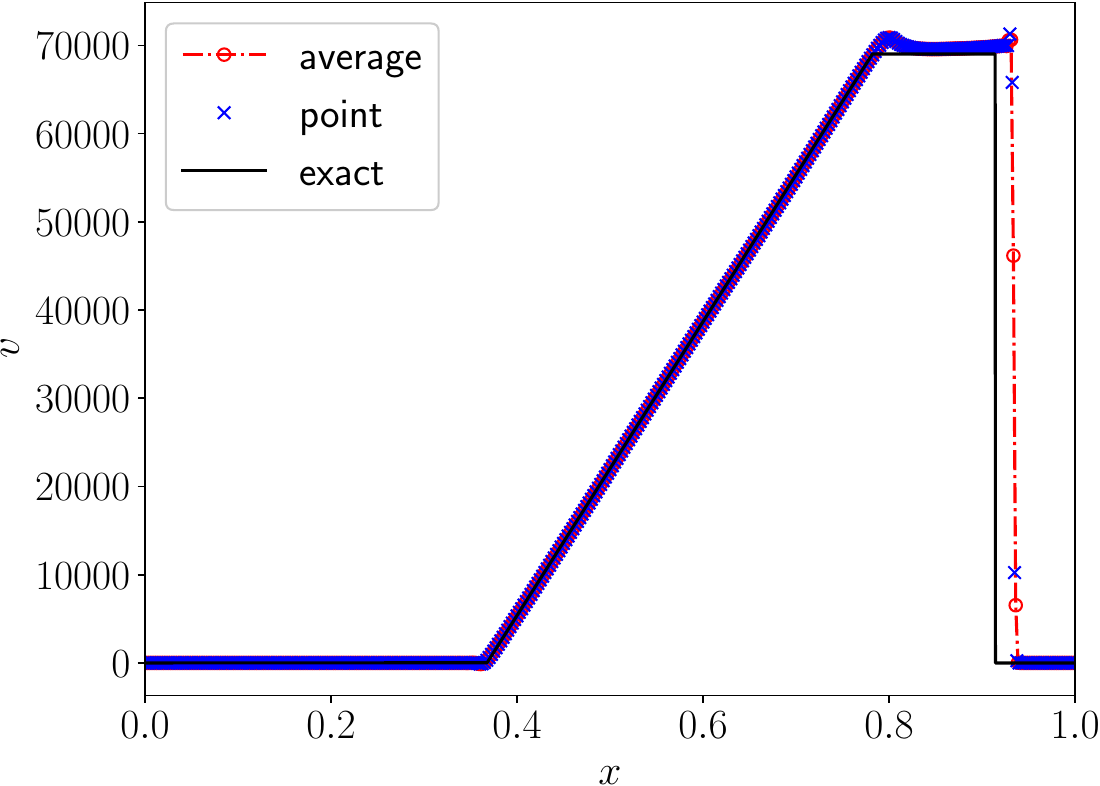}
	\end{subfigure}
	\begin{subfigure}[b]{0.24\textwidth}
		\centering
		\includegraphics[width=\linewidth]{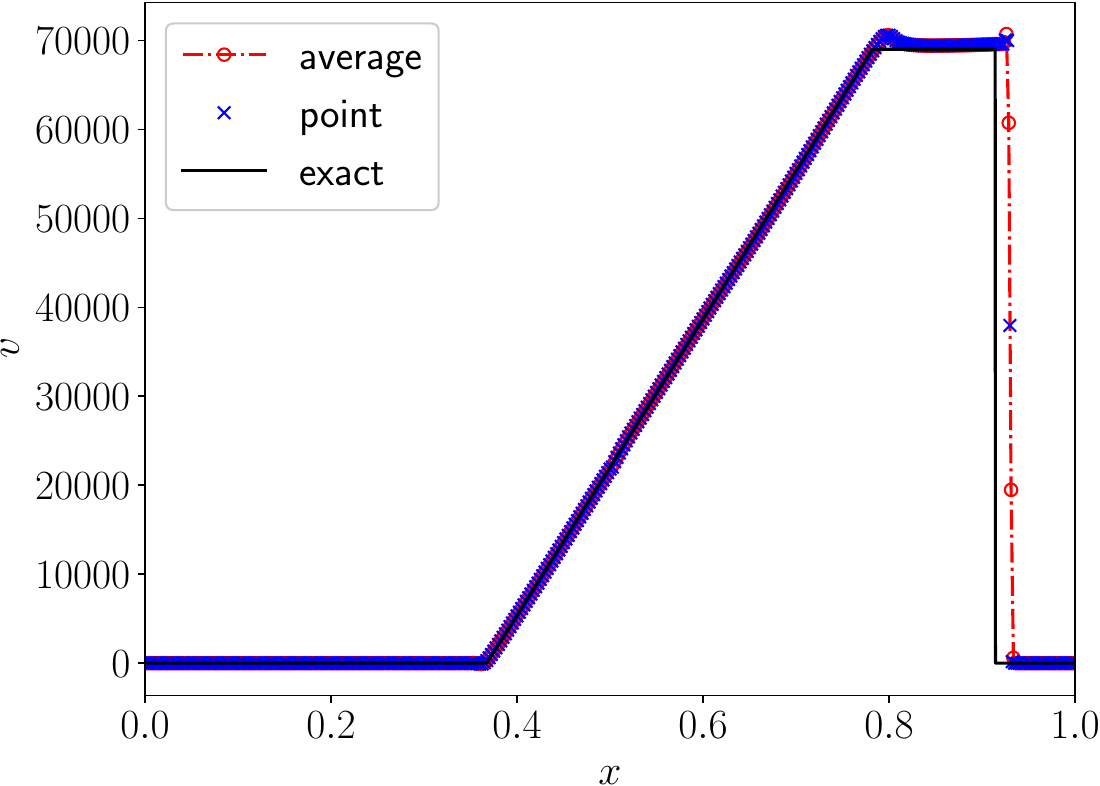}
	\end{subfigure}
	\begin{subfigure}[b]{0.24\textwidth}
		\centering
		\includegraphics[width=\linewidth]{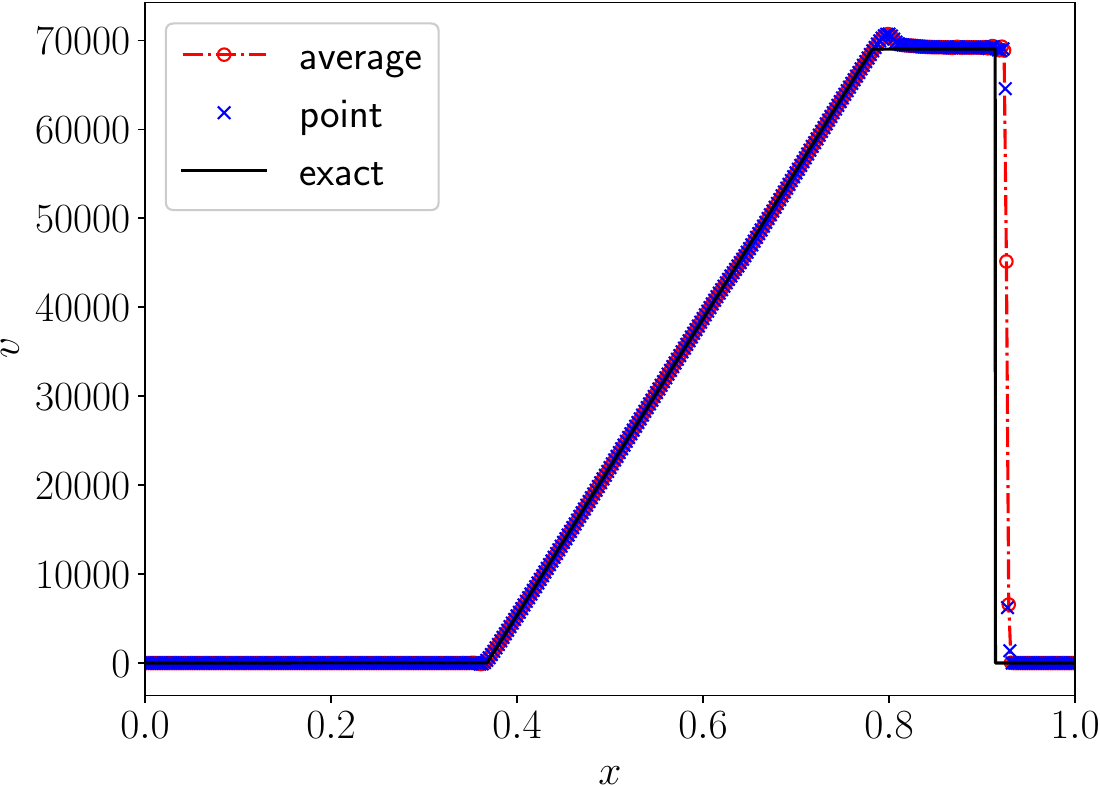}
	\end{subfigure}
	
	\begin{subfigure}[b]{0.24\textwidth}
		\centering
		\includegraphics[width=\linewidth]{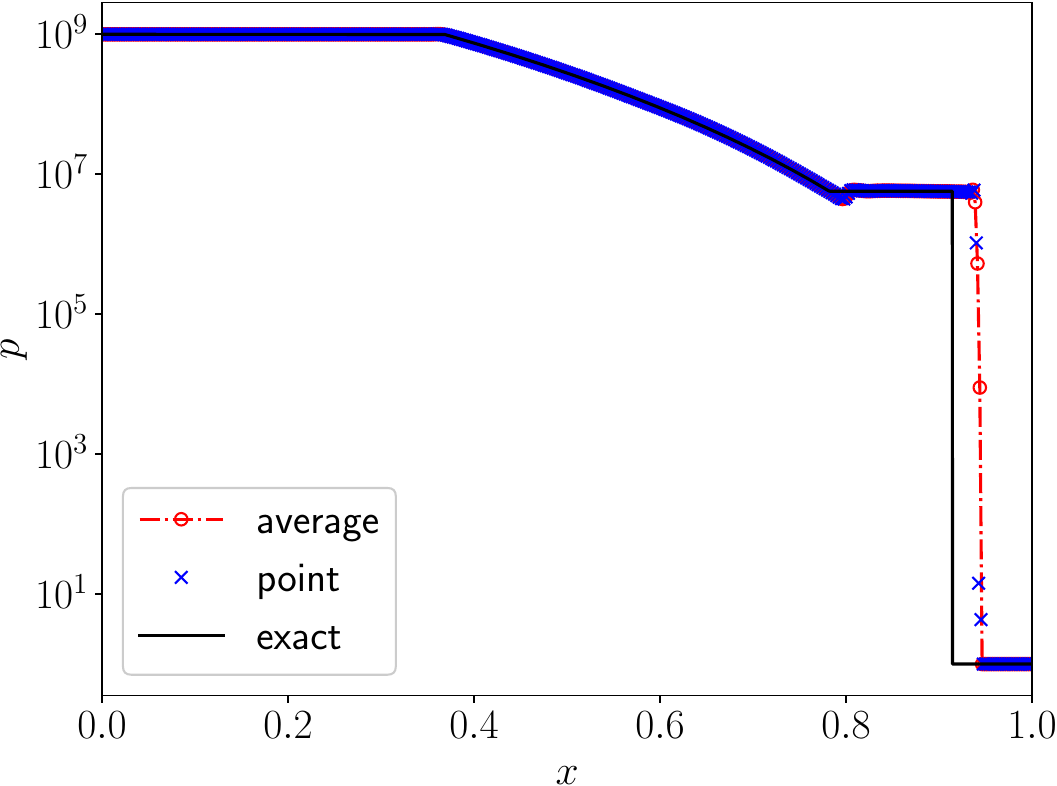}
	\end{subfigure}
	\begin{subfigure}[b]{0.24\textwidth}
		\centering
		\includegraphics[width=\linewidth]{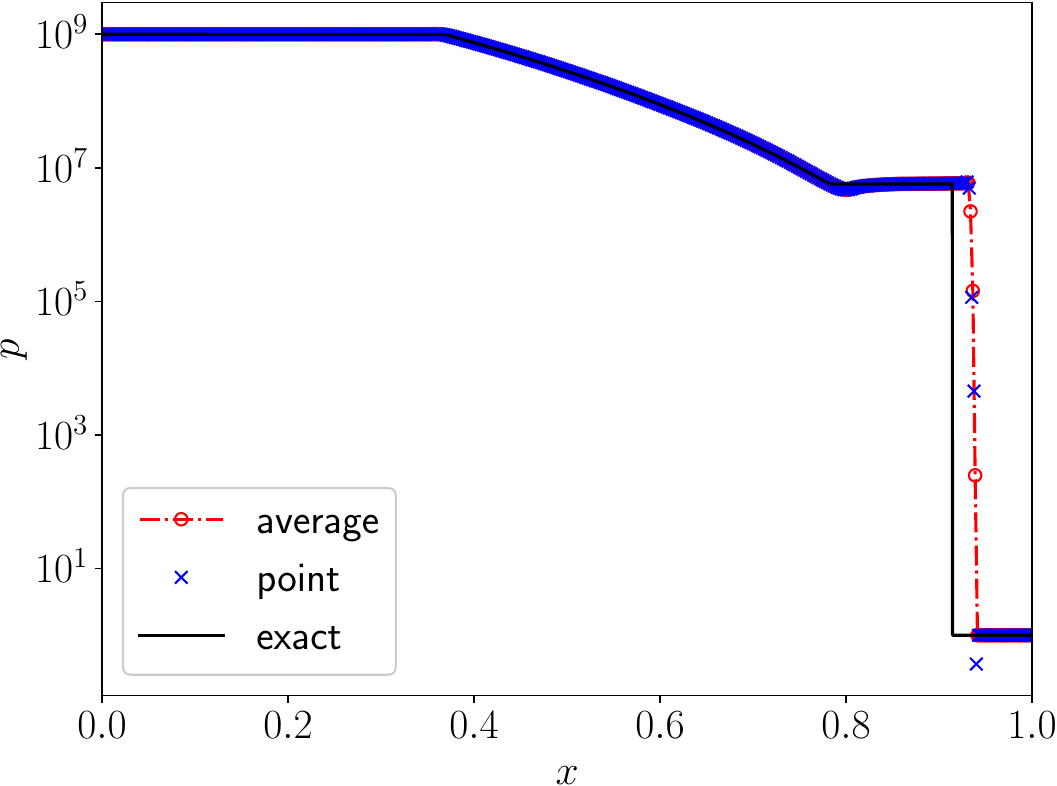}
	\end{subfigure}
	\begin{subfigure}[b]{0.24\textwidth}
		\centering
		\includegraphics[width=\linewidth]{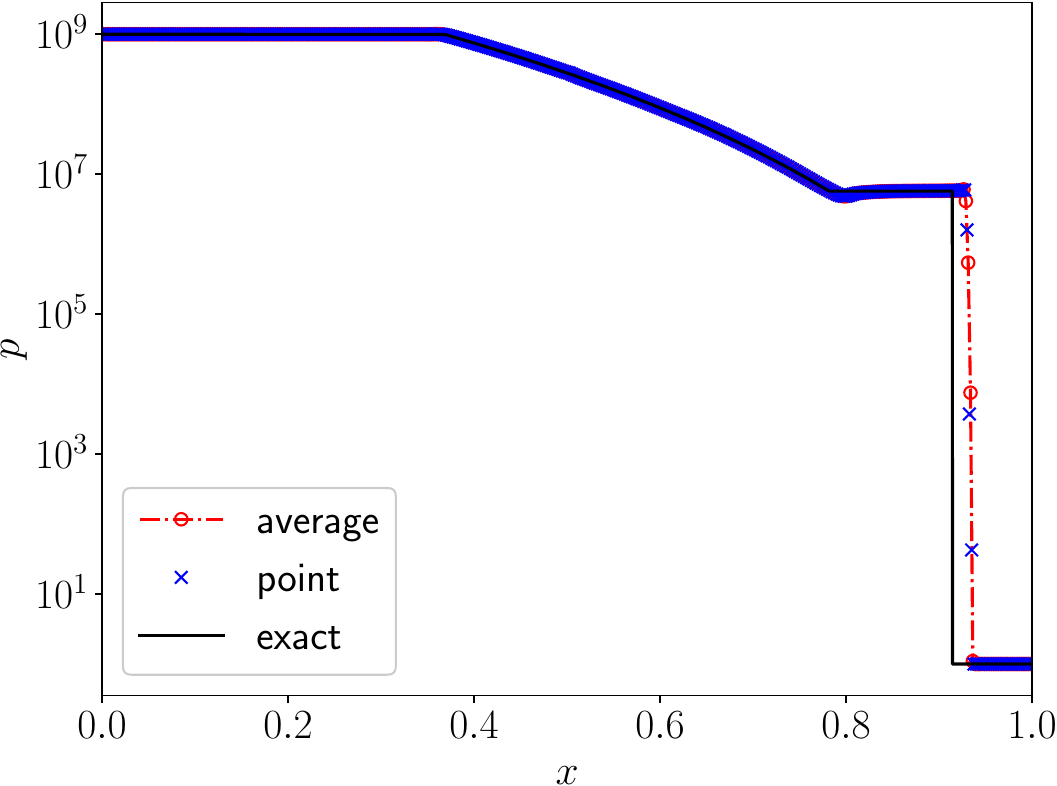}
	\end{subfigure}
	\begin{subfigure}[b]{0.24\textwidth}
		\centering
		\includegraphics[width=\linewidth]{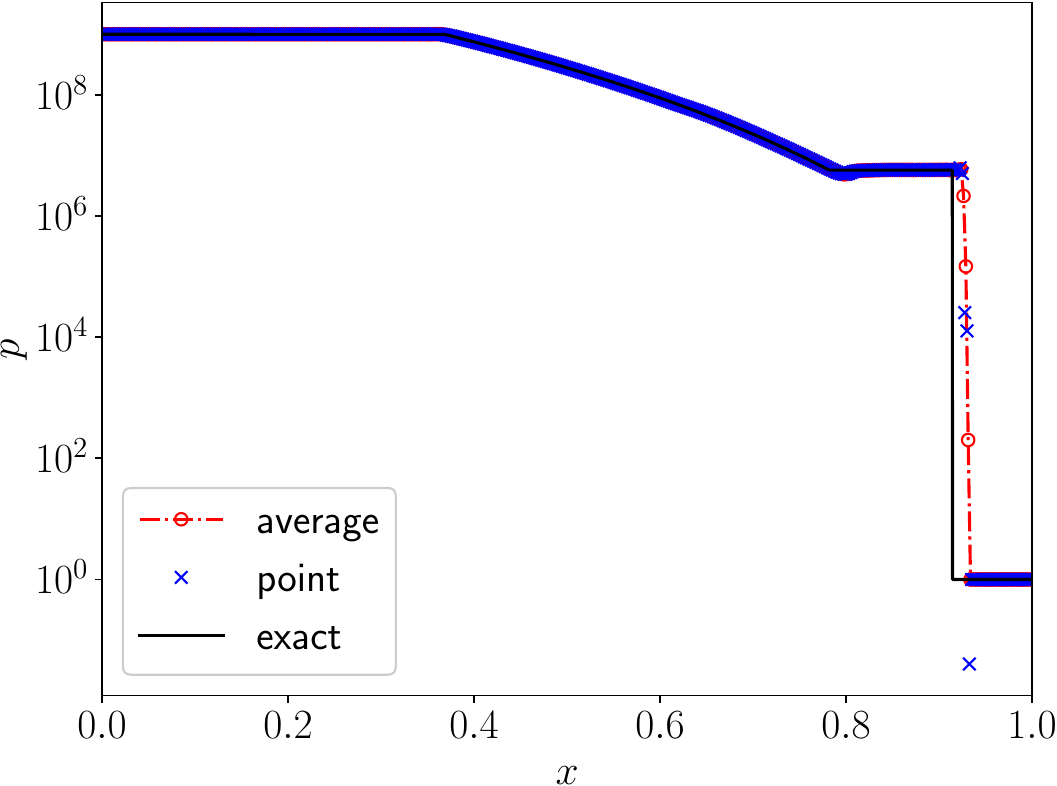}
	\end{subfigure}
	\caption{\Cref{ex:1d_leblanc_supp}, LeBlanc Riemann problem.
		The numerical solutions are computed with the BP limitings for the cell average and point value updates on a uniform mesh of $400$ cells.
		The CFL number is $0.4$ and the power law reconstruction is not used.
		From left to right: JS, LLF, SW, and VH FVS.}
	\label{fig:1d_leblanc_coarse_mesh_v_p}
\end{figure}

\begin{figure}[hptb!]
	\centering
	\begin{subfigure}[b]{0.24\textwidth}
		\centering
		\includegraphics[width=\linewidth]{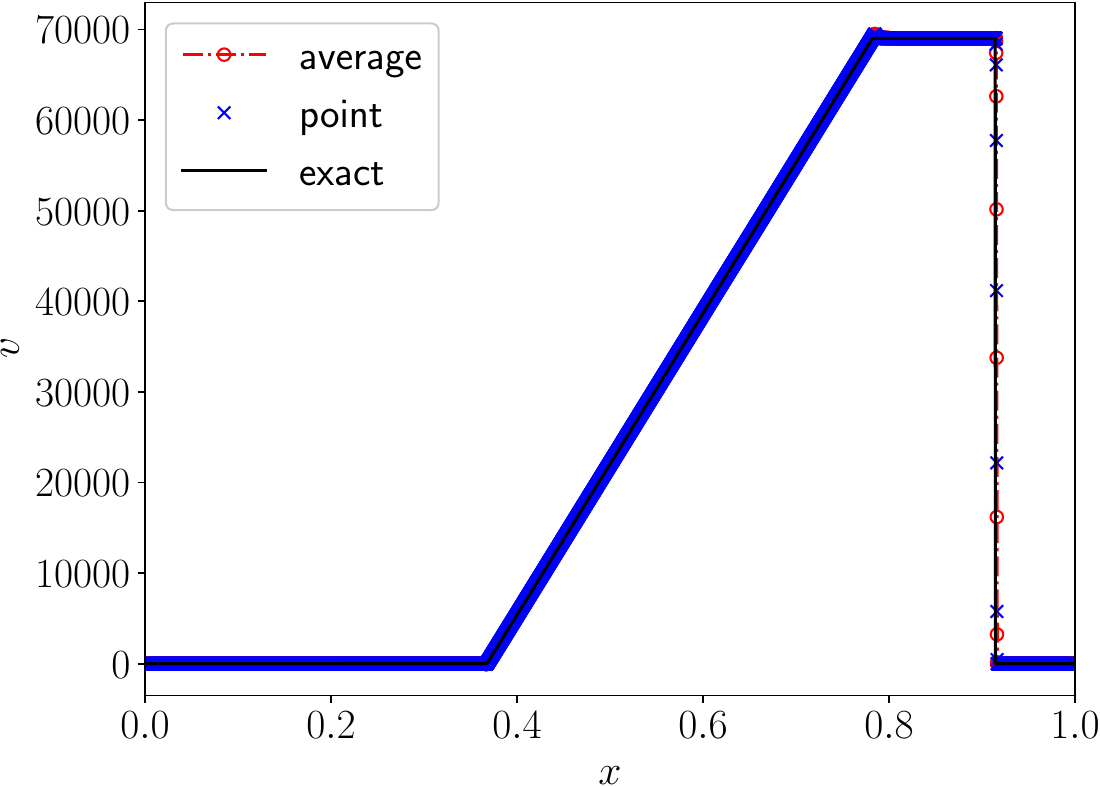}
	\end{subfigure}
	\begin{subfigure}[b]{0.24\textwidth}
		\centering
		\includegraphics[width=\linewidth]{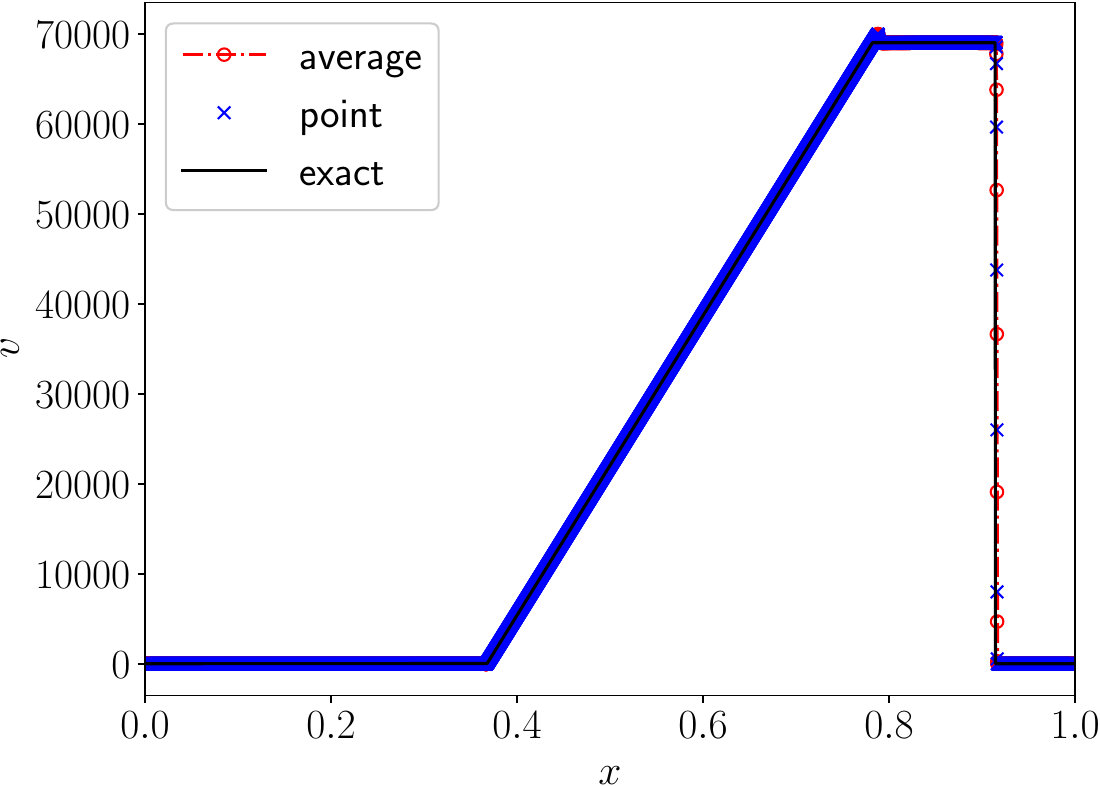}
	\end{subfigure}
	\begin{subfigure}[b]{0.24\textwidth}
		\centering
		\includegraphics[width=\linewidth]{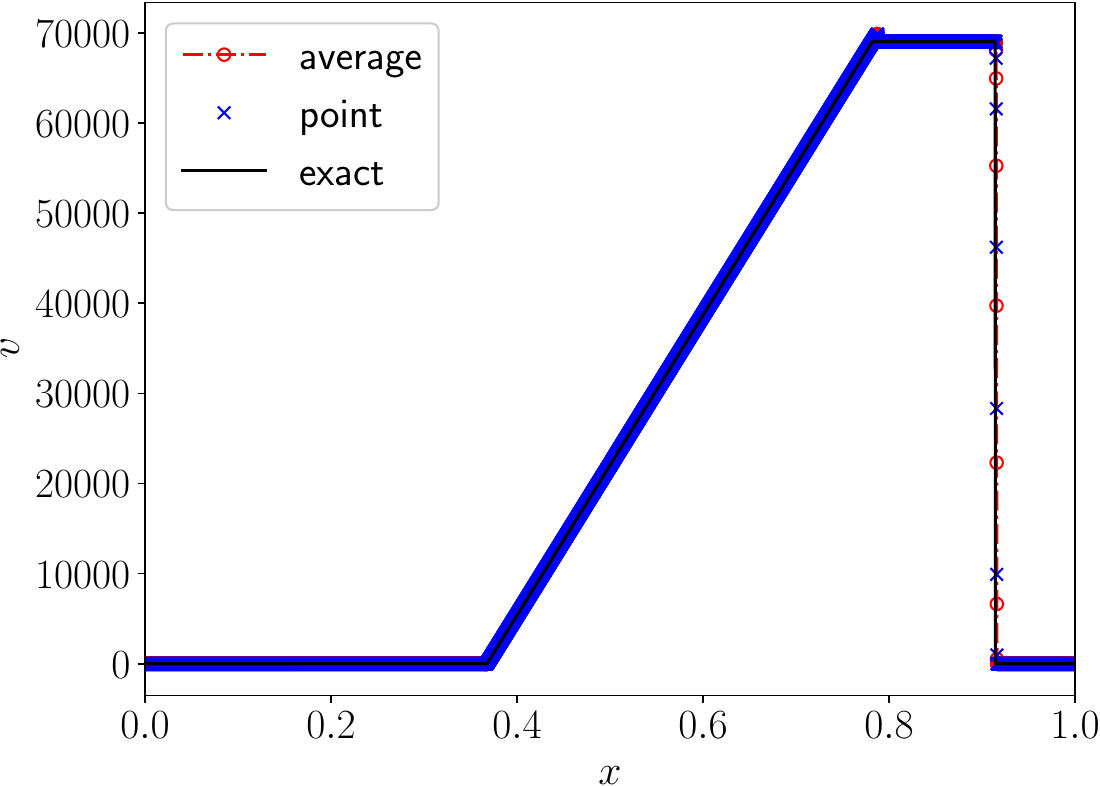}
	\end{subfigure}
	\begin{subfigure}[b]{0.24\textwidth}
		\centering
		\includegraphics[width=\linewidth]{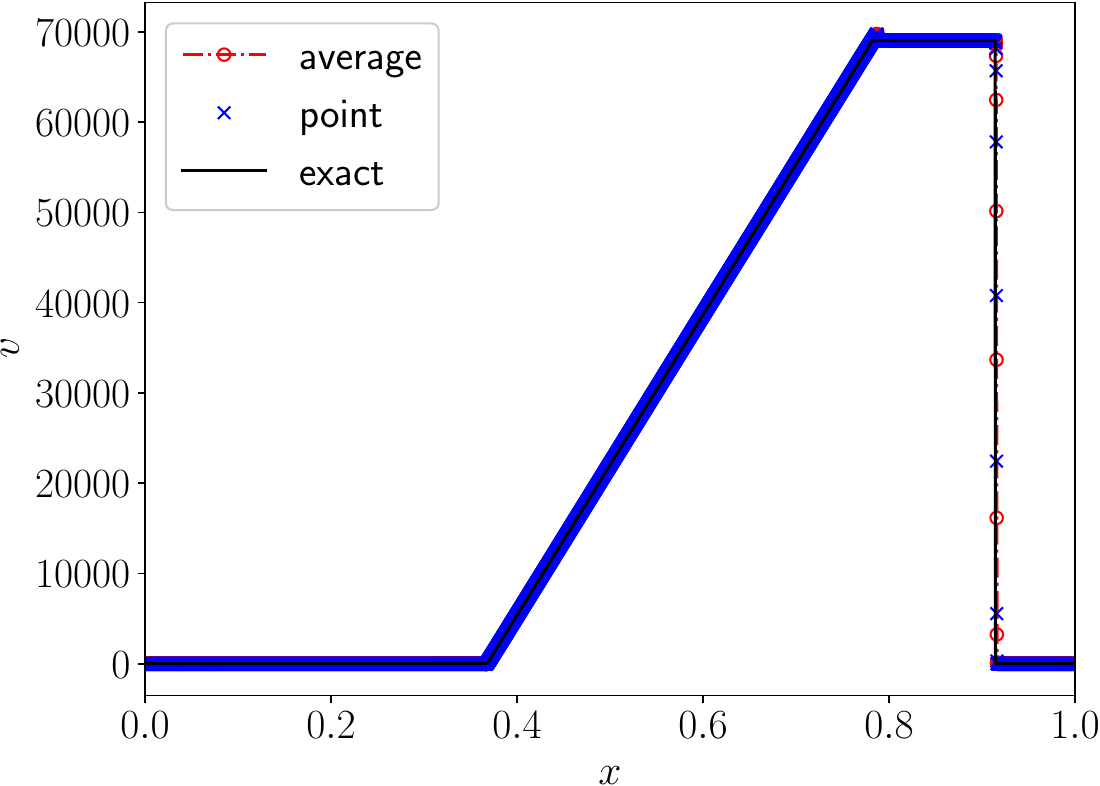}
	\end{subfigure}
	
	\begin{subfigure}[b]{0.24\textwidth}
		\centering
		\includegraphics[width=\linewidth]{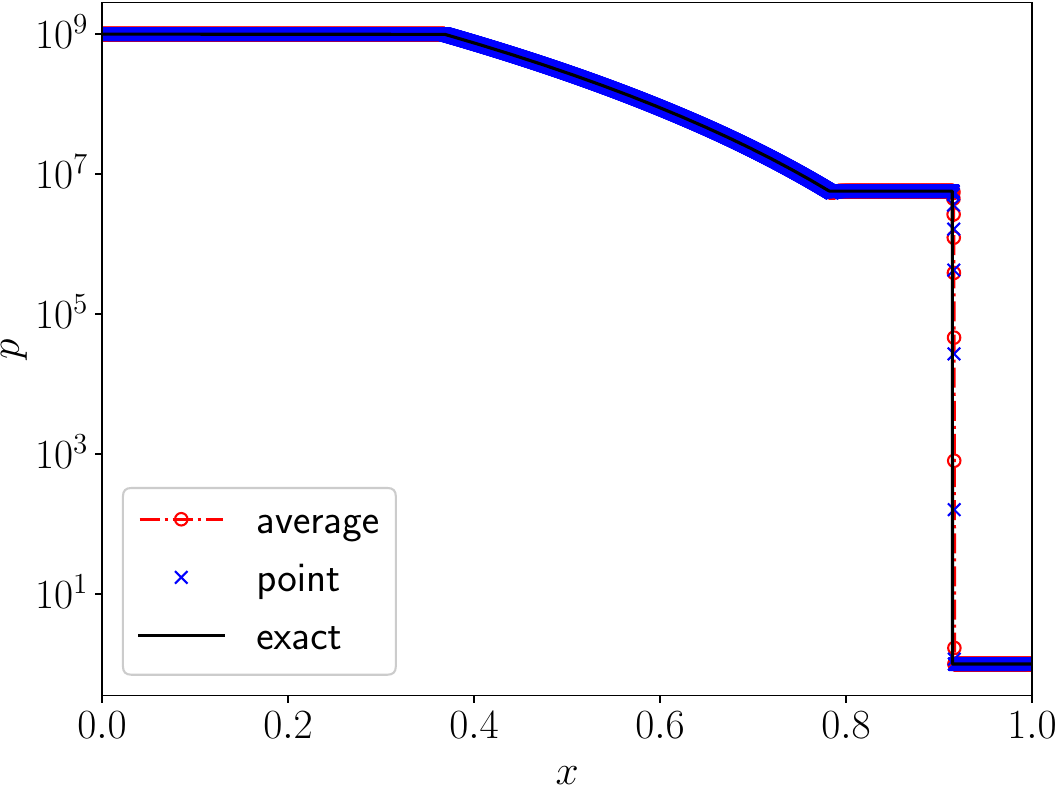}
	\end{subfigure}
	\begin{subfigure}[b]{0.24\textwidth}
		\centering
		\includegraphics[width=\linewidth]{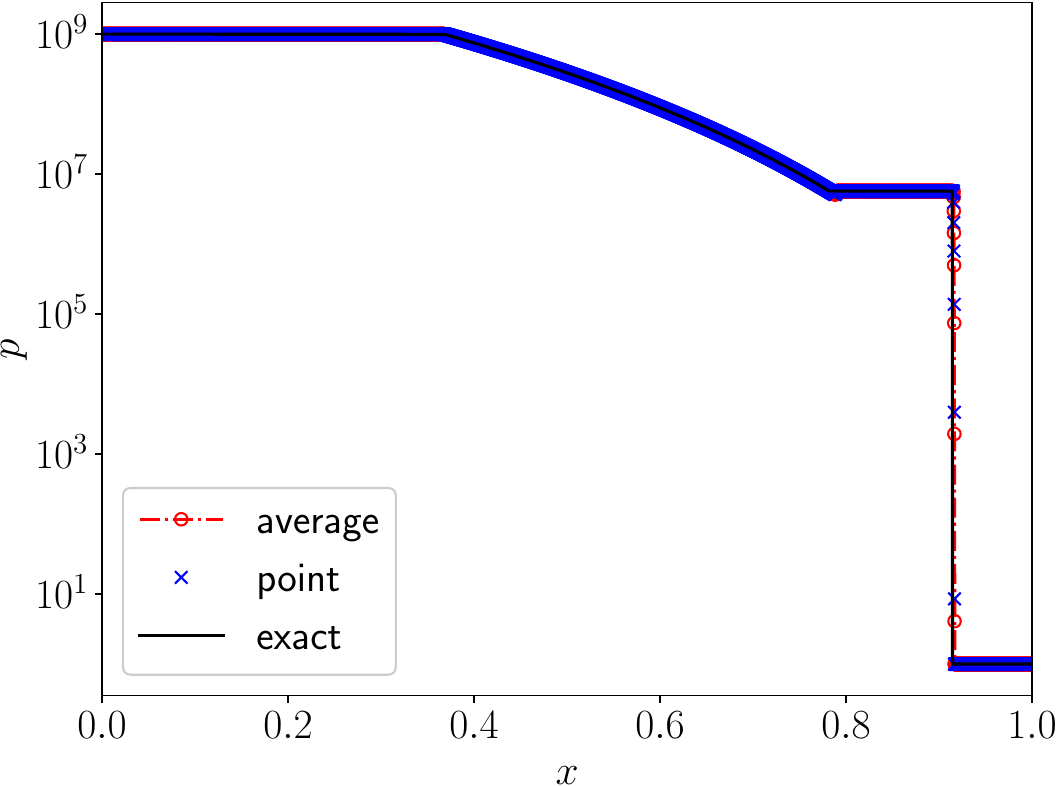}
	\end{subfigure}
	\begin{subfigure}[b]{0.24\textwidth}
		\centering
		\includegraphics[width=\linewidth]{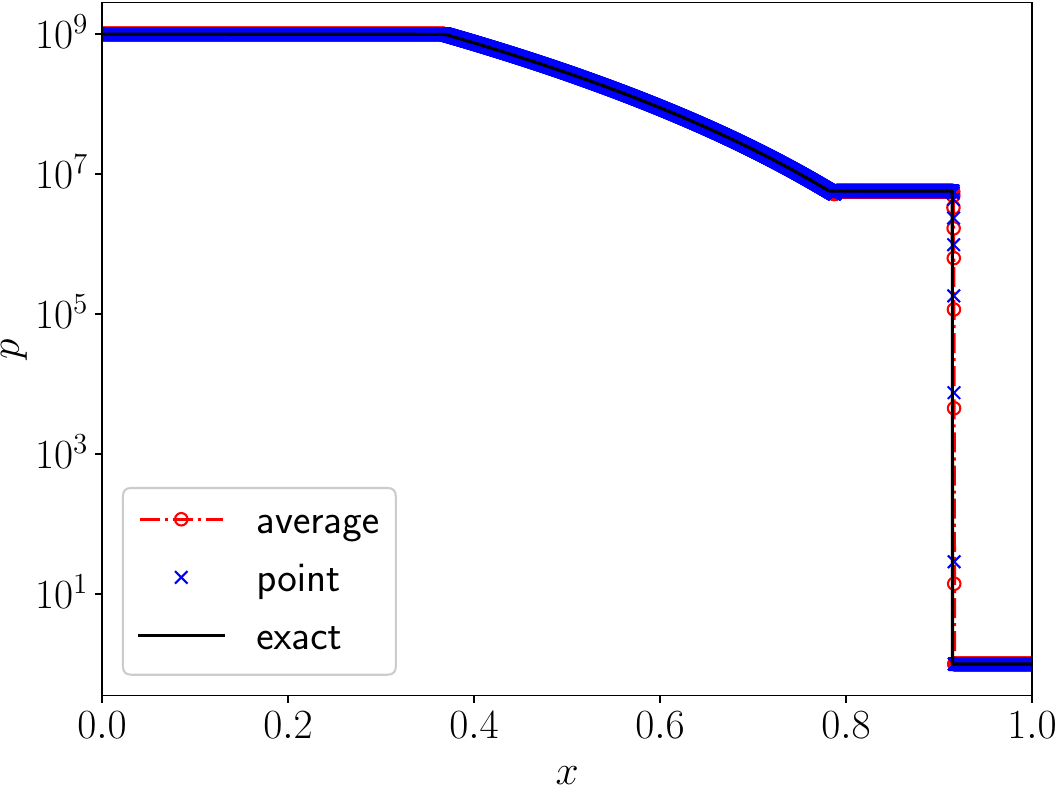}
	\end{subfigure}
	\begin{subfigure}[b]{0.24\textwidth}
		\centering
		\includegraphics[width=\linewidth]{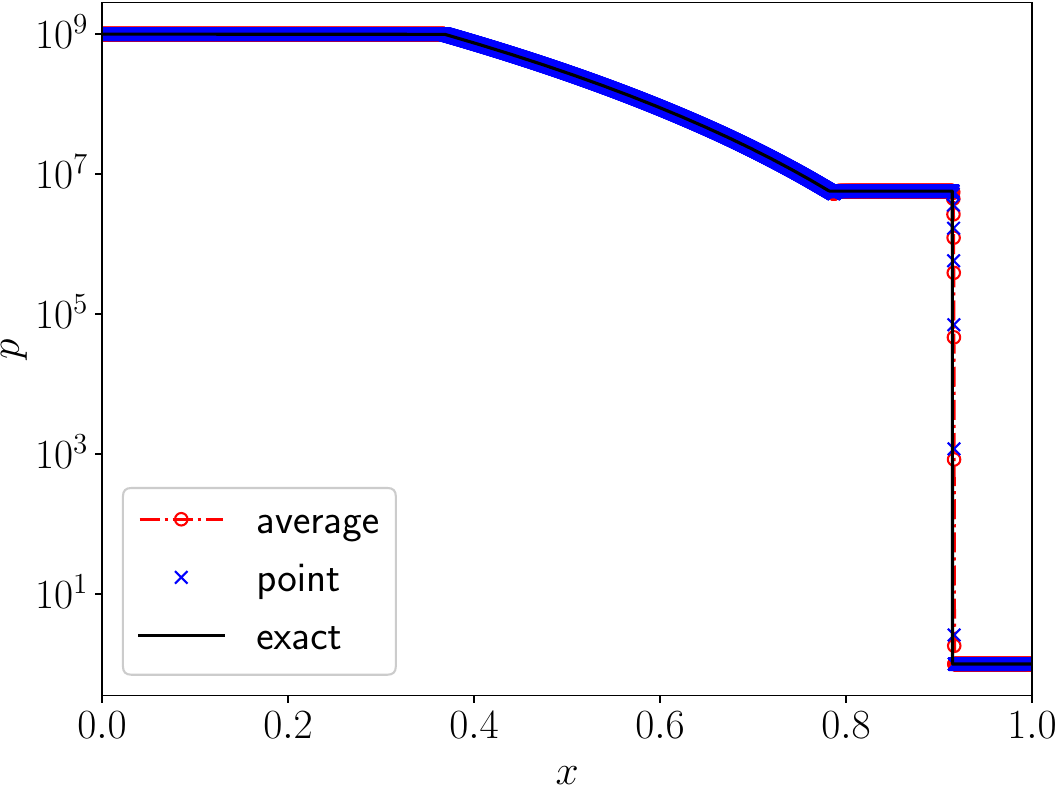}
	\end{subfigure}
	\caption{\Cref{ex:1d_leblanc_supp}, LeBlanc Riemann problem.
		The numerical solutions are computed with the BP limitings for the cell average and point value updates on a uniform mesh of $6000$ cells.
		The CFL number is $0.4$ and the power law reconstruction is not used.
		The shock sensor-based limiting with $\kappa=10$ is used.
		From left to right: JS, LLF, SW, and VH FVS.}
	\label{fig:1d_leblanc_fine_mesh_cfl0.4_v_p}
\end{figure}
\end{example}

\bibliographystyle{siamplain}
\bibliography{/Users/Junming/Research/references.bib}

\end{document}